\documentclass{article}
\usepackage[T1]{fontenc}
\usepackage[english]{babel}

\usepackage[letterpaper,top=2cm,bottom=2cm,left=3cm,right=3cm,marginparwidth=1.75cm]{geometry}
\RequirePackage{tikz} 
\usepackage{amsmath,amssymb,mathtools,amsthm}
\usepackage{graphicx}
\usepackage{cite}
\usepackage{xcolor}
\usepackage[inline]{enumitem}
\usepackage[colorlinks=true, allcolors=blue]{hyperref}
\usepackage{cleveref}
\usepackage{aligned-overset}
\usepackage{xparse}
\usepackage{caption}
\usepackage{tikz-cd}
\usetikzlibrary{positioning,arrows,decorations.markings}
\usetikzlibrary{external}
\tikzset{
block/.style={
  draw, 
  rectangle, 
  minimum height=1.5cm, 
  minimum width=2.5cm, align=center,
  fill=blue!20
  }, 
line/.style={->,>=latex'}
}
\tikzset{negated/.style={
      decoration={markings,
           mark= at position 0.5 with {
               \node[transform shape] (tempnode) {$\backslash\!\!\backslash$};
            }
       },
       postaction={decorate}
    }
}

\usepackage{csquotes}

\newcommand{\R}{\mathbb R}                              
\newcommand{\C}{\mathbb C}                              
\newcommand{\N}{\mathbb N}                              
\newcommand{\bmat}[1]{\begin{bmatrix}#1\end{bmatrix}}   
\newcommand{\ct}[1]{{#1}^{\mathrm{H}}}                  
\newcommand{\ict}[1]{{#1}^{-\mathrm{H}}}                
\newcommand{\cc}[1]{\overline{#1}}                      
\newcommand{\pd}[2]{\frac{\partial #1}{\partial #2}}    
\newcommand{\dd}[2]{\frac{\mathrm{d}#1}{\mathrm{d}#2}}  
\newcommand{\wt}{\widetilde}                            
\newcommand{\wh}{\widehat}                              
\newcommand{\timeInt}{\mathbb T}                        
\newcommand{\posR}{\mathbb R_+}                         
\newcommand{\td}{\,\mathrm{d}}                          
\newcommand{\heavi}{\boldsymbol{1}}                     
\newcommand{\avSt}{V_a}                                 
\newcommand{\stm}{\Phi}                                 
\newcommand{\LAMBDA}{\boldsymbol{\Lambda}}              
\newcommand{\loc}{\mathrm{loc}}                         
\newcommand{\imagUnit}{\mathrm{i}}                      
\newcommand{\wildcard}{\square}
\newcommand{\partInt}[1]{                               
    \mathfrak{P}\ifstrempty{#1}{}{_{#1}}}             
\newcommand{\partit}{\mathfrak p}                       
\newcommand{\quadSt}[1]{V_{#1}}
\newcommand{\Cantor}{\mathfrak c}                       
\newcommand{\CantorReverse}{\breve{\Cantor}}            

\NewDocumentCommand{\rqSuArg}{sO{t_0}O{t_{-1}}O{x_0}O{0}O{x}O{u}O{y}}{\IfBooleanTF{#1}{}{\substack{(#3,#7) \,:\, #3\leq#2, \\ #6(#2)=#4 ,\, #6(#3)=#5}}}
\NewDocumentCommand{\avStArg}{sO{t_0}O{t_1}O{x_0}O{x}O{u}O{y}}{\IfBooleanTF{#1}{#6\,:\,#5(#2)=#4}{\substack{(#3,#6)\,:\,#3\geq#2, \\ #5(#2)=#4}}}
\NewDocumentCommand{\totVar}{oo}{
    {\operatorname{Var}}\IfNoValueTF{#1}{}{_{#1,#2}}}                      
\NewDocumentCommand{\HerMat}{o}{                        
    \operatorname{\mathbf{S}}\IfValueTF{#1}{^{#1}}{}(\C)}
\NewDocumentCommand{\posSD}{o}{                         
    \operatorname{\mathbf{S}}\IfValueTF{#1}{^{#1}}{}_+(\C)}
\NewDocumentCommand{\posDef}{o}{                        
    \operatorname{\mathbf{S}}\IfValueTF{#1}{^{#1}}{}_{++}(\C)}
\NewDocumentCommand{\GL}{o}{                            
    \operatorname{GL}\IfValueTF{#1}{_{#1}}{}(\C)}

\newcommand{\refPa}{\hyperref[def:passive]{(Pa)}}       
\newcommand{\refNN}{\hyperref[def:nonnegativeSupply]{(NN)}}   
\newcommand{\refKYP}{\hyperref[def:KYP]{(KYP)}}         
\newcommand{\refPH}{\hyperref[def:pH]{(pH)}}            
\newcommand{\RS}[1]{\mathcal R_{#1}}                    

\DeclareMathOperator{\realPart}{Re}                     
\DeclareMathOperator{\linGroup}{GL}                     
\let\ker\relax\DeclareMathOperator{\ker}{Ker}           
\DeclareMathOperator{\BV}{BV}                           
\DeclareMathOperator{\AUC}{AUC}                         
\DeclareMathOperator{\rank}{rank}                       

\DeclarePairedDelimiter{\set}{\{}{\}}
\DeclarePairedDelimiter{\pset}{(}{)}
\DeclarePairedDelimiter{\bset}{[}{]}
\DeclarePairedDelimiter{\abs}{\lvert}{\rvert}
\DeclarePairedDelimiter{\aset}{\langle}{\rangle}
\DeclarePairedDelimiter{\norm}{\lVert}{\rVert}

\theoremstyle{definition}
\newtheorem{theorem}{Theorem}[section]
\newtheorem{lemma}[theorem]{Lemma}
\newtheorem{corollary}[theorem]{Corollary}
\newtheorem{definition}[theorem]{Definition}
\newtheorem{remark}[theorem]{Remark}
\newtheorem{example}[theorem]{Example}

\title{Relationship between dissipativity concepts for linear time-varying port-Hamiltonian systems}

\author{Karim Cherifi\footnotemark[1]
\, \and Hannes Gernandt\footnotemark[1]\, \and Dorothea Hinsen\footnotemark[2]\,  \and Volker Mehrmann\footnotemark[2]\, \and Riccardo Morandin\footnotemark[3]}

\date{2024}

\begin{document}

\renewcommand*{\thefootnote}{\fnsymbol{footnote}}
\maketitle
\begin{abstract}
The relationship between different dissipativity concepts for linear time-varying systems is studied, in particular between port-Hamiltonian systems, passive systems, and systems with nonnegative supply. It is shown that linear time-varying port-Hamiltonian systems are passive, have nonnegative supply rates, and solve (under different smoothness assumptions) Kalman-Yakubovich-Popov differential and integral inequalities. The converse relations are also studied in detail. In particular,  sufficient conditions are presented to obtain a port-Hamiltonian representation starting from any of the other dissipativity concepts. Two applications are presented.
\end{abstract}


\footnotetext[2]{
Institut f\"ur Mathematik, MA 4-5, TU Berlin, Str. des 17. Juni 136,
D-10623 Berlin, FRG.
\texttt{\{hinsen,mehrmann\}@math.tu-berlin.de}. The authors gratefully acknowledge the support by  Deutsche
Forschungsgemeinschaft (DFG) as part of the collaborative research center \textsf{SFB TRR} 154 (grant no.~239904186), excellence cluster \textsf{MATH\textsuperscript{+}, Project AA4-12}, and by \textsf{BMBF (grant no.~05M22KTB)} through \textsf{EKSSE}.}

\footnotetext[1]{
Institute for Mathematical Modeling, Analysis and Computational Mathematics, University of Wuppertal, Gau\ss stra\ss e 20, 42119 Wuppertal, FRG. \texttt{\{cherifi,gernandt\}@uni-wuppertal.de}.}

\footnotetext[3]{
Institute of Analysis und Numerics, Otto von Guericke University Magdeburg, Universit\"atsplatz 2, 39106 Magdeburg. \texttt{riccardo.morandin@ovgu.de}.
The author gratefully acknowledge the support by Deutsche Forschungsgemeinschaft (DFG) through the project 446856041.}

\renewcommand{\thefootnote}{\arabic{footnote}}
\setcounter{footnote}{0}

\section*{Notation}

For an interval $I\subseteq\R$ we denote the length by $|I|$.
In this paper $\timeInt$ denotes a given open (possibly unbounded) interval $\timeInt\subseteq\R$, unless otherwise specified.
Given a vector $v\in\C^n$ or a matrix $A\in\C^{m,n}$ with $m,n \in \N$, we denote by $\norm{v}_2$ and $\norm{A}_2$ their 2-norm and induced 2-norm, respectively.
The matrix $\ct{A} \in \C^{n,m}$ stands for the complex conjugate of $A$.
For $p\in[1,\infty]$ and any measurable set $\Omega\subseteq\R^d$ we denote by $L^p(\Omega,\C)$ the usual Lebesgue spaces and by $W^{k,p}(\Omega,\C)$ for $k\in\mathbb N$ the corresponding Sobolev spaces, see \cite{Bre10}. \label{glo:W_kp}
Differentiability in this paper is to be intended not in the complex, holomorphic sense, but in the real sense, identifying $\C$ with $\R^2$.
Furthermore, we define $L^p(\Omega,\C^n)$ \label{glo:Lp} as the space of vector-valued functions whose entries are in $L^p(\Omega,\C)$. We define analogously $L^p(\Omega,\C^{m,n})$, $W^{k,p}(\Omega,\C^n)$.
In particular, we equip $L^p(\Omega,\C^{m,n})$ with the norm
\begin{alignat*}{3}
    & \norm{\,\cdot\,}_{L^p} &&: L^p(\Omega,\C^{m,n}) \to \R, &\qquad A &\mapsto \norm[\big]{ \norm{A}_2 }_{L^p}, \\
    & \norm{A}_2 &&: \Omega \to \R, &\qquad \omega &\mapsto \norm{A(\omega)}_2.
\end{alignat*}
In what follows, we will more often consider the local variants of these function spaces, i.e., \label{glo:Lploc}
\[
L^2_\loc(\Omega,\C^n) = \set[\big]{ f : \Omega\to\C^n \mid f|_K \in L^2(\Omega,\C^n)\text{ for all compact subsets }K\subseteq\Omega },
\]
and analogous definitions for the other spaces.
We sometimes omit the domain and co-domain from the function space notation, when they are general or clear from the context.
Note that $L^p$ for $p\in[1,\infty]$ is a Banach space, and $L^2$ is a Hilbert space with respect to the inner product 
\[
\aset{f,g}_{L^2} = \int_{\Omega}\ct{g(\omega)}f(\omega)\,\td \omega, \qquad\text{for all }f,g\in L^2(\Omega,\C^n).
\] 
Moreover, $\mathcal C(I,\C^{n})$ denotes the set of continuous functions $f:I\rightarrow \C^n$ \label{glo:cnt}. Furthermore, we denote by $\GL[n]$ \label{glo:Gl} the set of all invertible matrices $M\in\mathbb{C}^{n, n}$.
Given a pointwise invertible matrix function $A:\timeInt\to\GL[n]$, we denote by 
$
A^{-1} : \timeInt \to \GL[n], \ t \mapsto A(t)^{-1}
$
the pointwise inverse matrix function.
If $A$ is defined and invertible only for a.e.~$t\in\timeInt$, e.g.~if $A\in L^p_\loc(\timeInt,\GL[n])$, then $A^{-1}$ is only defined a.e.~on $t\in\timeInt$.

We denote by $\HerMat[n]\subseteq\C^{n,n}$ \label{glo:hermat}\label{glo:pos_semidef}\label{glo:pos_def} the subspace of Hermitian matrices of size $n$, and by $\posSD[n],\posDef[n]\subseteq\HerMat[n]$ the subsets of positive semi-definite and positive definite matrices, respectively. If $A\in\posSD[n]$ (resp.~$A\in\posDef[n]$), we also write $A\geq 0$ (resp.~$A>0$).
We equip $\HerMat[n]$ with the Loewner partial order, i.e., we write $A\geq B$ for $A,B\in\HerMat[n]$ if $A-B\geq 0$.

Given $t_0,t_1\in\R,\ t_0\leq t_1$, we denote by $\BV([t_0,t_1],\C^{m,n})$ the matrix functions of bounded variation on $[t_0,t_1]$ and by $\AUC([t_0,t_1],\HerMat[n])$ the absolutely upper semicontinuous Hermitian matrix functions on $[t_0,t_1]$.
For every open interval $\timeInt\subseteq\R$, we denote by $\BV_\loc(\timeInt,\C^{m,n})$ and $\AUC_\loc(\timeInt,\HerMat[n])$ the corresponding local variants of the spaces. We define these spaces precisely in \Cref{sec:null}. 

\section*{Glossary}
\begin{center}
\setlength{\tabcolsep}{2pt}
\scalebox{0.85}{
\begin{tabular}{|l |l | l|} 
 \hline
\begin{tabular}{@{}c@{}} \bfseries Abbreviations /  \\ \bfseries Symbols\end{tabular} & \bf ~Full name & \bf  \begin{tabular}{@{}c@{}} Reference  \\ in the text\end{tabular}  \\ [0.5ex] 
 \hline
$\AUC([t_0,t_1],\HerMat[n])$ & ~absolutely upper semicontinuous Hermitian matrix functions on $[t_0,t_1]$& p.\ \pageref{def:AUC_matrix}\\ \hline 
$\AUC_{\loc}(\timeInt,\HerMat[n])$ & $\begin{matrix} \text{locally absolutely upper semicontinuous Hermitian matrix functions} \\ \text{ on an open interval $\timeInt\subseteq\R$} \hspace{6.5cm}\end{matrix}$& p.\ \pageref{def:AUC_matrix}\\ \hline
$\BV([t_0,t_1],\C^{m,n})$ & ~matrix functions of bounded variation on $[t_0,t_1]$& p.\ \pageref{def:BV} \\ \hline
$\BV_\loc(\timeInt,\C^{m,n})$ & ~matrix functions of locally bounded variation on an open interval $\timeInt\subseteq\R$ & p.\ \pageref{def:BVloc} \\ \hline
$ \mathcal C(I,\C^{n})$  & ~set of continuous functions $f:I\rightarrow \C^n$ & p.\ \pageref{glo:cnt} \\ \hline 
$\GL[n]$ & ~set of invertible matrices in $\mathbb{C}^{n,n}$ & p.\ \pageref{glo:Gl}  
\\ \hline
\refKYP{} & ~Kalman-Yakubovich-Popov inequality for LTV system & p.\ \pageref{eq:KYP}
\\
 \hline
 $L^p(\Omega,\C^n)$ & ~space of measureable and $p$-integrable functions $f:\Omega\rightarrow\C^n$ & p.\ \pageref{glo:Lp}  \\ 
 \hline
  $L_{\loc}^p(\Omega,\C^n)$ & ~space of measureable and locally $p$-integrable functions $f:\Omega\rightarrow\C^n$ & p.\ \pageref{glo:Lploc} \\ \hline 
  LTI & ~linear time-invariant & p.\ \pageref{glo:LTI} \\ \hline 
LTV & ~linear time-varying & p.\ \pageref{eq:tv_system} \\ \hline 
$\LAMBDA_{t_0,t_1}$ & ~Popov operator in $L^2\pset[\big]{[t_0,t_1],\C^m}$ & p.\ \pageref{def:popov_op} \\  \hline
\refNN& ~LTV system with nonnegative supply & p.\ \pageref{def:nonnegativeSupply}
 \\
 \hline
\refPa{}& ~passive LTV system & p.\ \pageref{def:passive} \\ \hline
\refPH{}& ~linear time-varying port-Hamiltonian system & p.\ \pageref{def:pH}   
\\ \hline
$\stm$ & ~state-transition matrix associated with $\dot x=A(t)x$ & p.\ \pageref{def:stateTransMatrix}  
\\ \hline
$\HerMat[n]$ & ~set of Hermitian matrices in $\mathbb{C}^{n,n}$ & p.\ \pageref{glo:hermat}    \\ \hline
$\posSD[n]$ & ~set of Hermitian positive semi-definite matrices in $\mathbb{C}^{n,n}$ & p.\ \pageref{glo:pos_semidef} \\ \hline $\posDef[n]$ & ~set of Hermitian positive definite matrices in $\mathbb{C}^{n,n}$ & p.\ \pageref{glo:pos_def}\\
 \hline
$\timeInt$& ~open (possibly unbounded) time interval & 
 \\ \hline
 $\totVar[t_0][t_1]$ & ~total variation on $[t_0,t_1]$ & \pageref{glo:transfer_op}
 \\ \hline
 $W^{k,p}(\Omega,\C^n)$ & ~Sobolev space of $k$-times weakly differentiable $f:\Omega\rightarrow\C^n$ with derivatives in $L^p$ & p.\ \pageref{glo:W_kp}\\ \hline
 $W_{\loc}^{k,p}(\Omega,\C^n)$ & ~Sobolev space of $k$-times weakly differentiable $f:\Omega\rightarrow\C^n$ with derivatives in $L_{\loc}^p$ & p.\ \pageref{glo:W_kp}
 \\ \hline
 $\mathbf{Z}_{t_0,t_1}$ & ~transfer operator & p.\ \pageref{glo:transfer_op}\\ \hline
\end{tabular}
}
\end{center}

\section{Introduction}\label{sec:intro}
This paper is devoted to the analysis of linear time-varying (LTV) systems and their dissipativity properties. LTV systems appear quite naturally in many applications, for example, when nonlinear systems are linearized around nonstationary reference solutions \cite{Cam95} or in the context of linear systems when some of the system parameters are time-dependent. Examples are the rocket problem \cite{ForD10}, where the movement of a mass is described that decreases with time, or district heating systems that contain water storages with volumes that vary over time \cite{MacFCS22}. In addition, LTV systems also appear in the modeling and stability analysis of power systems, e.g., in the context of Harmonic Power Flow \cite{CecBPZLP23} or when time-varying switching signals of power converters are considered \cite{GerSZMS24}.

The popularity of dissipative system models in engineering applications stems from the fact that they are important in stabilization and optimal control design. 
In the literature, there are many ways to formalize and incorporate the physically motivated notion of dissipativity into the system model. Among the most common approaches are \emph{port-Hamiltonian} system formulations and the use of \emph{storage functions} that lead to the classical notions of (passive or) dissipative systems \cite{Wil72}. There are different variations of these concepts, and
the aim of this paper is to discuss proper definitions for each of these properties for linear time-varying systems and to analyze the relations between them. 
This work is inspired by previous work of the authors in the linear time-invariant (LTI) \label{glo:LTI}case~\cite{CheGH23,CheGHM23}.

We consider linear time-varying (LTV) systems of the form 
\begin{equation}
\label{eq:tv_system}
\begin{split}
    \dot x(t) &= A(t)x(t)+B(t)u(t),\\
    y(t) &= C(t)x(t)+D(t)u(t),
\end{split}
\end{equation}
in an open (possibly unbounded) time interval $\timeInt\subseteq\mathbb R$, with state 
$x\in W^{1,1}_\loc(\timeInt,\C^n)$, input and output variables $u,y\in ~L^2_\loc(\timeInt,\C^m)$, and the coefficients $A\in L^1_\loc(\timeInt,\C^{n,n})$, $B\in L^2_\loc(\timeInt,\C^{n,m})$, $C\in L^2_\loc(\timeInt,\C^{m,n})$ and $D\in L^\infty_\loc(\timeInt,\C^{m,m})$ are matrix functions. 

Note that since the (time) domain is one-dimensional, it holds that $W^{1,1}_\loc(\timeInt,\C^n)\subseteq\mathcal C(\timeInt,\C^n)\subseteq L^\infty_\loc(\timeInt,\C^n)$. Furthermore, for every input $u\in L^2_\loc(\timeInt,\C^m)$, system \eqref{eq:tv_system} satisfies the Carath\'eodory conditions \cite{Fil88} for the existence and uniqueness of solutions. Thus, for every pair $(t_0,x_0)\in\timeInt\times\C^n$ there exists exactly one state solution $x\in W^{1,1}_\loc(\timeInt,\C^n)$ that satisfies $x(t_0)=x_0$.
Moreover, the choice of the function spaces for $C$ and $D$ ensures that $y\in L^2_\loc(\timeInt,\C^m)$, see \Cref{cor:solution_tv} in the appendix for more details. Note that throughout the paper, we always use the weakest possible choices of function spaces in order to be able to deal with special situations like jumps in the coefficients or inputs as well as possible switching behavior.

We consider a complex-valued setting here and, therefore, our results are formulated accordingly. However, most of them also hold in the real-valued case. 

\noindent Extending to complex values what is presented in \cite{BeaMXZ18,MehM19}, we make use of the following definition of linear time-varying port-Hamiltonian (pH) systems.
\begin{definition}\label{def:pH}
    A linear time-varying \emph{port-Hamiltonian}  system \refPH{}  is a dynamical system of the form \eqref{eq:tv_system}
    together with a (time-varying) quadratic Hamiltonian $\mathcal{H}(t,x)=\frac{1}{2}\ct{x}Q(t)x$, such that
    \begin{equation}\label{eq:pH_coefficients}
    \bmat{A(t) & B(t) \\ C(t) & D(t)} = \bmat{ \pset[\big]{J(t)-R(t)}Q(t)-K(t) & G(t)-P(t) \\ \pset[\big]{\ct{G(t)}+\ct{P(t)}}Q(t) & S(t)-N(t) },
    \end{equation}
    where $Q\in W^{1,1}_\loc(\timeInt,\posSD[n])$, $J,R,K:\timeInt\to\mathbb C^{n,n}$, $G,P:\timeInt\to\mathbb C^{n,m}$ and $S,N:\timeInt\to\mathbb C^{m,m}$ are such that $J(t)=-\ct{J(t)}$, $N(t)=-\ct{N(t)}$, and \begin{align}\label{eq:Q_lyapunov}Q(t)K(t)+\ct{K(t)}Q(t)=\dot Q(t),\end{align} as well as 
    \[
        W(t) \coloneqq \bmat{R(t) & P(t) \\ \ct{P(t)} & S(t)}
        \in \posSD[n+m]
        \]
    for a.e.~$t\in\timeInt$.
    We say that an LTV system \eqref{eq:tv_system} admits a port-Hamiltonian formulation if there exist $Q,K,J,R,G,P,S,N$
    satisfying the properties stated above.
\end{definition}

\noindent The concept of \emph{passivity} was introduced independently by Kalman, Yakubovich and Popov  for linear time-invariant systems based on the solution of a linear matrix inequality, which is called the \emph{Kalman-Yakubovich-Popov (KYP) inequality}. For linear time-varying systems, this concept was generalized in \cite{AndM74} as follows.
\begin{definition} \label{def:KYP}
    Consider a system of the form \eqref{eq:tv_system}. We say that a matrix function $Q\in W^{1,1}_\loc(\timeInt,\posSD[n])$ pointwise satisfies an associated \emph{Kalman-Yakubovich-Popov inequality} \refKYP{} if
    \begin{equation}\label{eq:KYP}
        \bmat{ (-\ct AQ - QA - \dot Q)(t) & (\ct{C}-QB)(t) \\ (C-\ct B Q)(t) & (D+\ct D)(t) }\geq 0
    \end{equation}
    holds for a.e.~$t\in\timeInt$.
\end{definition}

\noindent The alternative definition of passivity of~\cite{Wil72} for LTI systems in terms of storage functions was generalized in \cite{HilM80} to time-varying systems.
\begin{definition}\label{def:passive}
    A system of the form \eqref{eq:tv_system} is called \emph{passive} \refPa{} if there exists a pointwise nonnegative real-valued function $V:\timeInt\times\C^n\to[0,\infty)$ that satisfies $V(t_0,0)=0$ as well as the \emph{dissipation inequality}
    \begin{equation}
    \label{def:passive_ineq}
        V\pset[\big]{t_1,x(t_1)} - V\pset[\big]{t_0,x(t_0)} \leq \int_{t_0}^{t_1}\realPart(\ct{y(t)}u(t))\td t
    \end{equation}
    for all $t_0,t_1\in\timeInt$ with $t_0\leq t_1$ and for all state-input-output solutions $(x,u,y)$ of \eqref{eq:tv_system}. 
    We call a function that satisfies these properties a \emph{storage function} for \eqref{eq:tv_system}.
\end{definition}

\noindent The property introduced in \Cref{def:passive} is sometimes just called passivity \cite{CheGH23} or \emph{impedance passivity}~\cite{CheGHM23,KurS07,Sta03}. In some references, the right-hand side in \eqref{def:passive_ineq} is replaced by a more general term $\int_{t_0}^{t_1}w(u(t),y(t))\td t$, where $w:\C^m\times\C^m\to\R$ denotes a more general \emph{supply rate} \cite{Wil72} which is usually quadratic. 
Moreover, in \cite{HilM80} a different but related \emph{dissipativity} concept is introduced, requiring the inequality 
\begin{align*}
    \int_{t_0}^{t_1}w(u(t),y(t))\td t\geq 0 
\end{align*}
to hold for every solution such that $x(t_0)=0$.
Passivity is then defined as the corresponding property for $w(u,y)=\realPart(\ct{y}u)$, i.e.,
\begin{equation*}
        \inf_{\substack{(t_0,t_1,u)\,:\,t_0\leq t_1, \\ x(t_0)=0}}\int_{t_0}^{t_1}\realPart\pset[\big]{\ct{y(t)}u(t)}\td t \geq 0.
\end{equation*}
We will show below that this property is weaker than that in~\Cref{def:passive}, but related to the nonnegativity of the \emph{Popov function}. In \Cref{def:nonnegativePopov} we will introduce a time-varying analogue of the Popov function, as in \cite{LozBEM00}, that
replaces the \emph{positive realness of the transfer function} (see, e.g.~\cite{CheGH23}). At this point, we do not discuss the Popov-operator, we will do this
in detail in Section~\ref{sec:popov},
but instead introduce \emph{nonnegativity of the supply rate}.
\begin{definition}\label{def:nonnegativeSupply}
    We say that the system \eqref{eq:tv_system} has \emph{nonnegative supply} \refNN{} if
    \begin{equation*}
        \int_{t_0}^{t_1}\realPart\pset[\big]{\ct{y(t)}u(t)}\td t \geq 0
    \end{equation*}
    holds for every $t_0,t_1\in\timeInt,\ t_0\leq t_1$ and every state-input-output solution $(x,u,y)$ such that $x(t_0)=0$.
\end{definition}
\noindent Note that
\cite{HilM80} characterizes also \emph{cyclo-dissipativity}, see \cite{Sch17}, where the nonnegativity is required to hold only for solutions satisfying $x(t_0)=x(t_1)=0$.

The nonnegativity condition for $Q$, or correspondingly the property for the Hamiltonian $\mathcal H$ to be bounded from below, is not always included in the definition of port-Hamiltonian systems, see e.g. \cite{SchJ14,SchJ21}. 
Since our goal is to compare the \refPH{} condition with \refKYP{}, \refPa{} and \refNN{}, it is natural to include $Q\geq 0$ in our definition.
If one allows $Q$ to be indefinite, then the pH systems are only \emph{cyclo-passive}, i.e., they satisfy \eqref{def:passive_ineq} when considering only \emph{cyclic trajectories}, i.e., trajectories with $x(t_1)=t_0$, see, e.g.~\cite{HilM80,SchJ21}. This is then
related to the existence of indefinite solutions of the KYP inequality and to the nonnegativity of the Popov operator when restricted to the inputs that induce cyclic trajectories.

For LTI systems, the mapping between input and output is usually characterized by a transfer operator
resulting from an impulsive input in the frequency domain. The extension of the transfer operator for time-varying systems is introduced in different ways in \cite{AndM74,BayE05,LozBEM00,Vid02,Zad50}. The challenge for linear time-varying systems is that the transfer operator is in general not independent of the time point where the
input impulse is applied.
In \cite{KamK82}, a definition of a transfer operator was considered based on a time-varying impulse. There, it is shown that the time-varying transfer operator can be considered as arising from a delayed linear time-invariant system with a delay corresponding to the time point where the impulse is applied. 
Another variant is used in \cite{DewV98} to convert the LTV system to an LTI system with an infinite number of entries.
Alternatively, in~\cite{BayE05}, a two-variable transfer function is considered, where the second variable depends on the time point at which the impulse is applied.
We present a new concept for the transfer operator in \eqref{def:transfer_op}.

In this paper, we provide a detailed discussion of the subtle relationships between the four properties \refPH{}, \refPa{}, \refKYP{}, and \refNN{}. Up to now, there exist only partial results on these relationships, and different notation is also used for each property. In \cite{AndM74} the relationship between \refKYP{} and \refNN{} is studied and the invertibility of solutions of the KYP inequality under adequate observability assumptions is presented. Such systems are called passive.
In \cite{HilM80} for real systems, more general quadratic supply rates in an input-output setting are studied. The \refNN{} property is called dissipativity in that context. In the same paper, storage functions and available storage are introduced and the dissipativity and cyclo-dissipativity are characterized by the finiteness of available storage using a reachability assumption.  In \cite{ForD10} the relationship between \refNN{} and the existence of solutions to Lur'e equations is studied under additional observability and controllability assumptions. The property \refNN{} is also called passive and strict passivity notions are considered. 
In \cite{MehM19} a definition of nonlinear time-varying pH systems is introduced, and the passivity of such systems is shown. 
Our paper is partly motivated by the detailed overview in \cite{LozBEM00}  where the notions \refNN{}, \refKYP{}, and \refPa{} are presented and related to those of the LTI case. Here we include \refPH{} and generally weaken the assumptions.

The content of the paper is as follows. In Section~\ref{sec:prelim} 
we present some preliminary results that are necessary for subsequent sections. 
In Section~\ref{sec:phsystem} we give a detailed discussion of the (dissipativity) properties of linear time-varying port-{H}amiltonian systems.
Section~\ref{sec:KYPinequality} analyzes the KYP inequality and its implications for port-Hamiltonian systems, and Section~\ref{sec:PatopH} analyzes the relationship between passivity and port-Hamiltonian systems. In Section~\ref{sec:popov} we study systems with nonnegative supply and the relationship to port-Hamiltonian systems including intermediate relations to KYP inequality and passivity. The paper ends with some applications in Section~\ref{sec:applications}  and conclusions in Section~\ref{sec:conclusions}. Several important results that are essentially known but often hard to track down are presented in the appendix. 

An overview of our results is presented in Figure~\ref{fig:overvieweinvertible}.

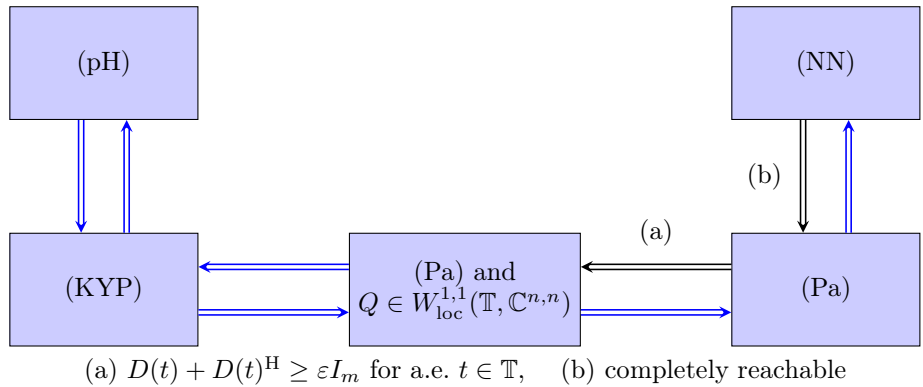
\begin{figure}[htbp!]
    \centering
        \begin{tikzpicture}
        \node[block] (a) {(pH)};
        \node[block, below =1.5cm of a]   (b){(KYP)};
        \node[block, right =2cm of b]   (bc){
        (Pa) and \\$Q\in W^{1,1}_{\mathrm{\loc}}(\mathbb T,\mathbb C^{n,n})$};
        \node[block, right =2cm of bc]   (c){(Pa)};
        \node[block, above =1.5cm of c]   (d){(NN)};
        
        \node[below = 0cm of bc] (e){
         (a) $D(t)+D(t)^{\mathrm{H}}\geq\varepsilon I_m$ for a.e.~$t\in\mathbb T$, \quad (b) completely reachable
          };
        \draw[->,semithick,double,double equal sign distance,>=stealth, color=blue] ([xshift=-2ex]a.south) -- ([xshift=-2ex]b.north);
        \draw[->,semithick,double,double equal sign distance,>=stealth, color=blue] ([xshift=2ex]b.north) -- ([xshift=2ex]a.south) node[midway,right = 1 ex]{};
         \draw[->,semithick,double,double equal sign distance,>=stealth, color=blue] ([yshift=-2ex]b.east) -- ([yshift=-2ex]bc.west) node[midway,below = 1 ex]{};
        \draw[->,semithick,double,double equal sign distance,>=stealth, color=blue] ([yshift=2ex]bc.west) -- ([yshift=2ex]b.east) node[midway,above = 1 ex]{};
        \draw[->,semithick,double,double equal sign distance,>=stealth, color=blue] ([yshift=-2ex]bc.east) -- ([yshift=-2ex]c.west) node[midway,below = 1 ex]{};
        \draw[->,semithick,double,double equal sign distance,>=stealth, color=black] ([yshift=2ex]c.west) -- ([yshift=2ex]bc.east) node[midway,above = 1 ex]{(a)};
        \draw[->,semithick,double,double equal sign distance,>=stealth, color=blue]([xshift=2ex]c.north) -- ([xshift=2ex]d.south);
        \draw[->,semithick,double,double equal sign distance,>=stealth,color=black] ([xshift=-2ex]d.south) -- ([xshift= -2ex]c.north) node[midway,left = 1 ex]{(b)};
          \end{tikzpicture}
    \caption{Relationship between \refPH{}, \refKYP{}, \refPa{} and \refNN{} for linear time-varying system \eqref{eq:tv_system}. Blue arrows indicate implications without additional assumptions, and black ones indicate 
    implications with additional assumptions. 
    For the LTI case complete reachability is equivalent to complete controllability resulting in the relationships from \cite{CheGH23}. A more detailed overview of the results including references can be found in \Cref{fig:overvieweinvertibleThm}. }
    \label{fig:overvieweinvertible}
\end{figure}

\section{Preliminaries}\label{sec:prelim}

In this section, we recall the solution theory for linear time-varying systems of the form \eqref{eq:tv_system}, which is used to define the Popov operator and to study the notions of controllability and observability. The definitions for these notions are not uniform in the literature. In this section, to create a complete picture of the different concepts, we present the definitions according to our notation. The proofs are given in the appendix.

\begin{definition}\label{def:stateTransMatrix}
    The \emph{state-transition matrix} of \eqref{eq:tv_system} is the unique solution $\stm\in W^{1,1}_\loc~(\timeInt~\times~\timeInt,\C^{n,n})$ of the initial value problem
    \begin{equation}\label{eq:stateTransMatrix}
        \pd{}{t}\stm(t,s) = A(t)\stm(t,s), \qquad \stm(s,s)=I_n,
    \end{equation}
    for $t,s\in\timeInt$.
\end{definition}

\noindent Details on the existence, uniqueness, and properties of the state-transition matrix are discussed in the appendix; see \Cref{thm:stateTransitionMatrix}.
In particular, we also have $\stm\in\mathcal C(\timeInt\times\timeInt,\GL[n])$, i.e., $\stm$ has a pointwise invertible continuous representative.
In the same result, a clear correspondence between the state transition matrix $\stm$ and the solutions of the homogeneous differential equation 
\begin{equation}\label{eq:tv_system_hom}
    \dot x(t) = A(t)x(t)
\end{equation}
is shown. 
In particular, for every $(t_0,x_0)\in\timeInt\times\C^n$ the unique solution of \eqref{eq:tv_system_hom} with $x(t_0)=x_0$ can be expressed as $x(t)=\stm(t,t_0)x_0$. 
Furthermore, the general solution of the inhomogeneous system \eqref{eq:tv_system} can be written as
\begin{equation}\label{eq:state_via_stm}
    x(t) = \stm(t,t_0)x_0 + \int_{t_0}^t\stm(t,s)B(s)u(s)\td s
\end{equation}
for all $t\in\timeInt$.

\begin{remark}\label{rem:fundamental}
In addition to using the state transition matrix $\stm$, the solutions of the initial value problem \eqref{eq:stateTransMatrix} can also be expressed as $\stm(s,t)=X(s)X(t)^{-1}$ where $X\in W^{1,1}_\loc(\timeInt,\C^{n,n})$ is the unique \emph{fundamental solution matrix} of $\dot X(t)=A(t)X(t)$ with initial condition $X(t_0)=I_n$ for some $t_0\in\timeInt$, see also \Cref{thm:fundamentalSolution} in the appendix. An advantage of using $\stm$ instead of $X$ is that no reference time $t_0$ has to be fixed.
\end{remark}

\noindent In the following, we recall some definitions of controllability and observability concepts for linear time-varying systems, see e.g.\ \cite{Ilc89,JikH14,KnoK13}.

\begin{definition}
\label{def:reachable}
    A state $x_0\in\C^n$ of the system \eqref{eq:tv_system} is called \emph{reachable} at time $t_0\in\timeInt$ if there exists a state-input-output solution $(x,u,y)$ of \eqref{eq:tv_system} such that $x(t_0)=x_0$ and $x(t_{-1})=0$ for some $t_{-1}\in\timeInt,\ t_{-1}\leq t_0$.
    The system \eqref{eq:tv_system} is called \emph{reachable} at time $t_0\in\timeInt$ if all states in $\C^n$ are reachable at time $t_0$.
    The system \eqref{eq:tv_system} is called \emph{completely reachable} if it is reachable at all times in $\timeInt$.
\end{definition}

\noindent Further, we use an observability concept that appears in \cite{AndM74} and is called \emph{reconstructability} in \cite{KnoK13}, see also \cite[Definition 3.3]{Ilc89}.
\begin{definition}
\label{def:reconstruct}
    The system \eqref{eq:tv_system} is called \emph{completely reconstructable} if for all $t_0\in\timeInt$ there exists $t_{-1}<t_0$ such that for all $x_0\in\R^n$ 
    \[
    C(t)\Phi(t,t_0)x_0=0
    \]
   for almost all $t\in\timeInt$ with $t\geq t_0$ implies $x_0=0$.
\end{definition}
\noindent The complete reconstructability is equivalent to the complete controllability of the dual system $\dot x(t)=A(-t)^\top x(t)+C(-t)^\top u(t)$, see \cite[Proposition 3.5]{Ilc89}.

\begin{remark}
It should be noted that, in contrast to LTI systems, there is a much greater variety of controllability and observability notions for linear time-varying systems. In particular, it is well known for LTI systems that reachability is equivalent to controllability and, therefore, to the reconstructability of the dual system. Consequently, it is known that for LTV systems with piecewise analytic coefficients, the complete reachability of Definition~\ref{def:reachable} is equivalent to complete controllability, see, e.g., \cite[Corollary 3.9]{Ilc89} and also \cite{JikH14} for LTV systems with continuous coefficients and piecewise continuous input. 
Moreover, it is known in the LTI case that the controllability (resp. reachability, reconstructability) at time $t_0\in\timeInt$ is equivalent to complete controllability (resp. reachability, reconstructability).
\end{remark}

\subsection{Transformation invariance of pH systems, KYP inequalities, passivity and nonnegative supplies}
In this subsection, we consider the invariance of the discussed dissipativity properties under state-space transformations. The following classical lemma, see, e.g. \cite{Ilc89}, describes the behavior of the coefficients of a LTV system under time-varying state space transformations.

\begin{lemma}\label{lem:changeOfVariables}
    Consider a LTV system of the form \eqref{eq:tv_system} and let $Z\in W^{1,1}_\loc(\timeInt, \GL[n])$. Then the change of variables $x(t)\coloneqq Z(t)\wt x(t)$ induces the equivalent LTV system
    \begin{equation}\label{eq:tv_system_changeOfVariables}
        \begin{split}
            \dot{\wt x}(t) &= \wt A(t)\wt x(t) + \wt B(t)u(t), \\
            y(t) &= \wt C(t)\wt x(t) + D(t)u(t),
        \end{split}
    \end{equation}
    where $\wt A=Z^{-1}(AZ-\dot Z)$, $\wt B=Z^{-1}B$ and $\wt C=CZ$.
    In particular, $(\wt x,u,y)$ is a solution of \eqref{eq:tv_system_changeOfVariables} if and only if $(x,u,y)$ is a solution of \eqref{eq:tv_system}.
\end{lemma}
\begin{proof}
    We first note that $Z^{-1}\in W^{1,1}_\loc(\timeInt,\GL[n])$, as shown in \Cref{lem:inverseContinuous}.
    Furthermore, note that $\wt A\in L^1_\loc$, $\wt B\in L^2_\loc$ and $\wt C\in L^2_\loc$ hold, since $Z,Z^{-1}\in L^\infty_\loc$ and $\dot Z\in L^1_\loc$.
    Suppose first that $(\wt x,u,y)$ is a solution of \eqref{eq:tv_system_changeOfVariables}. Then 
    \begin{align*}
        \dot x &= Z\dot{\wt x} + \dot Z\wt x = Z(\wt A\wt x+\wt Bu) + \dot ZZ^{-1}x = Ax + Bu, \\
        y &= \wt C\wt x + Du = CZZ^{-1}x + Du = Cx + Du,
    \end{align*}
    i.e., $(x,u,y)$ is a solution of \eqref{eq:tv_system}.
    Conversely, if $(x,u,y)$ is a solution of \eqref{eq:tv_system_changeOfVariables}, then
    \begin{align*}
        \dot{\wt x} &= Z^{-1}\dot x - Z^{-1}\dot ZZ^{-1}x = Z^{-1}(Ax+Bu) - Z^{-1}\dot Z\wt x = \wt A\wt x + \wt Bu, \\
        y &= Cx + Du = CZ\wt x + Du = \wt C\wt x + Du,
    \end{align*}
    i.e., $(\wt x,u,y)$ is a solution of \eqref{eq:tv_system_changeOfVariables}.
\end{proof}

\noindent Another relevant transformation is applying a change of input and output variables to the system.

\begin{lemma}\label{lem:LTV_IO_ChangeOfVariables}
    Consider an LTV system of the form \eqref{eq:tv_system} and let $V\in L^\infty_\loc(\timeInt,\GL[m])$ be such that also $V^{-1}\in L^\infty_\loc(\timeInt,\GL[m])$. Then the change of input and output variables $\check{u}\coloneqq V^{-1}u$, $\check{y}\coloneqq\ct{V}y$ induces the equivalent LTV system
    \begin{equation}\label{eq:LTV_IO_ChangeOfVariables}
        \begin{split}
            \dot x(t) &= A(t)x(t) + \check{B}(t)\check{u}(t), \\
            \check{y}(t) &= \check{C}(t)x(t) + \check{D}(t)\check{u}(t),
        \end{split}
    \end{equation}
    where $\check{B}=BV$, $\check{C}=\ct{V}C$, $\check{D}=\ct{V}DV$.
    In particular, $(x,\check{u},\check{y})$ is a solution of \eqref{eq:LTV_IO_ChangeOfVariables} if and only if $(x,u,y)$ is a solution of \eqref{eq:tv_system}.
\end{lemma}

\begin{proof}
    We note that, because of the generalized H\"older inequality (\Cref{thm:genHolder}), it holds that $u,y\in L^2_\loc$ if and only if $V^{-1}u,\ct{V}y\in L^2_\loc$.
    Analogously, we have $\check{B},\check{C}\in L^2_\loc$ and $\check{D}\in L^\infty_\loc$ and thus the transformed LTV system \eqref{eq:LTV_IO_ChangeOfVariables} is consistent with the function space assumptions for \eqref{eq:tv_system}.
    Suppose first that $(x,\check{u},\check{y})$ is a solution of \eqref{eq:LTV_IO_ChangeOfVariables}. Then
    \begin{align*}
        \dot x &= Ax + \check{B}\check{u} = Ax + BVV^{-1}u = Ax + Bu, \\
        y &= \ict{V}\check{y} = \ict{V}(\check{C}x+\check{D}u) = \ict{V}(\ct{V}Cx+\ct{V}DCV^{-1}u) = Cx + Du,
    \end{align*}
    i.e., $(x,u,y)$ is a solution of \eqref{eq:tv_system}.
    Conversely, if $(x,u,y)$ is a solution of \eqref{eq:tv_system}, then
    \begin{align*}
        \dot x &= Ax + Bu = Ax + BV\check{u} = Ax + \check{B}\check{u}, \\
        \check{y} &= \ct{V}y = \ct{V}(Cx + Du) = \ct{V}Cx + \ct{V}DV\check{u} = \check{C}x + \check{D}\check{u},
    \end{align*}
    i.e., $(x,\check u,\check y)$ is a solution of \eqref{eq:LTV_IO_ChangeOfVariables}.
\end{proof}

\noindent Note that, due to \Cref{lem:inverseContinuous}, the condition $V\in\mathcal C(\timeInt,\GL[n])$ is sufficient to satisfy the hypothesis of \Cref{lem:LTV_IO_ChangeOfVariables}.

It is also relevant to describe the behavior of the coefficients of a LTV system under time-varying transformations.

\begin{lemma}\label{lem:timeTransformation}
    Consider a LTV system of the form \eqref{eq:tv_system}, let $\wh\timeInt\subseteq\R$ be another open time interval and let $\theta\in\mathcal C^1(\wh\timeInt,\timeInt)$ be a diffeomorphism such that $\dot\theta>0$ 
    holds pointwise.
    Then the time-varying transformation $t=\theta(\wh t)$ induces the equivalent linear time-varying system
    \begin{equation}\label{eq:tv_system_timeTrans}
        \begin{split}
            \dot{\wh x}(\wh t) &= \wh A(\wh t)\wh x(\wh t) + \wh B(\wh t)\wh u(\wh t), \\
            \wh y(\wh t) &= \wh C(\wh t)\wh x(\wh t) + \wh D(\wh t)\wh u(\wh t),
        \end{split}
    \end{equation}
    where $\wh A=\dot\theta(A\circ\theta)$, $\wh B=\dot\theta(B\circ\theta)$, $\wh C=\dot\theta(C\circ\theta)$ and $\wh D=\dot\theta(D\circ\theta)$.
    In particular, $(x,u,y)$ is a state-input-output solution of \eqref{eq:tv_system} if and only if $(\wh x,\wh u,\wh y)\coloneqq(x\circ\theta,u\circ\theta,\dot\theta(y\circ\theta))$ is a state-input-output solution of \eqref{eq:tv_system_timeTrans}.
    Furthermore, for every $\wh t_0,\wh t_1\in\wh\timeInt,\ \wh t_0\leq\wh t_1$ and the corresponding $t_0\coloneqq\theta(\wh t_0)$ and $t_1\coloneqq\theta(\wh t_1)$ it holds that $t_0\leq t_1$ and
    \begin{equation}\label{eq:tv_system_timeTrans_supply}
        \int_{\wh t_0}^{\wh t_1}\realPart\pset[\big]{ \ct{\wh y(\wh t)}\wh u(\wh t) }\td\wh t = \int_{t_0}^{t_1}\realPart\pset[\big]{ \ct{y(t)}u(t) }\td t,
    \end{equation}
    i.e., the supply is invariant under time-varying transformations.
\end{lemma}

\begin{proof}
    Suppose first that $(x,u,y)$ is a state-input-output solution of \eqref{eq:tv_system}, and let $(\wh x,\wh u,\wh y)\coloneqq(x\circ\theta,u\circ\theta,\dot\theta(y\circ\theta))$. Then it holds that
    \begin{align*}
        \dot{\wh x} &= \dot\theta(\dot x\circ\theta) = \dot\theta(A\circ\theta)(x\circ\theta) + \dot\theta(B\circ\theta)(u\circ\theta) = \wh A\wh x + \wh B\wh u, \\
        \wh y &= \dot\theta(y\circ\theta) = \dot\theta(C\circ\theta)(x\circ\theta) + \dot\theta(D\circ\theta)(u\circ\theta) = \wh C\wh x + \wh D\wh u,
    \end{align*}
    i.e., $(\wh x,\wh u,\wh y)$ is a state-input-output solution of \eqref{eq:tv_system_timeTrans}.
    Since $\theta$ is a diffeomorphism, the converse statement can be proven analogously.

    To prove \eqref{eq:tv_system_timeTrans_supply}, it is sufficient to note that
    \[
        \int_{\wh t_0}^{\wh t_1}\realPart\pset[\big]{ \ct{\wh y(\wh t)}\wh u(\wh t) }\td\wh t
        = \realPart\pset*{\int_{\theta^{-1}(t_0)}^{\theta^{-1}(t_1)} \dot\theta(\wh t) \ct{(y\circ\theta)(\wh t)}(u\circ\theta)(\wh t) \td\wh t}
        = \int_{t_0}^{t_1}\realPart\pset[\big]{ \ct{y(t)}u(t) }\td t
    \]
    is valid by applying the change of variable under the integral sign.
\end{proof}

\noindent Note that, due to the inverse function theorem, a continuously differentiable map $\theta\in\mathcal C^1(\wt\timeInt,\timeInt)$ is a diffeomorphism if and only if it is surjective and $\dot\theta\neq 0$ holds in $\wt\timeInt$. In particular, either $\dot\theta>0$ or $\dot\theta<0$ holds on the whole interval.
We prefer the choice of $\dot\theta>0$, since it preserves the order in the time interval, and consequently our properties of interest.

\begin{theorem}\label{thm:invariance}
    The properties of being passive \refPa{}, having a self-adjoint and positive semidefinite solution of the \refKYP{} inequality, having a~nonnegative supply \refNN{}, and admitting a \refPH{} representation, are all invariant under state space transformations $Z\in W^{1,1}_\loc(\timeInt,\linGroup_n(\C))$, input-output change of variables $V\in L^\infty_\loc(\timeInt,\GL[m])$ with $V^{-1}\in L^\infty_\loc(\timeInt,\GL[m])$, and time-varying diffeomorphisms $\theta\in\mathcal C^1(\wh\timeInt,\timeInt)$ with $\dot\theta>0$ pointwise.
\end{theorem}

\begin{proof}
    See \Cref{sec:systemTransformations}.
\end{proof}

\begin{remark}\label{rem:stationaryAtRest}
    A particularly interesting case of \Cref{lem:changeOfVariables} is when we define a time-varying change of variables using the state transition matrix, that is, $Z(t)\coloneqq\stm(t,t_0)$ for any fixed $t_0\in\timeInt$. Then the resulting system has the form
    \begin{equation*}
        \begin{split}
            \dot{\wt x}(t) &= \stm(t_0,t)B(t)u(t), \\
            y(t) &= C(t)\stm(t,t_0)\wt x(t) + D(t)u(t),
        \end{split}
    \end{equation*}
    in particular $\wt A=0$.
    This is not surprising, since for vanishing input signals $u\equiv 0$ we have $x(t)=\stm(t,t_0)x(t_0)$ for all $t\in\timeInt$, so $\wt x(t)\equiv x(t_0)$ must be constant.
\end{remark}

\begin{remark}
    Because of \Cref{thm:invariance}, to study our properties of interest, it can be useful to apply first state space and time-varying transformations to bring the system \eqref{eq:tv_system} to some specific form.
    For example, up to applying some time-varying transformation, we may always assume without loss of generality that the time interval $\timeInt$ is a fixed open interval of our choice, e.g.~$\timeInt=\R$ or $\timeInt=(0,+\infty)$.
\end{remark}

\subsection{Null space decomposition}\label{sec:null}
 \noindent We recall the class of functions of bounded variation, which will be used in studying the smoothness properties of solutions of the KYP inequality~\refKYP{}.

 \noindent Given any nonempty compact interval $[t_0,t_1]\subseteq\timeInt$, we denote by
 \[
 \partInt{[t_0,t_1]} = \set[\big]{ \partit = \set{\wt t_0,\wt t_1,\ldots,\wt t_K} \mid t_0 = \wt t_0 < \wt t_1 < \ldots < \wt t_K = t_1,\ K\in\N }
 \]
 the set of \emph{partitions} of $[t_0,t_1]$. Given a partition $\partit=\set{\wt t_0,\wt t_1,\ldots,\wt t_K}\in\partInt{[t_0,t_1]}$, we denote by $\norm{\partit}\coloneqq\sup_{1\leq k\leq K}(\wt t_k-\wt t_{k-1})$ its \emph{norm}.
 To keep the notation brief, we often write $\partit\in\partInt{}$ as a~shorthand for $\partit=\set{\wt t_0,\wt t_1,\ldots,\wt t_K}\in\partInt{[t_0,t_1]}$.
 In particular, given a function $f:[t_0,t_1]\to\C$, we denote its \emph{total variation} as \label{glo:totVar}
 \[
 \totVar[t_0][t_1](f) \coloneqq \sup_{\partit\in\partInt{}} \sum_{k=1}^K \abs{ f(\wt t_k) - f(\wt t_{k-1}) }.
 \]
 We denote then the spaces of \emph{functions of bounded variation} on $[t_0,t_1]$ and \emph{functions of locally bounded variation} on $\timeInt$ as
 \begin{align}\label{def:BV}
 & \BV([t_0,t_1],\C) \coloneqq \set{ f:[t_0,t_1]\to\C \mid \totVar[t_0][t_1](f) < \infty } \quad\text{and}\\
 \label{def:BVloc}
 & \BV_\loc(\timeInt,\C) \coloneqq \set{ f:[t_0,t_1]\to\C \mid \totVar[t_0][t_1](f) < \infty \text{ for all }t_0,t_1\in\timeInt,\ t_0\leq t_1 },
 \end{align}
 respectively. Furthermore, given a matrix function 
 $F:[t_0,t_1]\to\C^{m,n}$ we define its \emph{total variation} as
 \[
 \totVar[t_0][t_1](F) \coloneqq \sup_{\partit\in\partInt{}} \sum_{k=1}^K \norm{ F(\wt t_k) - F(\wt t_{k-1}) },
 \]
 and we define analogously the spaces of matrices of (locally) bounded variation $\BV([t_0,t_1],\C^{m,n})$ and $\BV_\loc(\timeInt,\C^{m,n})$.

Recall that every function of bounded variation $f\in\BV([t_0,t_1],\R)$ is differentiable at a.e.~$t\in(t_0,t_1)$ and can be split into $f=f_a+f_s$, where $f_a\in W^{1,1}((t_0,t_1),\R)$, and $f_s\in\BV([t_0,t_1],\R)$ satisfies $\dot f_s(t)=0$ for a.e.~$t\in(t_0,t_1)$. The splitting is unique up to fixing any value of $f_a$, and is usually chosen such that $f_a(t_0)=0$. The components $f_a$ and $f_s$ are sometimes called the \emph{absolutely continuous part} and the \emph{singular part} of $f$, respectively.

We now need to introduce an intermediate regularity assumption, stricter than bounded variation but weaker than absolute continuity.
Such functions appear in the literature with different names and equivalent definitions, like ``upper absolutely continuous'', ``semi-absolutely continuous'', ``absolute upper semicontinuous'', or ``of bounded variation with nonincreasing singular part'' (see e.g.~\cite{Lee78,Pon77,Pou01,Top17}). We combine these equivalent definitions into one.

 \begin{definition}\label{def:AUC_function}
     We call a function $f:[t_0,t_1]\to\R$ \emph{absolutely upper semicontinuous} and we write $f\in\AUC([t_0,t_1])$ if any of the following equivalent definitions are satisfied:
     \begin{enumerate}
         \item[\rm (i)] For every $\varepsilon>0$ there is $\delta>0$ such that, for every choice of $r_1,s_1,\ldots,r_K,s_K\in[t_0,t_1]$ with $r_1<s_1\leq r_2<s_2\leq\ldots\leq r_K<s_K$ it holds that
         \begin{equation*}
             \sum_{k=1}^K\pset{s_k-r_k}<\delta \implies \sum_{k=1}^K\pset[\big]{f(s_k)-f(r_k)} < \varepsilon.
         \end{equation*}
         \item[\rm (ii)] There exists $g\in L^1([t_0,t_1],\R)$ such that
         \begin{equation*}
             f(s) - f(r) \leq \int_r^s g(t)\td t
         \end{equation*}
         holds for all $r,s\in[t_0,t_1],\ r\leq s$.
         \item[\rm (iii)] The derivative $\dot f(t)$ exists for a.e.~$t\in[t_0,t_1]$, $\dot f\in L^1([t_0,t_1],\R)$ and
         \begin{equation*}
             f(s) - f(r) \leq \int_r^s \dot f(t)\td t
         \end{equation*}
         holds for all $r,s\in[t_0,t_1],\ r\leq s$.
         \item[\rm (iv)] $f\in\BV([t_0,t_1],\R)$ with monotonically nonincreasing singular part.
     \end{enumerate}
     We call a function $f:\timeInt\to\R$ \emph{locally absolutely upper semicontinuous} and write $f\in\AUC_\loc(\timeInt)$ if $f|_{[t_0,t_1]}\in\AUC([t_0,t_1])$ for all $t_0,t_1\in\timeInt,\ t_0\leq t_1$.
 \end{definition}
 \noindent We now extend the concept of absolute upper semi-continuity to complex pointwise Hermitian matrix functions. One possibility to do this would be to apply the definition entrywise, after splitting $\C$ into $\R\times\R$. However, to use the results in the context of the Loewner ordering, we will proceed differently.

 \begin{definition}
 We say that $Q:\timeInt\to\HerMat[n]$ is \emph{weakly monotonically increasing}, or simply \emph{weakly increasing} if $Q(t_0)\leq Q(t_1)$ for all $t_0,t_1\in\timeInt,\ t_0\leq t_1$.    
     Analogously, we say that $Q$ is \emph{weakly monotonically decreasing} or simply \emph{weakly decreasing} if $Q(t_0)\geq Q(t_1)$ for all $t_0,t_1\in\timeInt,\ t_0\leq t_1$.
 \end{definition}
\noindent Using the concept of weakly monotonically decreasing matrix functions, in \cite{MorH24} the following result is proven.
 \begin{theorem}\label{thm:AUC_matrix}
     Let $Q:[t_0,t_1]\to\HerMat[n]$ be a pointwise Hermitian matrix function. Then the following statements are equivalent.
     \begin{enumerate}[label=\rm(\roman*)]
         \item 
         The function $[t_0,t_1]\to\R,\ t\mapsto\frac{1}{2}\ct{x}Q(t)x$ is absolutely upper semicontinuous for all $x\in\C^n$.
         \item 
         For every $\varepsilon>0$ there exists $\delta>0$ such that, for every choice of $r_1,s_1,\ldots,r_K,s_K\in[t_0,t_1]$ with $r_1<s_1\leq r_2<s_2\leq\ldots\leq r_K<s_K$, it holds that
         \begin{equation*}
             \sum_{k=1}^K\pset{s_k-r_k}<\delta \implies \sum_{k=1}^K\pset[\big]{Q(s_k)-Q(r_k)} < \varepsilon I_n.
         \end{equation*}
         \item 
         There exists $G\in L^1([t_0,t_1],\HerMat[n])$ such that
         \begin{equation*}
             Q(s) - Q(r) \leq \int_r^s G(t)\td t
         \end{equation*}
         holds for all $r,s\in[t_0,t_1],\ r\leq s$.
         \item 
         The derivative $\dot Q(t)$ exists for a.e.~$t\in(t_0,t_1)$, $\dot Q\in L^1([t_0,t_1],\HerMat[n])$, and the matrix inequality
         \begin{equation*}
             Q(s) - Q(r) \leq \int_r^s \dot Q(t)\td t
         \end{equation*}
         holds for all $r,s\in[t_0,t_1],\ r\leq s$.
         \item 
         $Q\in\BV([t_0,t_1],\HerMat[n])$ with weakly monotonically decreasing singular part.
     \end{enumerate}
     \end{theorem}

 \noindent We then introduce the following definition.
\begin{definition}\label{def:AUC_matrix}
     We call a pointwise Hermitian matrix function $Q:[t_0,t_1]\to\HerMat[n]$ \emph{absolutely upper semi-continuous} and write $Q\in\AUC([t_0,t_1],\HerMat[n])$ if any of the equivalent statements in \Cref{thm:AUC_matrix} are satisfied. We call a pointwise Hermitian matrix function $Q:\timeInt\to\HerMat[n]$ \emph{locally absolutely upper semi-continuous} and write $Q\in\AUC_\loc(\timeInt,\HerMat[n])$ if $Q|_{[t_0,t_1]}\in\AUC([t_0,t_1],\HerMat[n])$ for all $t_0,t_1\in\timeInt,\ t_0\leq t_1$.
 \end{definition}
 
A classical example of a function of bounded variation that is not absolutely continuous is the \emph{Cantor function} (see, e.g.~\cite{Car00}).

 \begin{example}
     
     Consider the Cantor function $\Cantor:[0,1]\to[0,1]$ which is weakly monotonic increasing from $\Cantor(0)=0$ to $\Cantor(1)=1$, but its derivative is $\dot\Cantor=0$ a.e.~on $[0,1]$.
     In fact, $\Cantor\in\BV([0,1],\R)\setminus\AUC([0,1])$, since it coincides with its singular part and is not weakly monotonically decreasing.
     Note that, up to scaling or extending the definition interval appropriately, we can use $\Cantor\cdot I_n$ as an example of a function in $\BV([t_0,t_1],\posSD[n])\setminus\AUC([t_0,t_1],\posSD[n])$ or $\BV_\loc(\timeInt,\posSD[n])\setminus\AUC_\loc(\timeInt,\posSD[n])$.
 \end{example}

\noindent The following result from \cite{MorH24} characterizes weakly decreasing matrix functions in terms of their absolute semicontinuity and classical derivative.
 \begin{lemma}\label{lem:BV_decreasing}
     Let $Q:\timeInt\to\HerMat[n]$. Then $Q$ is weakly monotonically decreasing if and only if $Q\in\AUC_\loc(\timeInt,\HerMat[n])$ and $\dot Q(t)\leq 0$ for a.e.~$t\in\timeInt$.
 \end{lemma}

\noindent Below, we will study time-varying quadratic storage functions of the form
 \begin{equation}\label{eq:quadStorFunc}
    \quadSt{Q} : \timeInt \times \C^n \to \R, \qquad (t,x) \mapsto \frac{1}{2}\ct{x}Q(t)x,
 \end{equation}
 for general $Q:\timeInt\to\C^{n,n}$ such that $Q(t)=\ct{Q(t)}\geq 0$.
 
The following result from \cite{MorH24} characterizes the regularity of quadratic storage functions for passive LTV systems.

 \begin{theorem}\label{thm:storageFunctionAUC}
     Suppose that for some $Q:\timeInt\to\posSD[n]$, $\quadSt{Q}$ is a storage function for a passive LTV system of the form \eqref{eq:tv_system}. Then $Q\in\AUC_\loc(\timeInt,\posSD[n])$.
 \end{theorem}

 \noindent In general, the rank of $Q$, associated with a storage function $\quadSt{Q}$, is not necessarily constant. To address this, we state the following key theorem, also from \cite{MorH24}, which offers a smooth \emph{null space decomposition} for $Q$.
 This approach does not rely on the typical constant rank assumptions generally required for null space decompositions of matrix functions, as seen in works like \cite{Dol64,KunM24}.

 \begin{theorem}\label{thm:nullSpaceDec}
     Suppose that $\quadSt{Q}:\timeInt\times\C^n\to\R$ is a storage function for \eqref{eq:tv_system}. Then the following statements hold.
     \begin{enumerate}
         \item $r\coloneqq \rank\circ\,Q:\timeInt\to\set{0,1,\ldots,n}$ is weakly decreasing.
         \item There exists $V\in W^{1,1}_\loc(\timeInt,\GL[n])$ such that $\wt Q\coloneqq\ct{V}QV$ is weakly decreasing, of the form
         \begin{equation}\label{eq:nullSpaceDec}
             \wt Q(t) = \bmat{\wt Q_{11}(t) & 0 \\ 0 & 0_{n-r(t)}},
         \end{equation}
         and satisfies
         \begin{equation}\label{eq:nullSpaceLimit}
             \lim_{s\to t^\pm}\wt Q(s) = \bmat{\wt Q_{11}^\pm & 0 \\ 0 & 0_{n-r_\pm(t)}},
         \end{equation}
         where $\wt Q_{11}(t)\in\posDef[r(t)]$ and $\wt Q_{11}^\pm\in\posDef[r_\pm(t)]$ with $r_\pm(t)=\lim_{s\to t^\pm}r(t)$, for all $t\in\timeInt$.
         \item There exists $U\in W^{1,1}_\loc(\timeInt,\GL[n])$ pointwise unitary such that $\wh Q\coloneqq\ct{U}QU$ is of the form \eqref{eq:nullSpaceDec} and satisfies \eqref{eq:nullSpaceLimit}.
         \item If $Q\in W^{1,1}_\loc(\timeInt,\posSD[n])$, then $\wt Q,\wh Q\in W^{1,1}_\loc(\timeInt,\posSD[n])$ also holds.
     \end{enumerate}
 \end{theorem}

    \noindent It is important to emphasize that the size of $\wt Q_{11}$ in \eqref{eq:nullSpaceDec} is weakly decreasing in time.
    However, in every subinterval of $\timeInt$ where the rank is constant, the size of $\wt Q_{11}$ is also constant.

In the following lemma, we provide an inclusion result for the kernel of the  matrix that induces the storage function of passive systems, generalizing \cite[Proposition 11]{CheGH23} for LTI systems.

\begin{lemma}\label{lem:Pa_kernelInclusion}
    Suppose that $\quadSt{Q}$ as in \eqref{eq:quadStorFunc} with $Q\in\AUC_\loc(\timeInt,\posSD[n])$ is a storage function for a passive LTV system \eqref{eq:tv_system}.
    Then $\ker(Q(t))\subseteq\ker(Q(t)A(t)+\dot Q(t))\cap\ker(C(t))$ holds for a.e.~$t\in\timeInt$.
\end{lemma}

In the next section, we consider linear time-varying pH systems and their properties.

\section{Linear  time-varying pH systems}\label{sec:phsystem}
In this section we discuss the properties of port-Hamiltonian systems of the form in \Cref{def:pH}. For this it is often useful to rewrite the \refPH{} system in the alternative form
\begin{equation}\label{eq:PHS}
    \begin{split}
        \dot x + Kx &= \big(J-R\big)Qx + \big(G-P\big)u, \\
        y &= \ct{\big(G+P\big)}Qx + \big(S-N\big)u,
    \end{split}
\end{equation}
or in compact form 
\begin{equation}\label{eq:PHS_compact}
    \bmat{\dot x+Kx \\ -y} = \pset[\big]{L-W}\bmat{Qx \\ u},
\end{equation}
where it holds that
\[
L \coloneqq \bmat{J & G \\ -\ct{G} & N} = -\ct{L} \in \mathbb C^{n+m,n+m}, \qquad W(t)\in\posSD[n+m],\, \text{for a.e.~$t\in\timeInt$.}
\]
and
\begin{align*}
QK+\ct{K}Q=\dot Q.\end{align*} 
The Hamiltonian $\mathcal H:\timeInt\times\C^{n,n}\to\R,\ (t,x)\mapsto\frac{1}{2}\ct{x}Q(t)x$, which can be equivalently written as $\mathcal H=\quadSt{Q}$, is a continuous map, and it is continuously differentiable with respect to $x$ (in the real sense) and weakly differentiable with respect to $t$. In particular, $\mathcal H\in W^{1,1}_\loc(\timeInt\times\C^n,\R)$ with derivatives
\[
\nabla\mathcal H(t,x) = Q(t)x, \qquad
\frac{\partial\mathcal H}{\partial t}(t,x) = \frac{1}{2}\ct{x}\dot Q(t)x.
\]
Note that, since $\mathcal H$ is a real-valued function, for every $x\in W^{1,1}_\loc(\timeInt,\C^n)$ the chain rule
\begin{equation*}
    \dd{}{t}\pset[\big]{\mathcal H(\cdot,x)} = \pd{\mathcal H}{t} + \realPart\pset[\big]{\ct{\nabla\mathcal H}\dot x} = \frac{1}{2}\ct{x}Qx + \realPart\pset[\big]{\ct{x}Q\dot x}
\end{equation*}
holds. We thus have the following classical result for port-Hamiltonian systems.
\begin{theorem}\label{thm:pH_PBE}
    Let $(x,u,y)$ be any state-input-output solution of the pH system \eqref{eq:PHS}. Then the \emph{power balance equation}
    \begin{equation}\label{eq:pH_PBE}
        \dd{}{t}\mathcal H\pset[\big]{t,x(t)} = -\ct{\bmat{Q(t)x(t) \\ u(t)}}W(t)\bmat{Q(t)x(t) \\ u(t)} + \realPart\pset[\big]{\ct y(t)u(t)}
    \end{equation}
    and the \emph{dissipation inequality}
    \begin{equation}\label{eq:pH_dissIneq}
        \dd{}{t}\mathcal H\pset[\big]{t,x(t)} \leq \realPart\pset[\big]{\ct y(t)u(t)}
    \end{equation}
    hold for a.e.~$t\in\timeInt$.
\end{theorem}

\begin{proof}
    On the one hand, from the compact notation \eqref{eq:PHS_compact} it follows that
    \[
    \realPart\left(\ct{\bmat{Qx \\ u}}\bmat{\dot x+Kx \\ -y}\right) = -\ct{\bmat{Qx \\ u}}W\bmat{Qx \\ u} \leq 0
    \]
    almost everywhere.
    On the other hand, using \eqref{eq:Q_lyapunov}, we obtain 
    \begin{align*}
        \realPart\left(\ct{\bmat{Qx \\ u}}\bmat{\dot x+Kx \\ -y}\right)
        &= \frac{1}{2}\ct{x}(QK+\ct{K}Q)x + \realPart(\ct x Q\dot x-\ct{u}y)  \\
        &= \frac{1}{2}\ct{x}\dot Qx + \realPart(\ct xQ\dot x) - \realPart(\ct yu)
        = \dd{}{t}\mathcal H(\cdot,x) - \realPart(\ct yu).
    \end{align*}
    The power balance equation and the dissipation inequality follow immediately.
\end{proof}

\begin{remark}\label{rem:pH_PBE}
    Note that calling \eqref{eq:pH_dissIneq} a dissipation inequality is consistent with \Cref{def:passive}. In fact, since $\mathcal H=\quadSt{Q}\in W^{1,1}_\loc(\timeInt\times\C^n,\R)$, we obtain
    \[
    \quadSt{Q}\pset[\big]{t_1,x(t_1)} - \quadSt{Q}\pset[\big]{t_0,x(t_0)} = \int_{t_0}^{t_1}\dd{}{t}\mathcal H\pset[\big]{t,x(t)}\td t \leq \int_{t_0}^{t_1}\realPart\pset[\big]{\ct y(t)u(t)}\td t
    \]
    for all $t_0,t_1\in\timeInt,\ t_0\leq t_1$, i.e., $\mathcal H$ is a storage function and the system is passive.
    However, the power balance equation \eqref{eq:pH_PBE} provides more information than the dissipation inequality alone, since it includes an explicit term for the dissipated energy.
\end{remark}

\begin{remark}
    Another characterizing property of pH systems is that they can be expressed in terms of a Dirac structure, a resistive structure, and a Hamiltonian \cite{SchJ14} (or Lagrangian submanifold \cite{SchM18}). Up to making the system autonomous, this also applies to time-varying systems, including the ones of the form \eqref{eq:PHS}, see, e.g.~\cite{MehM19,Mor24}.
\end{remark}

\noindent The Hamiltonian $\mathcal{H}$ in the system formulation~\refPH{} depends on complex variables; however, when computing its derivative, we consider the differential in the real sense. The following remark puts it in relation to the complex derivative. 
\begin{remark}
    The real differential $\nabla\mathcal H$ of the Hamiltonian is closely connected to its \emph{Wirtinger derivatives}, see e.g.~\cite{Hen93}.
    More precisely, given a complex function $f:\C^n\to\C^m:z=z_\R+\imagUnit z_{\mathbb I}\mapsto f(z)$, differentiable with respect to both $z_\R$ and $z_{\mathbb I}$, its Wirtinger derivatives are defined as
    \[
    \pd{f}{z} = \frac{1}{2}\pset*{\pd{f}{z_\R} - \imagUnit\pd{f}{z_{\mathbb I}}} \qquad\text{and}\qquad
    \pd{f}{\overline{z}} = \frac{1}{2}\pset*{\pd{f}{z_\R} + \imagUnit\pd{f}{z_{\mathbb I}}}.
    \]
    Furthermore, if $g:\C^k\to\C^n$ is differentiable with respect to $z_\R,z_{\mathbb I}$, then the chain rules
    \begin{align*}
        \pd{}{z}(f\circ g) &= \pset*{\pd{f}{z}\circ g}^\top\pd{g}{z} + \pset*{\pd{f}{\overline z}\circ g}^\top\pd{\overline g}{z}, \\
        \pd{}{\overline z}(f\circ g) &= \pset*{\pd{f}{z}\circ g}^\top\pd{g}{\overline z} + \pset*{\pd{f}{\overline z}\circ g}^\top\pd{\overline g}{\overline z}
    \end{align*}
    are satisfied.
    When $f$ is real-valued, like in the case of the Hamiltonian, we deduce that $\pd{f}{\overline z}=\overline{\pd{f}{z}}$.
    When $g$ depends only on a real argument, like in the case of a state trajectory, it follows that $\pd{g}{\overline z}=\pd{g}{z}$.
    Under both assumptions, we obtain that
    \[
    \pd{}{z}(f\circ g) = \pd{}{\overline z}(f\circ g) = \realPart\pset*{\pset*{\pd{f}{z}\circ g}^\top \pd{g}{z}} = \realPart\pset*{\ct{\pset*{\pd{f}{\overline z}\circ g}} \pd{g}{z}}.
    \]
    Note that, up to applying the isomorphism $\R^{2n}\to\C^n,\ (z_\R,z_{\mathbb I})\mapsto z_\R+\imagUnit z_{\mathbb I}$, it is natural to identify the real differential $\nabla f$ with $2\pd{f}{\overline z}$, in particular
    \[
    \pd{}{z}(f\circ g) = \frac{1}{2}\realPart\pset*{\ct{(\nabla f\circ g)}\pd{g}{z}}.
    \]
    Assuming for the sake of simplicity that $Q$ is constant in time, and defining $f(x)\coloneqq\frac{1}{2}\ct{x}Qx$ and $g(t+\imagUnit s)\coloneqq x(t)$ for some (weakly) differentiable $x:\timeInt\to\C^n$, we obtain then
    \[
    \pd{f}{z} = \frac{1}{2}\overline{Qx}, \qquad
    \pd{f}{\overline z} = \frac{1}{2}Qx, \qquad
    \dd{}{t}(f\circ x) = 2\pd{}{z}(f\circ g) = \realPart\pset*{\ct{x}Q\dot x}.
    \]
\end{remark}

\subsection{Dissipativity properties of pH systems}
As in the linear time-invariant case, the port-Hamiltonian structure retains all dissipativity properties for time-varying systems, i.e. if a system is \refPH{} then it is passive \refPa{}, the KYP inequality has a solution according to \refKYP{} and the supply is \refNN{}. The implication that \refKYP{} implies \refPa{} and \refNN{} was already discussed in \cite{LozBEM00}. In particular, it is well known that the passivity for quadratic and differentiable storage functions is equivalent to the existence of nonnegative solutions to the KYP inequality~\eqref{eq:KYP}, see e.g.\ \cite[Theorem 16]{HilM80}. For completeness, we recall this result as well and show in the following theorem that the \refPH{} property implies all the other dissipativity properties.

\begin{theorem}
\label{thm:ph2All}
   Consider a LTV system of the form \eqref{eq:tv_system}. Then the following implications hold.
    \begin{align*}
\refPH{} \quad  \implies \quad \refKYP{} \quad \implies \quad  \refPa{} \quad \implies \quad \refNN{}.
\end{align*}
    In addition, if the system is \refPH{} of the form \eqref{eq:pH_coefficients}, then $Q$ is a solution to the \refKYP{} inequality \eqref{eq:KYP}.
    Furthermore, if $Q\in W^{1,1}_\loc(\timeInt,\posSD[n])$ is a solution of the \refKYP{} inequality \eqref{eq:KYP}, then $\quadSt{Q}$ is a storage function for \eqref{eq:tv_system}.
\end{theorem}

\begin{proof}
\underline{\em \refPH{} $\implies$ \refKYP{}:}
    Consider a pH system of the form \eqref{eq:PHS}. Since $W\geq 0$ pointwise, we also have
    \[
    \bmat{Q & 0 \\ 0 & I_m}\pset[\big]{(W-L)+\ct{(W-L)}}\bmat{Q & 0 \\ 0 & I_m} = 2\ct{\bmat{Q & 0 \\ 0 & I_m}}W\bmat{Q & 0 \\ 0 & I_m} \geq 0,
    \]
    cf.~\eqref{eq:PHS_compact}.
    Then, since
    \[
    \bmat{Q & 0 \\ 0 & I_m}(L-W)\bmat{Q & 0 \\ 0 & I_m}
    = \bmat{Q & 0 \\ 0 & I_m}\bmat{A+K & B \\ -C & -D}
    = \bmat{QA+QK & QB \\ -C & -D},
    \]
    it follows that pointwise
    \begin{align*}
    \bmat{ -\ct AQ - QA - \dot Q & -QB+\ct C \\ C-\ct B Q & D+\ct D }
    &= \bmat{ -\ct AQ - QA - QK - \ct{K}Q & -QB+\ct C \\ C-\ct B Q & D+\ct D }  \\
    &= \bmat{Q & 0 \\ 0 & I_m}\pset[\big]{(W-L)+\ct{(W-L)}}\bmat{Q & 0 \\ 0 & I_m} \geq 0,
    \end{align*}
    i.e., $Q$ is a solution of the KYP inequality \eqref{eq:KYP}.\smallskip

\underline{\em \refKYP{} $\implies$ \refPa{}:}
    If $Q$ solves the KYP inequality \eqref{eq:KYP}, then 
    we can rewrite the KYP inequality as 
    \begin{equation}
       \label{eq:KYP2PA}
\ct{\begin{bmatrix}
            x(t)\\u(t)
        \end{bmatrix}}\bmat{ (-\ct AQ - QA - \dot Q)(t) & (-QB+\ct C)(t) \\ (C-\ct B Q)(t) & (D+\ct D)(t) }\begin{bmatrix}
            x(t)\\u(t)
        \end{bmatrix} \geq 0.
    \end{equation}
    Passivity as in \Cref{def:passive} 
    with respect to the storage function $V=\quadSt{Q}$ follows by integrating \eqref{eq:KYP2PA} over intervals $(t_0,t_1)\subseteq\mathbb{T}$ and rearranging terms, resulting in 
    \begin{align*}
    &~~~~\int_{t_0}^{t_1}\realPart\pset[\big]{\ct{y(t)}u(t)}\td t\\
&=\frac{1}{2}\int_{t_0}^{t_1}     \ct{\begin{bmatrix}
            x(t)\\u(t)
        \end{bmatrix}} \bmat{ 0 & \ct C(t) \\ C(t) & (D+\ct D)(t) }\begin{bmatrix}
            x(t)\\u(t)
        \end{bmatrix}\mathrm{d}t \\
        &\geq  \frac{1}{2}\int_{t_0}^{t_1}    \ct{\begin{bmatrix}
            x(t)\\u(t)
        \end{bmatrix}}\bmat{ (\ct AQ + QA + \dot Q)(t) & Q(t)B(t) \\ \ct{B(t)} Q(t) & 0 }\begin{bmatrix}
            x( t)\\u(t)
        \end{bmatrix}\mathrm{d}t\\
        &= \int_{t_0}^{t_1} \pset*{ \frac{1}{2}\ct{\dot x(t)}Q(t)x(t) + \frac{1}{2}\ct{x(t)}\dot Q(t)x(t) + \frac{1}{2}\ct{x(t)}Q(t)\dot x(t) } \td t \\
        &= \int_{t_0}^{t_1} \dd{}{t} V\pset[\big]{t,x(t)} \td t
        = V\pset[\big]{t_1,x(t_1)}- V\pset[\big]{t_0,x(t_0)},
    \end{align*}
 which is the dissipation inequality \eqref{def:passive_ineq}. This shows that \refKYP{} implies \refPa{} with the requested storage function. 
    \smallskip
    
\underline{\em \refPa{}$\implies$ \refNN{}:}
    Since the system is passive, then according to \Cref{def:passive}, there exists a storage function $V:\timeInt\times\C^n\rightarrow\R$ such that $V(t_1,x(t_1))\leq V(t_0,x_0)+\int_{t_0}^{t_1}\realPart(\ct{y}(s)u(s)) {\rm d}s$ and $V(t_0,0)=0$ holds for all $t_0\in\timeInt$. Choosing $x_0=0$, it follows that $\int_{t_0}^{t_1}\realPart(\ct{y(s)} u(s)) {\rm d}s\geq V(t_1,x(t_1))\geq 0$ which shows that the supply is nonnegative \refNN{} for all $t_0,t_1\in\timeInt,\ t_0\leq t_1$.
\end{proof}

\noindent In \Cref{thm:ph2All} we have shown that if a system is \refPH{}, then all other properties follow. The converse implications (i.e. from \refNN{}, \refPa{} and \refKYP{} to \refPH{}) are discussed separately in detail in the following sections, as they require more involved derivations.

\subsection{Equivalent representations 
for pH systems}\label{sec:equiv}
In this subsection, we review different representations for pH systems and extend them to the case of complex systems. We recall that the property of admitting a pH \emph{representation} is invariant under coordinate transformations, see \Cref{lem:PHS_changeOfVariables} in the appendix.
Note that typically for the description of a LTI system via a transfer function in the frequency domain, one also has different possibilities to choose the coefficients for the same transfer function, which is again invariant under coordinate transformations. Choosing a specific set of coefficients, one typically speaks of a \emph{realization} of the transfer function.

A key ingredient in the construction of pH representations is the null space decomposition in \Cref{thm:nullSpaceDec}. 
For a \refPH{} system of the form \eqref{eq:PHS}, we obtain from \Cref{thm:ph2All} that the system is passive and $\mathcal H=\quadSt{Q}$ is a storage function.
Thus, because of \Cref{thm:nullSpaceDec}, there is a pointwise unitary $U\in W^{1,1}_\loc(\timeInt,\GL[n])$ such that $\wt Q=\ct{U}QU$ has pointwise the form \eqref{eq:nullSpaceDec} .
Then, by applying the change of variables $x=U\wt x$, we obtain an equivalent pH system where $Q$ is replaced by $\wt Q$ in the representation.
Therefore, up to applying a pointwise unitary change of variables, we may assume without loss of generality that $Q$ is of the form \eqref{eq:nullSpaceDec}.

It should be noted, however, that even by fixing $Q\in W^{1,1}_\loc(\timeInt,\posSD[n])$, i.e., the Hamiltonian, the pH representation \eqref{eq:PHS} is not unique.
The following lemma characterizes the degrees of freedom in the representation.
\begin{lemma}\label{lem:pH_degreesOfFreedom}
    Suppose that an LTV system \eqref{eq:tv_system} is \refPH{} with Hamiltonian $\mathcal H=\quadSt{Q}$ defined by a matrix function $Q\in W^{1,1}_\loc(\timeInt,\posSD[n])$ of the form \eqref{eq:nullSpaceDec} (for simplicity leaving out the tilde $\tilde{}$ ). Express the system in the form \eqref{eq:PHS} and partition the coefficients pointwise as
    \[
    K = \bmat{K_{11} & K_{12} \\ K_{21} & K_{22}}, \quad J = \bmat{J_{11} & J_{12} \\ J_{21} & J_{22}}, \quad R = \bmat{R_{11} & R_{12} \\ R_{21} & R_{22}}, \quad G = \bmat{G_1 \\ G_2}, \quad P = \bmat{P_1 \\ P_2},
    \]
    consistently with \eqref{eq:nullSpaceDec}. Partition analogously
    \[
    A = \bmat{A_{11} & A_{12} \\ A_{21} & A_{22}}, \qquad B = \bmat{B_1 \\ B_2}, \qquad C = \bmat{C_1 & C_2}.
    \]
    Then $A_{12}=0$, $C_2=0$, and the following coefficients are uniquely determined a.e.~on $\timeInt$:
 \begin{subequations}\label{eq:pH_necessary}
    \begin{align}
        K_{12} &= 0, \\
        K_{22} &= A_{22}, \\
        R_{11} &= -\frac{1}{2}\pset[\big]{A_{11}Q_{11}^{-1}+Q_{11}^{-1}\ct{A}_{11}(t)+Q_{11}^{-1}\dot Q_{11}Q_{11}^{-1}}, \label{eq:pH_necessary_R11} \\
        G_1 &= \frac{1}{2}\pset[\big]{Q_{11}^{-1}\ct{C}_1+B_1}, \label{eq:pH_necessary_G1} \\
        P_1 &= \frac{1}{2}\pset[\big]{Q_{11}^{-1}\ct{C}_1-B_1}, \label{eq:pH_necessary_P1} \\
        N &= \frac{1}{2}\pset[\big]{D-\ct{D}}, \label{eq:pH_necessary_N} \\
        S &= \frac{1}{2}\pset[\big]{D+\ct{D}}. \label{eq:pH_necessary_S}
    \end{align}
    Furthermore, the coefficients $J(t)$, $R_{12}(t)$, $R_{22}(t)$ and $P_2(t)$ can be arbitrarily chosen for all $t\in\timeInt$, as long as $W(t)\in\posSD[n]$ and $J(t)=-\ct{J(t)}$ hold for a.e.~$t\in\timeInt$. In that case, the remaining coefficients are uniquely determined as
    \begin{align}
        K_{11} &= \pset[\big]{ J_{11}-R_{11}}Q_{11} - A_{11}, \\
        K_{21} &= \pset[\big]{ J_{21}-R_{21} }Q_{11} - A_{21}, \\
        G_2 &= B_2 + P_2.
    \end{align}
    \end{subequations}
\end{lemma}

\begin{proof}
    Since $\quadSt{Q}$ is a storage function for \eqref{eq:tv_system}, we deduce from \Cref{lem:Pa_kernelInclusion} that $A_{12}=0$ and $C_2=0$. Because of the specific form of $Q$ and $\dot Q$, the coefficients $K,J,R,Q,G,P,S,N$ determine a pH formulation for \eqref{eq:tv_system} if and only if
    \[
    \bmat{ (J_{11}-R_{11})Q_{11} - K_{11} & -K_{12} & G_1 - P_1 \\ (J_{21}-R_{21})Q_{11} - K_{21} & -K_{22} & G_2 - P_2 \\ \ct{(G_1+P_1)}Q_{11} & 0 & S - N}
    = \bmat{ A_{11} & 0 & B_1 \\ A_{21} & A_{22} & B_2 \\ C_1 & 0 & D }
    \]
    holds pointwise, $J$ and $N$ are pointwise skew-symmetric, and we have the two conditions
    \begin{equation}\label{eq:pH_degreesOfFreedom:1}
        \bmat{Q_{11}K_{11} + \ct{K}_{11}Q_{11} & Q_{11}K_{12} \\ \ct{K}_{12}Q_{11} & 0} = \bmat{\dot Q_{11} & 0 \\ 0 & 0}    
    \end{equation}
    and
    \[
    W(t) = \bmat{R_{11}(t) & R_{12}(t) & P_1(t) \\ R_{21}(t) & R_{22}(t) & P_2(t) \\ \ct{P}_1(t) & \ct{P}_2(t) & S(t)} \in \posSD[n+m],\, \text{for a.e.~$t\in\timeInt$.}
    \]
    We immediately see that $K_{12}=0$ and $K_{22}=A_{22}$ are necessary conditions.
    Splitting $D$ into its Hermitian and skew-Hermitian parts, we also obtain $S=\frac{1}{2}(D+\ct{D})$ and $N=\frac{1}{2}(D-\ct{D})$.
    By combining $G_1-P_1=B_1$ and $\ct{(G_1+P_1)}Q_{11}=C_1$, we deduce that $G_1=Q_{11}^{-1}\ct{C_1}+B_1$ and $P_1=Q_{11}^{-1}\ct{C_1}-B_1$.
    By studying the Hermitian part of $Q_{11}(J_{11}-R_{11})Q_{11}-Q_{11}K_{11}=Q_{11}A_{11}$, we obtain that
    \begin{align*}
        & -Q_{11}R_{11}Q_{11} - \frac{1}{2}(Q_{11}K_{11} + \ct{K}_{11}Q_{11}) = \frac{1}{2} (Q_{11}A_{11} + \ct{A}_{11}Q_{11} ) \\
        &\implies Q_{11}R_{11}Q_{11} = -\frac{1}{2}( Q_{11}A_{11} + \ct{A}_{11}Q_{11} + \dot Q_{11} ) \\
        &\implies R_{11} = -\frac{1}{2}( A_{11}Q_{11}^{-1} + Q_{11}^{-1}\ct{A}_{11} + Q_{11}^{-1}\dot Q_{11}Q_{11}^{-1} ).
    \end{align*}
    Suppose now that $J$, $R_{12}$, $R_{22}$ and $P_2$ are fixed. Then we immediately obtain $K_{11} = (J_{11}-R_{11})Q_{11} - A_{11}$, $K_{21} = (J_{21}-R_{11})Q_{11} - A_{21}$, and $G_2=B_2+P_2$ pointwise.
    It remains to verify that, as long as $J$ is pointwise skew-Hermitian and $W(t)\in\posSD[n+m]$ for a.e.~$t\in\timeInt$, every choice of $J$, $R_{12}$, $R_{22}$ and $P_2$, and the consequent values for $K_{11}$, $K_{21}$ and $G_2$, are consistent with the pH formulation.
    In fact, it only remains to verify that \eqref{eq:pH_degreesOfFreedom:1} holds.
    Clearly $Q_{11}K_{12}=0$ and $\ct{K}_{12}Q_{11}=0$ pointwise, since $K_{12}=0$.
    By construction, we also have
    \[
    Q_{11}K_{11} = Q_{11}(J_{11}-R_{11})Q_{11} - Q_{11}A_{11} = Q_{11}J_{11}Q_{11} - \frac{1}{2}Q_{11}A_{11} + \frac{1}{2}\ct{A}_{11}Q_{11} + \frac{1}{2}\dot Q_{11}
    \]
    a.e.~on $\timeInt$, from which, by extracting the Hermitian part, we obtain that
    $
    Q_{11}K_{11} + \ct{K}_{11}Q_{11} = \dot Q_{11},
    $
    concluding the proof.
\end{proof}

\begin{remark}\label{rem:pH_degreesOfFreedom}
In view of \Cref{lem:pH_degreesOfFreedom}, it is useful to discuss the degrees of freedom for the representation in \refPH{} for LTV systems, to be able to choose the coefficients in a canonical way.
The coefficients $R_{12}$, $R_{21}$, $R_{22}$ and $P_2$, although part of $W$, do not contribute to the dissipation of the system, since the dissipation term in the power balance equation \eqref{eq:pH_PBE} is
    \[
    \ct{\bmat{Qx \\ u}}W\bmat{Qx \\ u} = \ct{\bmat{x \\ u}}\bmat{ Q_{11}R_{11}Q_{11} & 0 & Q_{11}P_1 \\ 0 & 0 & 0 \\ \ct{P_1}Q_{11} & 0 & S }\bmat{x \\ u}.
    \]
    Furthermore, the coefficients $J_{12}$, $J_{22}$, $R_{12}$ and $R_{22}$ do not even influence the dynamics of the system, since
    \[
    (J-R)Q = \bmat{J_{11}-R_{11} & J_{12}-R_{12} \\ J_{21}-R_{21} & J_{22}-R_{22}}\bmat{Q_{11} & 0 \\ 0 & 0} = \bmat{(J_{11}-R_{11})Q_{11} & 0 \\ (J_{21}-R_{21})Q_{21} & 0 }.
    \]
It seems natural to choose $J_{12}=R_{12}=0$ and $J_{22}=R_{22}=0$, and therefore $J_{21}=-\ct{J}_{12}=0$, $R_{21}=\ct{R}_{12}=0$, and $P_2=0$ (since $W\geq 0$ must hold a.e.), from which we get $K_{21}=A_{21}$ and $G_2=B_2$.

Taking this choice, the only remaining degree of freedom is $J_{11}$, whose choice would then impact the coefficient $K_{11}$.
On the one hand, since we only require $J_{11}$ to be a.e.~skew-Hermitian, one possible choice is to take $J_{11}=0$, from which $K_{11}=-R_{11}Q_{11}-A_{11}$.
On the other hand, the choice of $K_{11}$ is constrained by the condition $Q_{11}K_{11}+\ct{K}_{11}Q_{11}=\dot Q_{11}$, i.e., that the Hermitian part of $Q_{11}K_{11}$ is $\frac{1}{2}\dot Q_{11}$, which is ensured by the definition of $K_{11}$ as a function of $J_{11}$.
In particular, we in general cannot impose the condition $K_{11}=0$; however, we can minimize the 2-norm of $Q_{11}K_{11}$ by imposing that $Q_{11}K_{11}=\frac{1}{2}\dot Q_{11}$, or equivalently
\begin{equation*}
        K_{11} = \frac{1}{2}Q_{11}^{-1}\dot Q_{11}, \qquad J_{11}  = \frac{1}{2}(A_{11}Q_{11}^{-1}-Q_{11}^{-1}\ct{A_{11}})
\end{equation*}
which is a valid choice, since $J_{11}$ is indeed skew-Hermitian.

To summarize, we have a canonical \refPH{} representation whenever $Q$ is of the form \eqref{eq:nullSpaceDec}, defined by
\begin{subequations}\label{eq:pH_canonical}\begin{align*}
            J &\coloneqq \frac{1}{2}\bmat{A_{11}Q_{11}^{-1}-Q_{11}^{-1}\ct{A}_{11} & 0 \\ 0 & 0}, 
            &G &\coloneqq \frac{1}{2}\bmat{Q_{11}^{-1}\ct{C}_1 + B_1 \\ B_2}, &
            N &\coloneqq \frac{1}{2}(D-\ct{D}),\\ 
            R &\coloneqq \frac{1}{2}\bmat{A_{11}Q_{11}^{-1}+Q_{11}^{-1}\ct{A}_{11} & 0 \\ 0 & 0},  
            &P &\coloneqq \frac{1}{2}\bmat{Q_{11}^{-1}\ct{C}_1 - B_1 \\ 0}, & 
            S &\coloneqq \frac{1}{2}(D+\ct{D}), \\
            K &\coloneqq \bmat{\frac{1}{2}Q_{11}^{-1}\dot Q_{11} & 0 \\ A_{21} & A_{22}}.
        \end{align*}
    \end{subequations}
\end{remark}
\noindent From \Cref{lem:pH_degreesOfFreedom} and \Cref{rem:pH_degreesOfFreedom} one easily obtains a condensed pH representation, which extends the results from the constant coefficient case \cite{GerPPS23,MehS23}. 

\begin{theorem}\label{thm:pH_decoupled}
Suppose that the LTV system \eqref{eq:tv_system} admits a \refPH{} formulation. Then, by performing a pointwise unitary change of basis and scaling the system from the left, the system can be equivalently rewritten in the reduced form 
\begin{subequations}\label{eq:pH_decoupled}
        \begin{align}
            \dot x_1 + K_{11}x_1 &= \pset[\big]{J_{11}-R_{11}}Q_{11}x_1 + \pset[\big]{G_1-P_1}u, \label{eq:pH_decoupled:1} \\
            \dot x_2 &= A_{21}x_1 + A_{22}x_2 + B_2u, \label{eq:pH_decoupled:2} \\
            y &= \ct{\pset[\big]{G_1+P_1}}Q_{11}x_1 + \pset[\big]{S-N}u, \label{eq:pH_decoupled:3}
        \end{align}
    \end{subequations}
    with $J_{11}(t)=-\ct{J_{11}(t)}\in\C^{r(t),r(t)}$, $N(t)=-\ct{N(t)}\in\C^{n-r(t),n-r(t)}$, $Q_{11}(t)\in\posDef[r(t)]$, $K_{11}(t)\in\C^{r(t),r(t)}$ satisfies $Q_{11}K_{11}+\ct{K_{11}}Q_{11}=\dot Q_{11}$, and such that
    \[
    W_{11}(t) \coloneqq \bmat{R_{11}(t) & P_1(t) \\ \ct{P_1(t)} & S(t)} \in \posDef[r(t)+m]
    \]
holds for a.e.~$t\in\timeInt$,  where $r:\timeInt\to\set{0,\ldots,n}$ is a weakly decreasing map.
A canonical choice of coefficients is then given by \eqref{eq:pH_canonical}.
\end{theorem}
\begin{proof}
We may assume without loss of generality (by applying a pointwise unitary change of basis) that $Q$ is in the form \eqref{eq:nullSpaceDec}, i.e., that
    \[
    Q(t) = \bmat{Q_{11}(t) & 0 \\ 0 & 0_{n-r(t)}},
    \]
    with $Q_{11}(t)\in\posDef[r(t)]$ for all $t\in\timeInt$, where $r\coloneqq\rank\circ\,Q:\timeInt\to\set{0,1,\ldots,n}$ is weakly decreasing.
    By applying \Cref{lem:pH_degreesOfFreedom} and choosing the remaining free coefficients as described in \Cref{rem:pH_degreesOfFreedom}, we obtain \eqref{eq:pH_decoupled}.
    The remaining conditions are then clearly satisfied by construction.
\end{proof}

\begin{remark}\label{rem:pH_decoupled}
    A possible interpretation of \Cref{thm:pH_decoupled} is that, up to first applying a pointwise unitary change of variables, the state $x(t)$ admits a natural pointwise splitting into two components $x_1(t)\in\C^{n_1(t)}$ and $x_2(t)\in\C^{n_2(t)}$, such that the Hamiltonian $\mathcal H$ is a strictly convex function of $x_1$, independent of $x_2$.
    The first state equation \eqref{eq:pH_decoupled:1} and the output equation \eqref{eq:pH_decoupled:3} are completely decoupled from the second state $x_2$, and form a pH system with Hamiltonian $\mathcal H_1:\timeInt\times\C^{n_1(t)}\to\R,\ (t,x_1)\mapsto\frac{1}{2}\ct{x_1}Q_{11}(t)x_1=\mathcal H(t,(x_1,0))$, in all open subintervals of $\timeInt$ where $Q$ has constant rank.
    The second state equation~\eqref{eq:pH_decoupled:2} that does not contribute to the Hamiltonian is completely unstructured, and its solution depends, in general, on the solution of \eqref{eq:pH_decoupled:1} and \eqref{eq:pH_decoupled:3},
    see analogous results in the time-invariant case, \cite{MehU23,MehS23}.
\end{remark}

\begin{remark}\label{rem:eqform}
If $Q$ is pointwise invertible, 
i.e. it has constant rank $n$, then no unitary change of variables is necessary, and the representation in \refPH{} can be equivalently reformulated as
\begin{align*}
    \dot x + \frac{1}{2}Q^{-1}\dot Qx &= (J-R)Qx + (G-P)u, \\
    y &= \ct{(G+P)}Qx + (S-N)u.
    \end{align*}
    In other words, for pointwise  $Q$ a canonical choice is $K=\frac{1}{2}Q^{-1}\dot Q$.
\end{remark}
\noindent One might have observed that in \Cref{def:pH} we have omitted any integrability condition for the new coefficients $K,J,R,G,P,S,N$. Although from \eqref{eq:pH_necessary_S} and \eqref{eq:pH_necessary_N} we immediately deduce that $S,N\in L^\infty_\loc(\timeInt,\C^{m,m})$, it is more problematic to infer integrability properties for the other coefficients without including additional assumptions for $Q$, as the following example illustrates.
\begin{example}\label{exm:pH_integrability}
Consider the one-dimensional LTV system
    \begin{equation*}
        \dot x = u, \qquad y = \max(-t,0)x
    \end{equation*}
    on $\timeInt=\R$, and let us look for a pH formulation of the form \eqref{eq:pH_coefficients}, by using the necessary conditions \eqref{eq:pH_necessary}.
    Since $D=0$, we necessarily have $S=N=0$.
    From $W(t)\in\posSD[n+m]$ we obtain $P(t)=0$ and therefore $G(t)=B(t)+P(t)=1$, for a.e.~$t\in\timeInt$.
    It follows that
    \[
    Q(t) = \ct{\pset[\big]{G(t)+P(t)}}Q(t) = C(t) = \min(-t,0)
    \]
    for a.e.~$t\in\timeInt$, in particular $Q\in W^{1,1}_\loc(\timeInt,\posSD[1])$, as expected. From \eqref{eq:pH_necessary_R11} we deduce that
    \[
    R(t) = -\frac{1}{2}Q(t)^{-1}\dot Q(t)Q(t)^{-1} = \frac{1}{2t^2}
    \]
    for a.e.~$t<0$, thus $R\notin L^1_\loc(\timeInt,\posSD[1])$, since $[-1,1]\subseteq\timeInt$ is compact and
    \[
    \int_{-1}^{1}\abs{R(t)}\td t \geq \int_{-1}^0\frac{1}{2t^2}\td t = +\infty.
    \]
    However, we can easily obtain a \refPH{} formulation by defining $R(t)=K(t)=0$ for $t\geq 0$ and $K(t)=-\frac{1}{2t^2}$ for $t<0$.
\end{example}

\noindent 
In order to avoid difficulties as in \Cref{exm:pH_integrability}, additional conditions on $Q$ are sufficient to guarantee the local integrability of the remaining coefficients.

\begin{theorem}
    Suppose that an LTV system of the form \eqref{eq:tv_system} has a \refPH{} representation such that $Q\in W^{1,1}_\loc(\timeInt,\posSD[n])$ has constant rank. Then there exists a pH representation such that $K,J,R\in L^1_\loc(\timeInt,\C^{n,n})$, $G,P\in L^2_\loc(\timeInt,\C^{n,m})$, and $S,N\in L^\infty_\loc(\timeInt,\C^{m,m})$.
\end{theorem}
\begin{proof}
Since $Q$ has constant rank $r$, by \Cref{thm:nullSpaceDec} there exists a pointwise unitary change of variables, such that the transformed system has
\[
Q = \bmat{Q_{11} & 0 \\ 0 & 0}, \qquad \dot Q = \bmat{\dot Q_{11} & 0 \\ 0 & 0},
\]
with $Q_{11}\in W^{1,1}_\loc(\timeInt,\posDef[r])$. 
In particular, by \Cref{lem:inverseContinuous} then $Q_{11}^{-1}\in W^{1,1}_\loc(\timeInt,\posDef[r])$.     Then, by the generalized H\"older inequality (\Cref{thm:genHolder}) the choice of the pH coefficients as in \eqref{eq:pH_canonical} yields that $K,J,R\in L^1_\loc(\timeInt,\C^{n,n})$, $G,P\in L^2_\loc(\timeInt,\C^{n,m})$, and $S,N\in L^\infty_\loc(\timeInt,\C^{m,m})$.
\end{proof}

\begin{remark}\label{rem:relto BeaMXZ}
The presented formulation of linear time-varying pH systems is closely related to the one presented in \cite{BeaMXZ18} in the case where $E=I_n$.     The main differences are that the formulation in \cite{BeaMXZ18} is for real variables, and that the conditions
    \begin{subequations}
    \begin{align}
        & Q(K-JQ) + \ct{(K-JQ)}Q = \dot Q \qquad\text{and} \label{eq:BeaMXZ18_cond_1} \\
        & W_Q \coloneqq \bmat{QRQ & QP \\ \ct{P}Q & S} \in \posSD[n+m] \label{eq:BeaMXZ18_cond_2}
    \end{align}
    \end{subequations}
    replace
    \begin{align}
         J=-\ct{J},\qquad  \label{eq:PH_cond_1} 
         QK+\ct{K}Q=\dot Q, \qquad
         W = \bmat{R & P \\ \ct{P} & S} \in \posSD[n+m]. 
    \end{align}
It is clear that conditions \eqref{eq:PH_cond_1} imply \eqref{eq:BeaMXZ18_cond_1} and \eqref{eq:BeaMXZ18_cond_2}.
Furthermore, the converse statement also holds, up to switching to an equivalent formulation. This can be seen as follows.
    
Suppose that conditions \eqref{eq:BeaMXZ18_cond_1} and \eqref{eq:BeaMXZ18_cond_2} are satisfied.
Since     \Cref{thm:pH_PBE}, \Cref{rem:pH_PBE} and \Cref{lem:PHS_changeOfVariables} apply to the formulation in \cite{BeaMXZ18} as well (see Theorem 15 and Theorem 18 in \cite{BeaMXZ18}), it follows, in particular, that $\quadSt{Q}$ is a storage function for \eqref{eq:tv_system} and we may assume, up to pointwise unitary change of variables, that $Q$ is of the form \eqref{eq:nullSpaceDec}.
Then condition \eqref{eq:BeaMXZ18_cond_1} is equivalent to
\[
    \bmat{Q_{11}(K_{11}-J_{11}Q_{11}) + \ct{(K_{11}-J_{11}Q_{11})}Q_{11} & Q_{11}K_{12} \\ \ct{K_{12}}Q_{11} & 0} = \bmat{\dot Q_{11} & 0 \\ 0 & 0},
\]
i.e., the Hermitian part of $Q_{11}(K_{11}-J_{11}Q_{11})$ is $\frac{1}{2}\dot Q_{11}$ and $K_{12}=0$.
Furthermore, condition \eqref{eq:BeaMXZ18_cond_2} is equivalent to
\[
    \bmat{Q_{11}(t)R_{11}(t)Q_{11}(t) & 0 & Q_{11}(t)P_1(t) \\ 0 & 0 & 0 \\ \ct{P_1}(t)Q_{11}(t) & 0 & S(t)} \in \posSD[n+m],
\]
    i.e.,
\[
    \bmat{R_{11}(t) & P_1(t) \\ \ct{P_1}(t) & S(t)} \in \posSD[r(t)+m],
\]
for all $t\in\timeInt$.
This suggests to replace $K,J,R,G,P$ with
\begin{align*}
        \wt K &= \bmat{\wt K_{11} & \wt K_{12} \\ \wt K_{21} & \wt K_{22}} \coloneqq \bmat{K_{11} - \frac{1}{2}(J_{11}+\ct{J_{11}})Q_{11} & 0 \\ K_{21} & K_{22}}, \\
        \wt J &= \bmat{\wt J_{11} & \wt J_{12} \\ \wt J_{21} & \wt J_{22}} \coloneqq \bmat{\frac{1}{2}(J_{11}-\ct{J_{11}}) & -\ct{(J_{21}-R_{21})} \\ J_{21}-R_{21} & 0}, \\
        \wt R &= \bmat{\wt R_{11} & \wt R_{12} \\ \wt R_{21} & \wt R_{22}} \coloneqq \bmat{R_{11} & 0 \\ 0 & 0}, \\
        \wt G &= \bmat{G_1 \\ G_2 - P_2}, \\
        \wt P &= \bmat{P_1 \\ 0}.
    \end{align*}
In fact, the matrix functions $\wt K,\wt J,\wt R,\wt G,\wt P$ satisfy the conditions \eqref{eq:PH_cond_1} by construction.
Furthermore, it holds that
    \begin{align*}
        (\wt J-\wt R)Q - \wt K &= \bmat{ \pset[\big]{\frac{1}{2}(J_{11}-\ct{J}_{11})-R_{11}}Q_{11} - K_{11} + \frac{1}{2}(J_{11}+\ct{J}_{11})Q_{11} & 0 \\ (J_{21}-R_{21})Q_{11} - K_{21} & -K_{22} }  \\
        &= \bmat{(J_{11}-R_{11})Q_{11} - K_{11} & -K_{12} \\ (J_{21}-R_{21})Q_{11} - K_{21} & -K_{22}} = (J-R)Q - K, \\
        \wt G-\wt P &= \bmat{G_1 - P_1 \\ G_2 - P_2} = G - P, \\
        \ct{(\wt G+\wt P)}Q &= \bmat{\ct{(G_1 + P_1)} & \ct{(G_2-P_2)}}\bmat{Q_{11} & 0 \\ 0 & 0} = \bmat{\ct{(G_1+P_1)}Q_{11} & 0} \\
        &= \bmat{\ct{(G_1 + P_1)} & \ct{(G_2 + P_2)}}\bmat{Q_{11} & 0 \\ 0 & 0} = \ct{(G+P)}Q.
\end{align*}    
Thus the system that we have obtained by replacing $K,J,R,G,P$ with $\wt K,\wt J,\wt R,\wt G,\wt P$ is equivalent to the original one.
One can then apply the inverse of the unitary change of variables that we used to partition $Q$ to reformulate the pH system in terms of the original variables.
\end{remark}

\noindent Under some constant rank assumption for $Q$, using pointwise invertible but non-unitary
transformations we can even make the Hamiltonian autonomous as the next results shows.
\begin{theorem}\label{thm:autonomousQ}
    Consider a \refPH{} system of the form \eqref{eq:PHS}, and suppose that $Q$ has constant rank $r$ in $\timeInt$.
    Then there exists a change of variables with $Z\in W^{1,1}_\loc(\timeInt,\GL[n])$ such that, by applying \Cref{lem:changeOfVariables}, we obtain a pH system with
    \[
    \wt Q = \bmat{I_r & 0 \\ 0 & 0}
    \]
In particular, the 
transformed Hamiltonian $\wt{\mathcal H}(t,\wt x)=\quadSt{\wt Q}(t,\wt x)=\frac{1}{2}\norm{\wt x_1}^2$ does not depend explicitly on $t$.
\end{theorem}
\begin{proof}
    Since $Q$ has constant rank $r$, because of \Cref{thm:nullSpaceDec} there exists $V\in W^{1,1}_\loc(\timeInt,\GL[n])$ such that
\[
    \ct{V}QV = \bmat{Q_{11} & 0 \\ 0 & 0},
\]
where $Q_{11}\in W^{1,1}_\loc(\timeInt,\posDef[r])$.
    Let now $V_{11}\in W^{1,1}_\loc(\timeInt,\posDef[r])$ be the Cholesky factor of $Q_{11}$, cf.~\Cref{lem:Cholesky}, and let
\[
    Z \coloneqq V\bmat{V_{11}^{-1} & 0 \\ 0 & I_{n-r}} \in W^{1,1}_\loc(\timeInt,\GL[n]).
\]
By applying \Cref{lem:changeOfVariables} we then obtain a pH system with
\[
    \wt Q = \ct{Z}QZ = \ict{\bmat{V_{11} & 0 \\ 0 & I_{n-r}}} \ct{V} Q V \bmat{V_{11}^{-1} & 0 \\ 0 & I_{n-r}}
    = \bmat{\ict{V}_{11}\ct{V}_{11}V_{11}V_{11}^{-1} & 0 \\ 0 & 0} = \bmat{I_r & 0 \\ 0 & 0},
\]
as desired.
\end{proof}
\noindent Note that, while in the proof of \Cref{thm:autonomousQ} we define $V_{11}$ as the Cholesky factor of $Q_{11}$, it is sufficient to choose any $V_{11}\in W^{1,1}_\loc(\timeInt,\GL[n])$ such that $Q_{11}=\ct{V}_{11}V_{11}$.
\section{The KYP inequality and its implications for pH systems}\label{sec:KYPinequality}

In this section, in order to study the relationship between \refKYP{} and \refPH{}, we study weakly differentiable solutions $Q\in W^{1,1}_\loc(\timeInt,\C^{n,n})$ to \refKYP{}. This restriction of the solution class is motivated by the following example of an LTV system with continuous coefficients for which the solutions to the KYP inequality~\eqref{eq:KYP} are not continuously differentiable.
\begin{example}\label{ex:notdiffkyp}
Consider the linear time-varying system
\[
    \dot x(t) = u(t), \qquad y(t) = \abs{t}x(t),
\]
for $t\in\R$ and $x\in\C$. This system has continuous, but not continuously differentiable coefficients.
Since $D=0$, the only solution of the KYP inequality is $Q(t)=\abs{t}$, which is continuous but only weakly differentiable.
\end{example}

\noindent In order to present our analysis of the solutions to the KYP inequality, we first prove the following result.

\begin{theorem}\label{thm:posDef_equiv}
    Let $M\in L^1_\loc(\timeInt,\HerMat[n])$. Then the following statements are equivalent.
    \begin{enumerate}[label=\rm(\roman*)] 
        \item\label{it:posDef_equiv:1} $M\in L^1_\loc(\timeInt,\posSD[n])$.
        \item\label{it:posDef_equiv:2} $\int_{t_0}^{t_1}\ct{v(t)}M(t)v(t)\td t\geq 0$ for all $t_0,t_1\in\timeInt,\ t_0\leq t_1$ and $v\in L^1([t_0,t_1],\C^n)$.
        \item\label{it:posDef_equiv:3} $\int_{t_0}^{t_1}\ct{v(t)}M(t)v(t)\td t\geq 0$ for all $t_0,t_1\in\timeInt,\ t_0\leq t_1$ and $v\in\mathcal V_{t_0,t_1}$, where $\mathcal V_{t_0,t_1}\subseteq L^1([t_0,t_1],\C^n)$ is any family of subsets that contains constant functions. 
        \item\label{it:posDef_equiv:4} $\int_{t_0}^{t_1}M(t)\td t\in\posSD[n]$ for all $t_0,t_1\in\timeInt,\ t_0\leq t_1$.
    \end{enumerate}
Note that we allow the integrals in \ref{it:posDef_equiv:2} and \ref{it:posDef_equiv:3} to be $+\infty$.
\end{theorem}
\begin{proof}
    The implication \ref{it:posDef_equiv:1}$\implies$\ref{it:posDef_equiv:2} is obvious, since $\ct{v(t)}M(t)v(t)\geq 0$ holds for a.e.~$t\in[t_0,t_1]$.
    The implication \ref{it:posDef_equiv:2}$\implies$\ref{it:posDef_equiv:3} is self-evident, since $\mathcal V_{t_0,t_1}\subseteq L^1([t_0,t_1],\C^n)$.
    The implication \ref{it:posDef_equiv:3}$\implies$\ref{it:posDef_equiv:4} is also clear, since for every $v\in\C^n$ the corresponding constant function $\wt v:[t_0,t_1]\to\C^n,\ t\mapsto v$ is contained in $\mathcal V_{t_0,t_1}$, and therefore
    \[
    \infty > \ct{v}\pset*{\int_{t_0}^{t_1}M(t)\td t}v = \int_{t_0}^{t_1}\ct{\wt v(t)}M(t)\wt v(t)\td t \geq 0.
    \]
    It remains to show that \ref{it:posDef_equiv:4}$\implies$\ref{it:posDef_equiv:1}.
    For every $v\in\C^n$, let us define $f^v:\timeInt\to\R,\ t\mapsto\ct{v}M(t)v$. We clearly have $f^v\in L^1_\loc(\timeInt,\R)$ and
\[
    \int_{t_0}^{t_1}f^v(t)\td t = \ct{v}\pset*{\int_{t_0}^{t_1}M(t)\td t}v \geq 0
\]
 for all $t_0,t_1\in\timeInt,\ t_0\leq t_1$, and thus $f^v(t)\geq 0$ for a.e.~$t\in\timeInt$, because of \Cref{lem:positiveIntegrals}.
    In other words, we have $\ct{v}M(t)v\geq 0$ for every $v\in\C^n$ and for a.e.~$t\in\timeInt$, i.e., $M\in L^1_\loc(\timeInt,\posSD[n])$, as claimed.
\end{proof}

\noindent We obtain the following characterization of weakly differentiable solutions to \refKYP{}. 
\begin{corollary}
A matrix function  $Q\in W^{1,1}_\loc(\timeInt,\posSD[n])$ 
is a solution of the \refKYP{} inequality \eqref{eq:KYP} if and only if the bounded linear operator
    \begin{gather*}
        \mathrm{KYP}(Q)_{t_0,t_1} : L^\infty([t_0,t_1],\C^n)\times L^2([t_0,t_1],\C^m) \to L^1([t_0,t_1],\C^n)\times L^2([t_0,t_1],\C^m), \\
        \bmat{x \\ u} \mapsto \bmat{-\ct{A}Q-QA-\dot Q & \ct{C}-QB \\ C-\ct{B}Q & D+\ct{D}}\bmat{x \\ u}
    \end{gather*}
    satisfies
    \begin{equation}\label{eq:KYP_operator_pos}
        \int_{t_0}^{t_1} \ct{\bmat{x(t) \\ u(t)}}\pset*{\mathrm{KYP}(Q)_{t_0,t_1}\bmat{x \\ u}}(t)\td t \geq 0
    \end{equation}
    for all $t_0,t_1\in\timeInt,\ t_0\leq t_1$.
\end{corollary}

\begin{proof}
    The fact that the operator $\mathrm{KYP}(Q)_{t_0,t_1}$ is well-defined, linear and bounded is clear, since
    \begin{align*}
        & \norm{(- \ct{A}Q - QA-\dot Q )x + (\ct{C}-QB)u}_{L^1} \\
        &\qquad\leq (2\norm{A}_{L^1}\norm{Q}_{L^\infty}+\norm{\dot Q}_{L^1})\norm{x}_{L^\infty} + (\norm{C}_{L^2}-\norm{Q}_{L^\infty}\norm{B}_{L^2})\norm{u}_{L^2}, \\
        & \norm{(C-Q\ct{B})x + (D+\ct{D})u}_{L^2} \leq (\norm{C}_{L^2}+\norm{Q}_{L^\infty}\norm{B}_{L^2})\norm{x}_{L^\infty} + 2\norm{D}_{L^\infty}\norm{u}_{L^2}
    \end{align*}
    for all $x\in L^\infty([t_0,t_1],\C^n)$ and $u\in L^2([t_0,t_1],\C^m)$.
    By defining
 \[
    M_Q \coloneqq \bmat{-\ct{A}Q-QA -\dot Q & \ct{C}-QB \\ C-\ct{B}Q & D+\ct{D}} \in L^1_\loc(\timeInt,\C^{n,n}),
\]
we obtain that $Q$ is a solution of the KYP inequality if and only if $M_Q(t)\geq 0$ for a.e.~$t\in\timeInt$.
    Furthermore, it holds that
\[
    \int_{t_0}^{t_1} \ct{\bmat{x(t) \\ u(t)}}\pset*{\mathrm{KYP}(Q)_{t_0,t_1}\bmat{x \\ u}}(t)\td t = \int_{t_0}^{t_1}\ct{v(t)}M_Q(t)v(t)\td t
    \]
for all $v=(x,u)\in\mathcal V_{t_0,t_1}\coloneqq L^\infty([t_0,t_1],\C^n)\times L^2([t_0,t_1],\C^m)$.
The assertion then follows immediately from \Cref{thm:posDef_equiv}.
\end{proof}

\begin{remark}
Condition \eqref{eq:KYP_operator_pos} can be interpreted as self-adjointness and positive semi-definiteness of the operator $\mathrm{KYP(Q)_{t_0,t_1}}$.
    In fact, denoting more in general
    \[
    \aset{f,g}_{t_0,t_1} = \int_{t_0}^{t_1}\ct{g(t)}f(t)\td t
    \]
    whenever $f,g\in L^1_\loc(\timeInt,\C^k)$ are such that $\ct{g}f\in L^1_\loc(\timeInt,\C)$, one can easily see that
    \[
    \aset*{ \mathrm{KYP(Q)_{t_0,t_1}}\bmat{x \\ u} , \bmat{\wt x \\ \wt u} }
    = \aset*{ \bmat{x \\ u} , \mathrm{KYP(Q)_{t_0,t_1}}\bmat{\wt x \\ \wt u} }
    \]
    for all $x,\wt x\in L^\infty([t_0,t_1],\C^n)$ and $u,\wt u\in L^2([t_0,t_1],\C^m)$, so  $\mathrm{KYP(Q)_{t_0,t_1}}=\ct{\mathrm{KYP(Q)}}_{\mathrm{t_0,t_1}}$.
    Similarly, \eqref{eq:KYP_operator_pos} is equivalent to
    \[
    \aset*{ \mathrm{KYP(Q)_{t_0,t_1}}\bmat{x \\ u} , \bmat{x \\ u} } \geq 0
    \]
    for all $x\in L^\infty([t_0,t_1],\C^n)$ and $u\in L^2([t_0,t_1],\C^m)$. 
    This property is also related to the extended concept of operators,  introduced in \cite{Mor24}, which are symmetric positive semi-definite with respect to certain \emph{spaces of port variables}.
\end{remark}

\subsection{Relation of the pH representation to the KYP inequality}

In this section, we show that pointwise invertible solutions of \eqref{eq:KYP} can be used to obtain a pH formulation of a given time-varying system \eqref{eq:tv_system}.

\begin{theorem}\label{thm:KYP_to_PH}
    Suppose that $Q\in W^{1,1}_\loc(\timeInt,\posSD[n])$ is a solution of the \refKYP{} inequality \eqref{eq:KYP} associated to the LTV system \eqref{eq:tv_system}.
    Then the system can be formulated as a \refPH{} system of the form \eqref{eq:PHS} with Hamiltonian $\mathcal H=\quadSt{Q}$ as in \eqref{eq:quadStorFunc}.
\end{theorem}
\begin{proof}
    Due to \Cref{thm:ph2All}, $\quadSt{Q}$ is a storage function for \eqref{eq:tv_system}.
    Then, because of \Cref{thm:invariance} and \Cref{thm:nullSpaceDec}, we may assume without loss of generality that
    \begin{alignat*}{5}
        Q(t) &= \bmat{Q_{11}(t) & 0 \\ 0 & 0}, &\qquad& Q_{11}(t)\in\posDef[r(t)], &\qquad& \text{for all }t\in\timeInt, \\
        \dot Q(t) &= \bmat{\dot Q_{11}(t) & 0 \\ 0 & 0} &\qquad& \dot Q_{11}(t)\in\C^{r(t),r(t)}, &\qquad& \text{for a.e.~}t\in\timeInt,
    \end{alignat*}
    where $r=\rank\circ\,Q:\timeInt\to\set{0,\ldots,n}$ is weakly decreasing.
    Define the coefficients $K,J,R,G,P,S,N$ following \eqref{eq:pH_canonical}, noting that $A_{12}=0$ and $C_2=0$, because of \Cref{lem:Pa_kernelInclusion}.
    By construction, we have
    \begin{align*}
        Q(t)K(t) + \ct{K(t)}Q(t) &= \bmat{ \frac{1}{2}\dot Q_{11}(t) & 0 \\ 0 & 0 } + \bmat{ \frac{1}{2}\dot Q_{11}(t) & 0 \\ 0 & 0 } = \bmat{\dot Q_{11}(t) & 0 \\ 0 & 0} = \dot{Q}(t),\\
    J(t)&=-\ct{J}(t),\ N=-\ct{N}(t),
    \end{align*}
    and
    \begin{align*}
        0 &\leq \bmat{ -Q(t)A(t) - \ct{A(t)}Q(t) - \dot{Q}(t) & \ct{C(t)} - Q(t)B(t) \\ C(t) - \ct{B(t)}Q(t) & D(t) + \ct{D(t)} } \\
        &= \bmat{-Q_{11}(t)A_{11}(t) - \ct{A_{11}(t)} - \dot{Q}_{11}(t) & 0 & \ct{C_1(t)} - Q_{11}(t)B_1(t) \\ 0 & 0 & 0 \\ C_1(t) - \ct{B_1(t)}Q_{11}(t) & 0 & D(t) + \ct{D(t)} } \\
        &= \bmat{2Q_{11}(t)R_{11}(t)Q_{11}(t) & 0 & 2Q_{11}(t)P_1(t) \\ 0 & 0 & 0 \\ 2\ct{P_1(t)}Q_{11}(t) & 0 & 2S(t)} \\
        &= 2\bmat{Q_{11}(t) & 0 & 0 \\ 0 & I_{n-r(t)} & 0 \\ 0 & 0 & I_m} \bmat{R_{11}(t) & 0 & P_1(t) \\ 0 & 0 & 0 \\ \ct{P_1(t)} & 0 & S(t)} \bmat{Q_{11}(t) & 0 & 0 \\ 0 & I_{n-r(t)} & 0 \\ 0 & 0 & I_m}
        = 2\ct{V(t)}W(t)V(t),
    \end{align*}
    where $V(t)\coloneqq\operatorname{diag}(Q_{11}(t),I_{n-r(t)},I_m)\in\GL[n+m]$, and therefore
    \[
    W(t) = \ict{V(t)}\pset*{\ct{V(t)}W(t)V(t)}V(t)^{-1} \geq 0,
    \]
    for a.e.~$t\in\timeInt$. Thus, the conditions for \refPH{} are satisfied, and we have obtained a pH representation of \eqref{eq:tv_system} with Hamiltonian $\mathcal H=\quadSt{Q}$.
\end{proof}

\begin{remark}\label{rem:recon}
Assuming complete reconstructability, it was shown in \cite[Lemma 3]{AndM74} that the solutions $Q$ to \refKYP{} are pointwise invertible. However, in contrast to the LTI case, the pH representation of LTV systems can be obtained without further invertibility assumptions on $Q$. The reason for this is that the additional term $K$ in \refPH{} allows for more freedom in choosing $Q$ that can also be non-invertible. 
\end{remark}

\subsection{Further remarks on the solvability of KYP inequalities}

In order to study the solvability of the KYP inequality, one has to solve the linear matrix inequality \eqref{eq:KYP}. In the LTI case, this is usually done using a quadratic programming approach \cite{BoyGFB94} or by computing the invariant subspaces of matrix pencils \cite{BeaMX15ppt}. For the LTV case, there have been attempts to solve this differential matrix inequality directly \cite{AndM74}. However, most of these methods are numerically infeasible, especially for large-scale systems.
In the context of dissipativity, solutions $Q$ to \eqref{eq:KYP} are also required to be pointwise positive semidefinite or even positive definite. Furthermore, in a LTV setting, $Q$ is often required to have a constant rank.

An approach to compute such solutions is available if $(D+\ct D)>0$ pointwise. Then using the Schur complement, one can solve the \refKYP{} in \eqref{eq:KYP} via the solution of a differential Riccati inequality of the form 
\begin{equation*}
    -\ct AQ - QA - \dot Q -(QB-\ct C)(D+\ct D)^{-1}(\ct B Q-C)\geq 0,
\end{equation*}
and, as a special case, one can solve the equality case, see \cite{AboFIJ12,EisEM19,Rei72} for a detailed analysis of these equations, and \cite{BenM13,Die92,DieE94, KenL85} for numerical methods.

If $(D+\ct D)\geq 0$ is pointwise but singular, then a strategy to solve the \refKYP{} can be devised only on the non-singular part of $(D+\ct D)$ as presented in \cite{AndM74} or for the time-invariant case in \cite{BeaMX15ppt}. Under constant rank assumptions, a recursive procedure is proposed to reduce the order of the state space system until a nonsingular $(D+\ct D)$ is encountered. Then, one can apply the Schur complement and solve the Riccati inequality on the reduced system. 
 
Alternatively, instead of solving the \refKYP{}, one can use the equality case for the quadratic form associated with \eqref{eq:KYP} by multiplying from the left by $\ct z=\begin{bmatrix} \ct x & \ct u \end{bmatrix}$ and from the right by $z$. This results in an associated differential-algebraic system, similar to the approach to solving the \refKYP{} in the LTI case, see \cite{BeaMX15ppt}. In the LTV case, the associated differential-algebraic system is given by 
\begin{equation*}
\begin{bmatrix}
    \phantom{-}0 & I & 0\\
-I & 0 & 0\\
\phantom{-}0 & 0 & 0\end{bmatrix} \frac{d}{dt}
       \begin{bmatrix}  \lambda \\ x \\ u\end{bmatrix}=
\begin{bmatrix}
0 & A & B\\
\ct A & 0 & -\ct C\\
          \ct B & -C & D+\ct D
\end{bmatrix}
\begin{bmatrix}\lambda \\ x \\ u\end{bmatrix}.
\end{equation*}
Inserting the ansatz $\lambda=Qx$, we get the two equations $0=-(\dot Q+QA +\ct AQ) x+(\ct C-QB)u$ and $0=(C-\ct BQ) x+(D+\ct D)u$, which correspond to the two equations obtained by multiplying with $z$ from the right in the equality case of \eqref{eq:KYP}.

\begin{remark}
  Equivalently to the \refKYP{} inequality \eqref{eq:KYP}, one can solve a \emph{Lur'e type inequality}
\begin{equation*}
\label{eq:LMI_QSR}
\begin{bmatrix}
\mathcal{L}(t) & \mathcal{W}(t) \\
\ct{\mathcal{W}}(t) & (D+\ct{D})(t)
\end{bmatrix} \geq 0,
\end{equation*}
where
\begin{align*}
\ct{A}(t)Q(t) + Q(t)A(t) +\dot{Q}(t) &= -\mathcal{L}(t) \\
\ct{C}(t) - P(t)B(t) &= \mathcal{W}(t)
\end{align*}
for $Q(t) = \ct{Q}(t) \geq 0$ pointwise. Most of the literature on the subject, see e.g. \cite{Bru13,ForD10,LozBEM00,ReiRV15}, considers the solution of the \emph{Lur'e equation} 
\begin{align*}
& Q(t)A(t)+\ct{A}(t) Q(t)+\dot{Q}(t)=-L(t) \ct{L}(t), \\
& C(t)-P(t)B(t)=\ct{L}(t) W(t), \\
& D(t)+\ct{D}(t)=\ct{W}(t) W(t) .
\end{align*} 
This is again a special case that restricts the solution space of the \refKYP{}. Thus, we refrain from looking at this approach in more detail.  
\end{remark}

\section{Passivity and its implications for pH systems}
\label{sec:PatopH}
We have already seen in Section~\ref{sec:phsystem} that port-Hamiltonian systems are passive, that is, \refPH{} implies \refPa{}. In this section, we derive the extra conditions under which the converse implication holds. For this analysis, we proceed via the available storage as discussed in Section~\ref{sec:null}.

\subsection{Available storage}\label{sec:storage}

\noindent We define the available storage for LTV systems, as it was introduced, e.g.~in~\cite{HilM80,LozBEM00}.

\begin{definition}\label{def:avstor}
    The \emph{available storage} of the LTV system \eqref{eq:tv_system} at time $t_0\in\timeInt$ and state $x_0\in\C^n$ is
    \begin{equation}\label{eq:availableStorage}
            \avSt(t_0,x_0)
            = \sup_{\avStArg{}}\left( -\int_{t_0}^{t_1}\realPart(\ct{y(t)}u(t))\td t\right )
            = -\inf_{\avStArg{}}\int_{t_0}^{t_1}\realPart(\ct{y(t)}u(t))\td t,
        \end{equation}
        where the supremum (or infimum) is taken among all times $t_1\in\timeInt,\ t_1\geq t_0$ and inputs $u\in L^2_\loc(\timeInt,\C^n)$, and $x\in W^{1,1}_\loc(\timeInt,\C^n)$ and $y\in L^2_\loc(\timeInt,\C^m)$ denote the state and output trajectories uniquely determined by $u$ and by the condition $x(t_0)=x_0$.
\end{definition}

\noindent Since the supremum in \Cref{def:avstor} is taken over a nonempty set, it is always well defined.
Furthermore, since for $t_1=t_0$ the integral in \eqref{eq:availableStorage} vanishes, the available storage is always nonnegative.
We can then intepret it as a possibly infinite-valued function $V_a:\timeInt\times\C^n\to[0,+\infty]$.
If we actually have $V_a(t_0,x_0)<+\infty$ for all $t_0\in\timeInt$ and $x_0\in\C^n$, we say that the available storage is finite.

The finiteness of the available storage is closely related to the passivity of the system \cite{HilM80,LozBEM00}. In fact, in \cite{MorH24} the following result is proven.

\begin{theorem}\label{thm:availableStorage}
    Consider an LTV system of the form \eqref{eq:tv_system}. Then the following statements are equivalent.
    \begin{enumerate}[label=\rm (\roman*)]
        \item The system is passive.
        \item There exists $\beta:\timeInt\times\C^n\to\R$ such that for every $t_0,t_1\in\timeInt,\ t_0\leq t_1$ and state-input-output solution $(x,u,y)$ it holds that
        \begin{equation*}
            \int_{t_0}^{t_1}\realPart\pset[\big]{\ct{y}(t)u(t)}\td t \geq \beta\pset[\big]{t_0,x(t_0)}.
        \end{equation*}
        \item The available storage function is finite.
        \item The available storage function is of the form $V_a=\quadSt{Q}$ for some matrix function $Q:\timeInt\to\posSD[n]$.
    \end{enumerate}
\end{theorem}

\noindent Since the available storage of a passive LTV system is the minimal storage function \cite{HilM80}, it is then natural to summarize its properties in the following corollary from \cite{MorH24}.

\begin{corollary}
    The available storage of a passive LTV system of the form \eqref{eq:tv_system} is of the form $V_a=V_Q$ for some $Q\in\AUC_\loc(\timeInt,\posSD[n])$ and is minimal among all storage functions of \eqref{eq:tv_system}, in the sense that $V_a(t_0,x_0)\leq V(t_0,x_0)$ holds for every storage function $V:\timeInt\times\C^n\to\R$, $t_0\in\timeInt$ and $x_0\in\C^n$.
\end{corollary}

\noindent A trivial but very relevant consequence, which justifies our focus on quadratic storage functions, is then the following observation.

\begin{corollary}\label{cor:Pa_to_quadraticStorage}
    Every passive \refPa{} LTV system of the form \eqref{eq:tv_system} has a quadratic storage function $\quadSt{Q}$ for some $Q\in\AUC_\loc(\timeInt,\posSD[n])$.
\end{corollary}

\noindent \Cref{cor:Pa_to_quadraticStorage} has important ramifications. Since every passive \refPa{} LTV system has a quadratic storage function $\quadSt{Q}$ with $Q\in\AUC_\loc(\timeInt,\posSD[n])$, the properties studied in \Cref{sec:null}, most notably the null space decomposition, can be applied.
Furthermore, one can relate the matrix function $Q$ inducing a quadratic storage function to the \refKYP{} inequality and to \refPH{} formulations, while this cannot be done for general storage functions.

\subsection{Relation of passivity and KYP inequalities}
\noindent It is well known that the (algebraic) KYP inequality associated with a passive LTI system always admits a solution $Q\in\posSD[n]$, see, e.g. \cite[Proposition 1]{CheGH23} and \cite[Corollary 3.4]{LozBEM00}. In this subsection, we investigate what can be said in the case of LTV systems, distinguishing between the different regularity properties of $Q$.

If $Q\in\BV_\loc(\timeInt,\posSD[n])$ then we characterize the quadratic storage functions in terms of an \emph{integral KYP inequality} as follows. 

\begin{theorem}\label{thm:KYP_integral}
Let $Q\in\BV_\loc(\timeInt,\posSD[n])$.
    Then $\quadSt{Q}$, as defined in \eqref{eq:quadStorFunc}, is a storage function for \eqref{eq:tv_system} if and only if $Q$ satisfies the \emph{integral KYP inequality}
    \begin{equation}\label{eq:KYP_integral}
        \bmat{ - \displaystyle\int_{t_0}^{t_1}\pset[\big]{\ct{A(t)}Q(t)+Q(t)A(t)}\td t + Q(t_0) - Q(t_1) & \displaystyle\int_{t_0}^{t_1}\pset[\big]{\ct{C(t)} - Q(t)B(t)}\td t \\
        \displaystyle\int_{t_0}^{t_1}\pset[\big]{C(t) - \ct{B(t)}Q(t)}\td t & \displaystyle\int_{t_0}^{t_1}\pset[\big]{D(t)+\ct{D(t)}}\td t} \geq 0
    \end{equation}
for every $t_0,t_1\in\timeInt,\ t_0\leq t_1$.
\end{theorem}
\begin{proof}
Starting from \eqref{eq:tv_system} expressed as
\[
    \bmat{-\dot x \\ y} = \bmat{ -A & -B \\ C & D } \bmat{x \\ u},
\]
we obtain
\[
    - \ct{x}Q\dot x + \ct{u}y = \ct{ \bmat{Qx \\ u } } \bmat{-\dot x \\ y}
    = \ct{ \bmat{x \\ u} } \bmat{ -QA & -QB \\ C & D } \bmat{x \\ u},
    \]
    and therefore
    \[
    -\frac{1}{2}(\ct{\dot x}Qx + \ct{x}Q\dot x) + \realPart(\ct{y}u)
    = \frac{1}{2}\ct{ \bmat{x \\ u} } \bmat{ -QA-\ct{A}Q & \ct{C}-QB \\ C-\ct{B}Q & D+\ct{D} }\bmat{x \\ u},
    \]
    by extracting the real part.
    Furthermore, for every state-input-output solution $(x,u,y)$ of \eqref{eq:tv_system}, exploiting the properties of the Riemann-Stieltjes integral (see Section~\ref{sec:RiemannStieltjes} in the appendix), we have
    \begin{align*}
        & \quadSt{Q}\pset[\big]{t_0,x(t_0)} - \quadSt{Q}\pset[\big]{t_1,x(t_1)} + \int_{t_0}^{t_1}\realPart\pset[\big]{\ct{y(t)}u(t)}\td t  \\
        &\qquad= \frac{1}{2}\pset[\big]{ \ct{x(t_0)}Q(t_0)x(t_0) - \ct{x(t_1)}Q(t_1)x(t_1) } + \int_{t_0}^{t_1}\realPart\pset[\big]{\ct{y(t)}u(t)}\td t  \\
        &\qquad= -\frac{1}{2}\int_{t_0}^{t_1} \pset[\Big]{ \td\pset[\big]{\ct{x}Q(t)x} } + \int_{t_0}^{t_1} \realPart\pset[\big]{\ct{y(t)}u(t)}\td t  \\
        &\qquad= \int_{t_0}^{t_1}\pset*{ -\frac{1}{2}\ct{x}(\td Q)x } + \int_{t_0}^{t_1}\pset*{ -\frac{1}{2}(\ct{\dot x}Qx + \ct{x}Q\dot x) + \realPart(\ct{y}u) }\td t  \\
        &\qquad= \frac{1}{2} \int_{t_0}^{t_1} \ct{\bmat{x \\ u}} \bmat{ -\td Q - (QA+\ct{A}Q)\td t & (\ct{C}-QB)\td t \\ (C-\ct{B}Q)\td t & (D+\ct{D})\td t } \bmat{x \\ u},
    \end{align*}
    for every $t_0,t_1\in\timeInt,\ t_0\leq t_1$.
    In particular, $\quadSt{Q}$ is a storage function if and only if the Riemann-Stieltjes integral inequality
\begin{equation}\label{eq:KYP_integralScalar}
        \int_{t_0}^{t_1} \ct{\bmat{x \\ u}} \bmat{ -\td Q - (QA+\ct{A}Q)\td t & (\ct{C}-QB)\td t \\ (C-\ct{B}Q)\td t & (D+\ct{D})\td t } \bmat{x \\ u} \geq 0
    \end{equation}
holds for every state-input solution $(x,u)$ of \eqref{eq:tv_system} and $t_0,t_1\in\timeInt,\ t_0\leq t_1$.
It is then sufficient to show that \eqref{eq:KYP_integralScalar} is equivalent to \eqref{eq:KYP_integral}.
In fact, if \eqref{eq:KYP_integral} holds, since $Q(t_0)-Q(t_1)=\int_{t_0}^{t_1}\td Q$, it follows immediately from \Cref{thm:RiemannStieltjesIntegralInequality} that \eqref{eq:KYP_integralScalar} holds for all $t_0,t_1\in\timeInt,\ t_0\leq t_1$.%

For the converse implication suppose now that \eqref{eq:KYP_integralScalar} holds for all state-input solutions $(x,u)$ of \eqref{eq:tv_system} and $t_0,t_1\in\timeInt,\ t_0\leq t_1$.
    In particular, this holds for all constant input functions $u(t)\equiv u\in\C^m$.
    For constant inputs, we have that
    \[
    x(t) = X(t)x(t_0) + \pset*{X(t)\int_{t_0}^{t}X(s)^{-1}B(s)\td s}u = X(t)x_0 + Y(t)u
    \]
    for all $t\in[t_0,t_1]$, where $X\in W^{1,1}_\loc(\timeInt,\GL[n])$ is any fundamental solution of \eqref{eq:tv_system} and
    \[
    Y : \timeInt\to\C^{n,m}, \qquad t\mapsto X(t)\int_{t_0}^{t}X(s)^{-1}B(s)\td s.
    \]
    Thus,
    \[
    \bmat{x(t) \\ u(t)} = \bmat{X(t) & Y(t) \\ 0 & I_m}\bmat{x(t) \\ u}
    \]
    for all $t\in[t_0,t_1]$.
    In particular, we can rewrite \eqref{eq:KYP_integralScalar} as
    \[
    \ct{\bmat{x_0 \\ u}} \pset*{\int_{t_0}^{t_1} \ct{\bmat{X & Y \\ 0 & I_m}}\bmat{ -\td Q - (QA+\ct{A}Q)\td t & (\ct{C}-QB)\td t \\ (C-\ct{B}Q)\td t & (D+\ct{D})\td t } \bmat{X & Y \\ 0 & I_m}} \bmat{x_0 \\ u} \geq 0
    \]
    for all $x_0\in\C^n$ and $u\in\C^m$, and therefore
    \[
    \int_{t_0}^{t_1} \ct{\bmat{X & Y \\ 0 & I_m}}\bmat{ -\td Q - (QA+\ct{A}Q)\td t & (\ct{C}-QB)\td t \\ (C-\ct{B}Q)\td t & (D+\ct{D})\td t } \bmat{X & Y \\ 0 & I_m} \geq 0.
    \]
    Note that $Y\in L^\infty_\loc(\timeInt,\C^{n,m})$, since
    \[
    \norm{Y(t)} \leq \norm{X}_{L^\infty}\norm{X^{-1}}_{L^\infty}\int_{t_0}^{t_1}\norm{B(s)}\td s \leq (t_1-t_0)^{\frac{1}{2}}\norm{X}_{L^\infty}\norm{X^{-1}}_{L^\infty}\norm{B}_{L^2}.
    \]
    In particular,
    \[
    \bmat{X & Y \\ 0 & I_m}^{-1} = \bmat{X^{-1} & -X^{-1}Y \\ 0 & I_m} \in L^\infty_\loc(\timeInt,\C^{n+m,n+m})
    \]
    also holds.
    We can then apply \Cref{thm:RiemannStieltjesIntegralInequality} to conclude that \eqref{eq:KYP_integral} holds.
\end{proof}

\noindent Note that \Cref{thm:KYP_integral} only requires $Q$ to be of bounded variation, although we already know from \Cref{thm:storageFunctionAUC} that if $\quadSt{Q}$ is a storage function, then the stronger condition $Q\in\AUC_\loc(\timeInt,\posSD[n])$ holds.
Furthermore, since for $Q$ of bounded variation the classical derivative $\dot Q$ is defined a.e.~on $\timeInt$, one might wonder whether it is really necessary to change from the differential inequality \refKYP{} \eqref{eq:KYP} to the integral KYP inequality \eqref{eq:KYP_integral}.
The following result sheds light on the apparent gap.
\begin{corollary}\label{cor:storageCharacterization}
    Let $Q\in\BV_\loc(\timeInt,\posSD[n])$. If $\quadSt{Q}$ is a storage function for \eqref{eq:tv_system}, then $Q$ is a solution of the differential KYP inequality \eqref{eq:KYP}.
    Conversely, if $Q\in\AUC_\loc(\timeInt,\posSD[n])$ and $Q$ is a solution of the differential KYP inequality \eqref{eq:KYP}, then $\quadSt{Q}$ is a storage function.
   \end{corollary}

\begin{proof}
    Suppose first that $\quadSt{Q}$ is a storage function, in particular $Q$ is a solution of the integral KYP~\eqref{eq:KYP_integral} as in \Cref{thm:KYP_integral}.
    Since $Q\in\BV_\loc(\timeInt,\posSD[n])$, its time derivative $\dot Q(t)$ exists for a.e.~$t\in\timeInt$ in the classical sense, i.e.
    \[
    \timeInt_0 = \set*{ t\in\timeInt \;\;\middle|\;\; \dot Q(t) = \lim_{h\to 0}\frac{Q(t+h)-Q(t)}{h}\text{ exists and is finite} }
    \]
    satisfies $\abs{\timeInt\setminus\timeInt_0}=0$.
    From the integral KYP inequality we deduce that
    \[
    \bmat{ \displaystyle\frac{Q(t) - Q(t+h)}{h} - \frac{1}{h}\int_{t}^{t+h}\pset[\big]{\ct{A(t)}Q(t)+Q(t)A(t)}\td t & \displaystyle\frac{1}{h}\int_{t}^{t+h}\pset[\big]{\ct{C(t)} - Q(t)B(t)}\td t \\
        \displaystyle\frac{1}{h}\int_{t}^{t+h}\pset[\big]{C(t) - \ct{B(t)}Q(t)}\td t & \displaystyle\frac{1}{h}\int_{t}^{t+h}\pset[\big]{D(t)+\ct{D(t)}}\td t} \geq 0
    \]
    holds for all $t\in\timeInt_0$ and $h>0$ such that $t+h\in\timeInt$.
    Going to the limit for $h\to 0$ we then obtain
    \[
    \bmat{ -\dot Q(t) -\ct{A(t)}Q(t)-Q(t)A(t) & \ct{C(t)}-Q(t)B(t) \\ C(t)-\ct{B(t)}Q(t) & D(t)+\ct{D(t)} } \geq 0,
    \] 
    i.e., $Q$ is a solution of the differential KYP inequality. 

    Suppose now that $Q\in\AUC_\loc(\timeInt,\posSD[n])$ is a solution of the differential KYP inequality.
    In particular, for every $t_0,t_1\in\timeInt,\ t_0\leq t_1$ we have
    \[
    Q(t_1) - Q(t_0) \leq \int_{t_0}^{t_1}\dot Q(t)\td t
    \implies Q(t_0) - Q(t_1) \geq -\int_{t_0}^{t_1}\dot Q(t)\td t.
    \]
    We deduce that
    {\small\[
    \bmat{Q(t_0)-Q(t_1)-\int_{t_0}^{t_1}(\ct{A}Q+QA)\td t & \int_{t_0}^{t_1}(\ct{C}-QB)\td t \\ \int_{t_0}^{t_1}(C-\ct{B}Q)\td t & \int_{t_0}^{t_1}(D+\ct{D})}
    \geq \int_{t_0}^{t_1}\bmat{-\dot Q-\ct{A}Q-QA & \ct{C}-QB \\ C-\ct{B}Q & D+\ct{D}}\td t \geq 0,
    \]}%
    and therefore $Q$ is a solution of the integral KYP inequality.
    By \Cref{thm:KYP_integral} then $\quadSt{Q}$ is a storage function.
\end{proof}

\begin{remark}\label{rem:storageCharacterization}
    An important consequence of \Cref{cor:storageCharacterization} is that, for solutions $Q\in\AUC_\loc(\timeInt,\posSD[n])$, the differential KYP inequality \eqref{eq:KYP} and the integral KYP inequality \eqref{eq:KYP_integral} are equivalent.
    However, for more general solutions $Q\in\BV_\loc(\timeInt,\posSD[n])$, the integral KYP inequality is stricter than the differential KYP inequality.
    
    Consider, for example, the stationary system $\dot x=0$, where $A,B,C,D$ are identically zero, and any $Q\in\BV_\loc(\timeInt,\posSD[n])\setminus\AUC_\loc(\timeInt,\posSD[n])$, e.g.~the Cantor function.
    On the one hand, $\dot Q=0$ a.e.~in $\timeInt$, thus the differential KYP inequality is trivially satisfied.
    On the other hand, $\quadSt{Q}$ cannot be a storage function, since $Q\notin\AUC_\loc(\timeInt,\posSD[n])$, thus it cannot be a solution of the integral KYP inequality.
\end{remark}

\begin{remark}\label{rem:only_C}
    Consider an LTV system of the form
    \begin{equation}\label{eq:only_C}
        \dot x(t) = u(t), \qquad y(t) = C(t)x(t)
    \end{equation}
    for some $C\in L^2_\loc(\timeInt,\C^{n,n})$, where the state, input and output variables have the same dimension $n$.
    Since $A=0$, $B=I_n$ and $D=0$, the differential KYP inequality \eqref{eq:KYP} can be written as
    \[
    \bmat{-\dot Q & \ct{C}-Q \\ C-Q & 0} \geq 0.
    \]
    Equivalently, the differential KYP inequality \eqref{eq:KYP} admits a solution $Q\in\AUC_\loc(\timeInt,\posSD[n])$ if and only if $C\in\AUC_\loc(\timeInt,\posSD[n])$ with $\dot C\leq 0$ (i.e., if $C$ is pointwise Hermitian positive semi-definite and weakly decreasing, due to \Cref{lem:BV_decreasing}). In that case, the solution is uniquely determined as $Q=C$.
    In particular, for every weakly decreasing $C:\timeInt\to\posSD[n]$, the matrix function $C$ itself is a solution of the differential KYP inequality \eqref{eq:KYP}, thus the LTV system \eqref{eq:only_C} is passive \refPa{} with storage function $\quadSt{C}$ due to \Cref{cor:storageCharacterization}, and therefore it also has nonnegative supply \refNN{}, because of \Cref{thm:ph2All}.
    If additionally $C\in W^{1,1}_\loc(\timeInt,\posSD[n])$, then $C$ is an absolutely continuous solution of the \refKYP{} inequality. Thus, by \Cref{thm:KYP_to_PH} \eqref{eq:only_C} also admits a \refPH{} representation.
\end{remark}

\noindent \Cref{rem:only_C} is particularly useful to construct examples as follows.
\begin{example}\label{exm:discontinuousQ}
    Consider the LTV system
    \begin{equation}\label{eq:discontinuousQ}
        \dot x(t) = u(t), \qquad y(t) = \begin{cases}
            x(t) & \text{for }t<0, \\ 0 & \text{otherwise}
        \end{cases}
    \end{equation}
    in the time interval $\timeInt=\R$.
    In particular, $y(t)=C(t)x(t)$ with
    \[
    C(t) = \begin{cases}
        1 & \text{for }t<0, \\ 0 & \text{otherwise}.
    \end{cases}
    \]
    Clearly $C:\timeInt\to\posSD[1]$ is weakly decreasing and, therefore, \eqref{eq:discontinuousQ} is passive with storage function $\quadSt{C}$, as discussed in \Cref{rem:only_C}.
    In particular, $Q=C\in\AUC_\loc(\timeInt,\posSD[1])$ is a solution of the differential KYP inequality \eqref{eq:KYP}, although it is discontinuous.
\end{example}

\begin{remark}\label{rem:pH_with_AUC}
    One could consider requesting the weaker condition $Q\in\AUC_\loc(\timeInt,\posSD[n])$ instead of $Q\in W^{1,1}_\loc(\timeInt,\posSD[n])$ when defining \refPH{} systems, that is, in \Cref{def:pH}.
    In fact, in that case we obtain, analogously as in the proof of \Cref{thm:ph2All}, that
    \[
    \bmat{-\dot Q - \ct{A}Q - QA & \ct{C}-QB \\ C-\ct{B}Q & D+\ct{D}} = \bmat{Q & 0 \\ 0 & I_m}W\bmat{Q & 0 \\ 0 & I_m}
    \]
    holds a.e.~on $\timeInt$, i.e., $Q$ is still a solution of the differential KYP inequality \eqref{eq:KYP}.
    Then, due to \Cref{cor:storageCharacterization}, we deduce that $\quadSt{Q}$ is a storage function and \eqref{eq:tv_system} is passive \refPa{}. However, we do not have an immediate way to recover a power balance equation or associate a Dirac structure for the system.
    In fact, the system may be dissipative even when $W=0$, as the following example shows.
\end{remark}

\begin{example}\label{exm:pH_with_AUC}
    Consider the LTV system
    \[
    \dot x(t) = u(t), \qquad y(t) = \CantorReverse(t)x(t),
    \]
    in the time interval $\timeInt=(0,1)$, where $\CantorReverse\in\AUC(\timeInt,\posSD[1])\setminus W^{1,1}_\loc(\timeInt,\posSD[1])$ is the \emph{reverse Cantor function} $\CantorReverse:(0,1)\to[0,1],\ t\mapsto\Cantor(1-t)$, which coincides with its singular part but is weakly monotonically decreasing, in particular $\CantorReverse\in\AUC((0,1))\setminus W^{1,1}((0,1),\R)$. 

    Following the observations in \Cref{rem:pH_with_AUC}, we rewrite the system in the form~\eqref{eq:pH_coefficients} with $J=R=K=S=N=P=0$, $G=1$, and $Q=\CantorReverse$. This system satisfies all required properties, since $\dot Q=0$ a.e.~on $(0,1)$, and thus $Q\in\AUC_\loc(\timeInt,\posSD[1])$ is a solution of the differential KYP inequality \eqref{eq:KYP} and the system is passive.
    Furthermore, for the input signal $u\equiv 0$ the system is stationary, and therefore the dissipation inequality
    \[
    \quadSt{Q}\pset[\big]{t_1,x(t_1)} - \quadSt{Q}\pset[\big]{t_0,x(t_0)} = \frac{1}{2}\pset[\big]{\CantorReverse(t_1)-\CantorReverse(t_0)}x(t_0)^2 < 0
    \]
    holds strictly for all $t_0,t_1\in\timeInt$ such that $\CantorReverse(t_0)>\CantorReverse(t_1)$.
    Thus, the dissipation of this system does not originate from the dissipation operator $W=0$, but from the non-zero singular part of $Q$.
\end{example}

\noindent Note that in \Cref{def:KYP}, we required the solutions of inequality \refKYP{} to be in $W^{1,1}_\loc(\timeInt,\posSD[n])$. This condition is relevant for the transition from \refKYP{} to \refPH{}, as discussed in \Cref{rem:pH_with_AUC}. While \Cref{cor:storageCharacterization} shows that the matrix function $Q$ inducing a quadratic storage function $\quadSt{Q}$ is always a solution of \eqref{eq:KYP}, in general we can only guarantee $Q\in\AUC_\loc(\timeInt,\posSD[n])$, see, e.g.,~\Cref{exm:pH_with_AUC}.
However, since $W^{1,1}_\loc(\timeInt,\posSD[n])\subseteq\AUC_\loc(\timeInt,\posSD[n])$, we have the following result.

\begin{theorem}\label{thm:AC_Pa_to_KYP}
    Suppose that $\quadSt{Q}$ with $Q\in W^{1,1}_\loc(\timeInt,\posSD[n])$ is a storage function for a passive \refPa{} LTV system. Then $Q$ is a solution of the \refKYP{} inequality \eqref{eq:KYP}.
    In particular, the LTV system admits a \refPH{} representation.
\end{theorem}

\noindent In view of \Cref{thm:AC_Pa_to_KYP} it is interesting to derive sufficient conditions for a passive LTV system to have a storage function $\quadSt{Q}$ with $Q\in W^{1,1}_\loc(\timeInt,\posSD[n])$. In \cite{AndM74} it is mentioned that, for $\timeInt=\R$ and under the strong regularity assumptions $A\in\mathcal C^1(\R,\C^{n,n})$, $B\in\mathcal C^1(\R,\C^{n,m})$, $C\in\mathcal C^1(\R,\C^{m,n})$ and $D\in\mathcal C^1(\R,\C^{m,m})$, if the system is completely reachable, satisfies condition $D(t)+\ct{D(t)}>0$ for all $t\in\R$, and has nonnegative supply \refNN{}, then the Riccati differential equation
\begin{equation}\label{eq:RDE}
    \ct{A}Q + QA +\dot Q + (\ct{C}-QB)(D+\ct{D})^{-1}(C-\ct{B}Q) = 0
\end{equation}
has a solution $Q\in\mathcal C^1(\R,\posSD[n])$ with $\lim_{t\to\infty}Q(t)=0$, solving the minimization problem
\[
-\ct{x(t)}Q(t)x(t) = 2\inf_{u\in\mathcal C(\R,\C^m)}\int_{t}^{\infty}\realPart\pset[\big]{\ct{y(s)}u(s)}\td s.
\]
From this it follows that $Q\in\mathcal C^1(\timeInt,\posSD[n])$ induces the available storage function, i.e. $V_a=\quadSt{Q}$. In the following, 
using a different approach, we prove that a similar result holds in our more general setting.

The following lemma allows us to combine two storage functions defined on smaller time subintervals to obtain one defined on a larger time interval.
\begin{lemma}\label{lem:combineStorage}
    Let $\timeInt_1,\timeInt_2\subseteq\timeInt$ be two time subintervals such that $\inf(\timeInt_1)\leq\inf(\timeInt_2)$ and $\sup(\timeInt_1)\leq\sup(\timeInt_2)$, and let $V_1:\timeInt_1\times\C^n\to\R$ and $V_2:\timeInt_2\times\C^n\to\R$ be storage functions for an LTV system \eqref{eq:tv_system}, restricted to the corresponding time subintervals. Suppose that there exists $t_0\in\timeInt_1\cap\timeInt_2$ such that $V_1(t_0,x)=V_2(t_0,x)$ for all $x\in\C^n$. Let $\timeInt_0\coloneqq\timeInt_1\cup\timeInt_2$ and let
    \[
    V : \timeInt_0\times\C^n \to \R, \qquad t \mapsto
    \begin{cases}
        V_1(t) & \text{for }t< t_0, \\
        V_2(t) & \text{for }t\geq t_0.
    \end{cases}
    \]
    Then $V$ is a storage function for \eqref{eq:tv_system} restricted to $\timeInt_0$.
\end{lemma}
\begin{proof}
    It is clear from the assumption that $V(t,x)\geq 0$ and $V(t,0)=0$ hold for all $t\in\timeInt_0$ and $x(t)\in\C^n$.
    Suppose now that $(x,u,y)$ is a state-input-output solution of \eqref{eq:tv_system} on $\timeInt_0$.
    It is clear that the dissipation inequality
    \[
    V\pset[\big]{t_2,x(t_2)} - V\pset[\big]{t_1,x(t_1)} \leq \int_{t_1}^{t_2} \realPart\pset[\big]{ \ct{y(t)}u(t) }\td t
    \]
    holds for all $t_1,t_2\in\timeInt_0$ such that $t_1\leq t_2<t_0$ or $t_0\leq t_1\leq t_2$, since in those cases it reduces to the dissipation inequality for $V_1$ or $V_2$, respectively.
    Suppose now that $t_1<t_0\leq t_2$, then
    \begin{align*}
        & V\pset[\big]{t_2,x(t_2)} - V\pset[\big]{t_1,x(t_1)}
        = V_2\pset[\big]{t_2,x(t_2)} - V_2\pset[\big]{t_0,x(t_0)} + V_1\pset[\big]{t_0,x(t_0)} - V_1\pset[\big]{t_1,x(t_1)} \leq \\
        &\qquad\leq \int_{t_0}^{t_2}\realPart\pset[\big]{ \ct{y(t)}u(t) }\td t + \int_{t_1}^{t_0}\realPart\pset[\big]{ \ct{y(t)}u(t) }\td t
        = \int_{t_1}^{t_2} \realPart\pset[\big]{ \ct{y(t)}u(t) }\td t,
    \end{align*}
    since $V_1(t_0,\cdot)=V_2(t_0,\cdot)$. We then conclude that $V$ is a storage function on $\timeInt_0$.
\end{proof}
\noindent We now proceed to extend the aforementioned result of \cite{AndM74} to our setting.
\begin{theorem}\label{thm:Pa_to_KYP}
    Consider a passive LTV system of the form \eqref{eq:tv_system}
    and suppose that for every $t_0,t_1\in\timeInt,\ t_0\leq t_1$ there exists $c>0$ such that $D(t)+\ct{D(t)}\geq cI_m$ for a.e.~$t\in[t_0,t_1]$. Then the Riccati differential equation \eqref{eq:RDE} has a solution $Q_a\in W^{1,1}_\loc(\timeInt,\posSD[n])$ which induces the available storage function $V_a$, i.e., such that $V_a=\quadSt{Q_a}$.
    In particular, $Q_a$ is a solution of the differential \refKYP{} inequality \eqref{eq:KYP}.
\end{theorem}
\begin{proof}
    Note that the condition on $D+\ct{D}$ implies that $D+\ct{D}\in L^\infty_\loc(\timeInt,\posDef[m])$ and $(D+\ct{D})^{-1}\in L^\infty_\loc(\timeInt,\posDef[m])$.
    Thus, applying the Schur complement, we can write
    {\small\begin{multline*}
        \bmat{-\ct{A}Q - QA - \dot Q - (\ct{C}-QB)(D+\ct{D})^{-1}(C-\ct{B}Q) & 0 \\ 0 & D+\ct{D}} \\
        = \ct{\bmat{I_n & 0 \\ -(D+\ct{D})^{-1}(C-\ct{B}Q) & I_m}}\bmat{-\ct{A}Q-QA-\dot Q & \ct{C}-QB \\ C-\ct{B}Q & D+\ct{D}}\bmat{I_n & 0 \\ -(D+\ct{D})^{-1}(C-\ct{B}Q) & I_m}
    \end{multline*}}%
    a.e.~in $\timeInt$.
    In particular, we deduce from \Cref{cor:storageCharacterization} that $Q\in\AUC_\loc(\timeInt,\posSD[n])$ induces a storage function $\quadSt{Q}$ for \eqref{eq:tv_system} if and only if the Riccati differential inequality
    \begin{equation*}
        \ct{A}Q + QA +\dot Q + (\ct{C}-QB)(D+\ct{D})^{-1}(C-\ct{B}Q) \leq 0
    \end{equation*}
    holds a.e.~on $\timeInt$, or equivalently
    \begin{equation}\label{eq:RDI_simple}
       \ct{\wt A}Q + Q\wt A + \dot Q + Q\wt BQ + \wt C \leq 0,
    \end{equation}
    where
    $\wt A\coloneqq A-B(D+\ct{D})^{-1}C\in L^1_\loc(\timeInt,\C^{n,n})$,
    $\wt B\coloneqq B(D+\ct{D})^{-1}\ct{B}\in L^1_\loc(\timeInt,\posSD[n])$ and
    $\wt C\coloneqq \ct{C}(D+\ct{D})^{-1}C\in L^1_\loc(\timeInt,\posSD[n])$. 
    We consider analogously the Riccati differential equation \eqref{eq:RDE}
    \begin{equation}\label{eq:RDE_simple}
       \ct{\wt A}Q + Q\wt A + \dot Q + Q\wt BQ + \wt C = 0.
    \end{equation}
    Let $t_0\in\timeInt$ be any time point, and let $Q\in W^{1,1}_\loc(\timeInt_1,\HerMat[n])$ be the unique local solution of \eqref{eq:RDE_simple} satisfying $Q(t_0)=Q_a(t_0)$, where $\timeInt_1=(t_{-1},t_1)\subseteq\timeInt$, with $-\infty\leq t_{-1}<t_0<t_1\leq+\infty$, is its maximum interval of definition.
    In particular,
    \[
    Q_a(t_0) - Q(t) = -\int_{t_0}^{t}\dot Q(s)\td s = \int_{t_0}^{t}\pset[\big]{ \ct{\wt A}Q + Q\wt{A} + Q\wt{B}Q + \wt{C} }\td s
    \]
    holds for all $t\in\timeInt_1$.
    Let us rewrite \eqref{eq:RDE_simple} as
    $\dot Q = -\ct{\wt A}Q - Q\wt A - \wh C$
    with $\wh C\coloneqq Q\wt BQ+\wt C\in L^1_\loc(\timeInt,\posSD[n])$.
    Because of \Cref{thm:ODE_linearMatrix}, we can then write
    \[
    Q(t) = \stm_{-\wt A}(t,t_0)Q_a(t_0)\ct{\stm_{-\wt A}(t,t_0)} - \int_{t_0}^{t}\stm_{-\wt A}(t,s)\wh C(s)\ct{\stm_{-\wt A}(t,s)}\td s \geq 0
    \]
    for every $t\in(t_{-1},t_0)$. In particular, we deduce that $\quadSt{Q}$ is a storage function for \eqref{eq:tv_system} on the restricted time interval $(t_{-1},t_0]\subseteq\timeInt$, since $Q$ satisfies \eqref{eq:RDI_simple}.

    Let us now define
    \[
    \wt D : \timeInt_1\to\HerMat[n], \qquad t \mapsto Q_a(t) - Q_a(t_0) + \int_{t_0}^{t} \pset[\big]{\ct{\wt A}Q_a + Q_a\wt A + Q_a\wt BQ_a + \wt C} \td s,
    \]
    in particular $\wt D\in L^\infty_\loc(\timeInt_1,\HerMat[n])$, $\wt D(t_0)=0$, and
    \begin{align*}
        \wt D(s)-\wt D(r) &= Q_a(s) - Q_a(r) + \int_r^s\pset[\big]{\ct{\wt A}Q_a + Q_a\wt A + Q_a\wt BQ_a + \wt C} \td s \\
        &\leq \int_r^s\pset[\big]{\dot Q_a + \ct{\wt A}Q_a + Q_a\wt A + Q_a\wt BQ_a + \wt C} \td s \leq 0,
    \end{align*}
    for all $r,s\in\timeInt_1,\ r\leq s$, since $Q_a\in\AUC_\loc(\timeInt,\posSD[n])$ solves \eqref{eq:RDI_simple}, i.e., $\wt D$ is weakly decreasing.
    By defining $\wt Q\coloneqq Q_a-Q\in\AUC_\loc(\timeInt,\HerMat[n])$ we have that $\wt Q(t_0)=0$ and
    \begin{align*}
        \wt Q(t) &= Q_a(t) - Q(t)
        = \wt D(t) + Q_a(t_0) - Q(t) - \int_{t_0}^{t} \pset[\big]{\ct{\wt A}Q_a + Q_a\wt A + Q_a\wt BQ_a + \wt C} \td s \\
        &= \wt D(t) + \int_{t_0}^{t} \pset[\big]{\ct{\wt A}Q + Q\wt A + Q\wt BQ - \ct{\wt A}Q_a - Q_a\wt A - Q_a\wt BQ_a} \td s \\
        &= \wt D(t) + \int_{t_0}^{t} \pset[\big]{-\ct{\wt A}\wt Q - \wt Q\wt A - \wt Q\wt B\wt Q - \wt Q\wt BQ - Q\wt B\wt Q} \td s \\
        &= \wt Q(t_0) + \int_{t_0}^{t} \pset[\big]{ -\ct{\wh{A}}\wt{Q} - \wt{Q}\wh{A} - \wh{C} } \td s + \wt D(t) - \wt D(t_0),
    \end{align*}
    with $\wh{A}\coloneqq \wt A+\wt B Q\in L^1_\loc(\timeInt_1,\C^{n,n})$ and $\wh{C}\coloneqq\wt{Q}\wt{B}\wt{Q}\in L^1_\loc(\timeInt_1,\posSD[n])$. Because of \Cref{thm:OIE_linearMatrix_RS}, we can then write
    \[
    \wt Q(t) = -\int_{t_0}^{t}\stm_{-\ct{\wh A}}(t,s)\wh C(s)\ct{\stm_{-\ct{\wh A}}(t,s)}\td s + \int_{t_0}^{t}\pset*{\stm_{-\ct{\wh A}}(t,s)\td\wt D(s)\ct{\stm_{-\ct{\wh A}}(t,s)}}
    \]
    for all $t\in\timeInt_1$, where $\stm_{-\ct{\wh A}}\in W^{1,1}_\loc(\timeInt_1,\GL[n])$ denotes the state-transition matrix associated to $\dot x=-\ct{\wh A}x$.
    In particular, for $t\leq t_0$ we have $\wt Q(t)\geq 0$, since $\wh C$ is positive semi-definite a.e.~on $\timeInt_1$ and $\wt D$ is decreasing, see \Cref{lem:RiemannStieltjes_increasing}.
    Thus, $0\leq Q(t)\leq Q_a(t)$ for all $t\in(t_{-1},t_0]$.
    
    Suppose now for the sake of contradiction that $t_{-1}>\inf(\timeInt)$, in particular $t_{-1}\in\R$.
    Since $0\leq Q(t)\leq Q_a(t)$ for all $t\in(t_{-1},t_0]$, and $Q\in\mathcal C(\timeInt_1,\posSD[n])\subseteq W^{1,1}_\loc(\timeInt_1,\posSD[n])$, we deduce that $Q(t_{-1})=\lim_{t\to t_{-1}^+}Q(t)$ exists and satisfies $0\leq Q(t_{-1})\leq\norm{Q_a}_{L^\infty([t_{-1},t_0])}\cdot I_n<\infty$. Then we can extend $Q$ locally from $(t_{-1},Q(t_{-1}))$ and obtain a larger defining interval for $Q$, which contradicts the maximality of $\timeInt_1$.
    We conclude that necessarily $t_{-1}=\inf(\timeInt)$, i.e., that $\timeInt$ and $\timeInt_1$ have the same left end point.

    Let us then define
    \[
    \wh Q(t) : \timeInt \to \C^{n,n}, \qquad t \mapsto
    \begin{cases}
        Q(t) & \text{for }t<t_0, \\
        Q_a(t) & \text{for }t\geq t_0.
    \end{cases}
    \]
    It is clear by construction that $\wh Q\in\AUC_\loc(\timeInt,\posSD[n])$ with $\wh{Q}\leq Q_a$ pointwise.
    Furthermore, since $Q$ and $Q_a$ define storage functions for \eqref{eq:tv_system} restricted on $(t_{-1},t_0]$ and $\timeInt$, respectively, and $Q(t_0)=Q_a(t_0)$, we deduce from \Cref{lem:combineStorage} that $\wh Q$ defines a storage function on $\timeInt$.
    In particular, $\wh Q$ defines a smaller storage function than $Q_a$.
    However, since $Q_a$ is minimal, this can only be the case if $\wh Q=Q_a$.
    Due to \Cref{thm:AC_Pa_to_KYP}, we conclude that $Q_a$ is a solution of the differential KYP inequality \eqref{eq:KYP}.
\end{proof}

\begin{remark}
Note that if we have the condition $D+\ct{D}\geq c I_m$ as in \Cref{thm:Pa_to_KYP}, then we can modify \Cref{fig:overvieweinvertible} to equivalently have the Riccati inequality instead of \refKYP{}.
\end{remark}

\begin{remark}
 There is a close connection between the Riccati differential equation and the matrix differential-algebraic equation
\begin{equation}\label{eq:tpbvp}
\bmat{\dot X \\ -\dot Y \\ 0} = \bmat{0 & A & B \\ \ct{A} & 0 & -\ct{C} \\ \ct{B} & -C & -D-\ct{D}}\bmat{Y \\ X \\ U},
\end{equation}  
with the extra condition $Y=QX$.
The condition that $D+\ct{D}\geq c I_m$, in particular, implies that this differential algebraic equation is of differentiation index one, see \cite{KunM24}, and by taking the Schur complement one gets to the matrix Riccati differential equation 
\[
\pset[\big]{\dot Q + \ct{A}Q + QA + (\ct{C}-QB)(D+\ct{D})^{-1}(C-\ct{B}Q)}X = 0,
\]
which for pointwise invertible $X\in W^{1,1}_\loc(\timeInt,\C^{n,n})$ implies that $Q_a=Q=YX^{-1}$, for which the solution theory is discussed in \cite{AboFIJ12,Rei72}.

Note that many of the properties of \eqref{eq:tpbvp} still hold if $D+\ct{D}$ is indefinite, see \cite{Meh91,ReiV19} but then typically the system is not equivalent to a \refPH{} system; see \cite{ChuM24d_ppt} in the LTI case and~\Cref{thm:nonnegativeimpliesDsemipos} for the LTV case. 
\end{remark}

\section{Nonnegative supply and its implications for pH systems}
\label{sec:popov}

In this section, we study the properties of LTV systems with nonnegative supply \refNN{}, and in particular what this property implies in terms of passivity \refPa{}, solvability of the \refKYP{} inequality, and admitting a \refPH{} representation.

\subsection{Nonnegativity of the Popov operator}\label{sec: subpopov}
For LTI systems, the nonnegativity of the Popov operator is equivalent to the positive realness of the transfer function, see \cite[Proposition 2.36]{LozBEM00}. In this subsection, we define the Popov and transfer operators for LTV systems and study their properties and their relation to each other. These operators can be considered as time-domain counterparts of the Popov function and the transfer function, respectively. 
\begin{definition}\label{def:nonnegativePopov}
    For all $t_0,t_1\in\timeInt,\ t_0\leq t_1$ we call
\begin{equation}
 \label{def:popov_op}
\begin{split}
        &\LAMBDA_{t_0,t_1} : L^2\pset[\big]{[t_0,t_1],\C^m} \to L^2\pset[\big]{[t_0,t_1],\C^m}, \qquad u \mapsto \LAMBDA_{t_0,t_1} u, \\
        &\LAMBDA_{t_0,t_1} u(t) \coloneqq \int_{t_0}^t C(t)\stm(t,s)B(s)u(s)\td s + \int_{t}^{t_1} \ct{B(t)}\ct{\stm(s,t)}\ct{C(s)}u(s)\td s + \pset[\big]{D(t)+\ct{D(t)}}u(t).
\end{split}
\end{equation}
the \emph{Popov operator} of the system~\eqref{eq:tv_system} in $(t_0,t_1)$.
\end{definition}
\noindent We first verify that the Popov operator is well defined. 
\begin{lemma}
    For all $t_0,t_1\in\timeInt,\ t_0\leq t_1$, the Popov operator $\LAMBDA_{t_0,t_1}$ as in \eqref{def:popov_op} is a well-defined bounded linear operator.
\end{lemma}

\begin{proof}
For the proof, we refer to \Cref{lem:popovbounded}.
\end{proof}

\begin{remark}\label{rem:globalPopov}
  Note that $\LAMBDA_{t_0,t_1}$ is also well defined as an operator on $L^2_\loc(\timeInt,\C^m)$.  However, the properties that make this operator of interest only concern the fixed compact interval $[t_0,t_1]\subseteq\timeInt$.
    Furthermore, since $L^2_\loc(\timeInt,\C^m)$ is not a normed vector space, the boundedness of $\LAMBDA_{t_0,t_1}$ is lost in that setting.
    Note that the dependency of $\LAMBDA_{t_0,t_1}$ on $t_0,t_1$ cannot be removed, as these parameters determine the integration intervals. 
\end{remark}

\noindent By interpreting the general state solution in terms of the initial condition $x(t_0)=x_0$ for $(t_0,x_0)\in\timeInt\times\C^n$ and of the state transition matrix using \eqref{eq:state_via_stm}, we write the output trajectory in the form of
\begin{align}
    y(t) &= C(t)\Phi(t,t_0)x_0 + C(t)\int_{t_0}^t\Phi(t,s)B(s)u(s)\td s + D(t)u(t) 
    \nonumber \\
    &= C(t)\Phi(t,t_0)x_0 + \int_{t_0}^t C(t)\Phi(t,s)B(s)u(s)\td s + D(t)u(t)  \label{eq:u_to_y}
\end{align}
for a.e.~$t\in\timeInt$. We express \eqref{eq:u_to_y} as
\begin{equation}\label{eq:transferOp_to_y}
    y(t) = C(t)\stm(t,t_0)x_0+\mathbf Z_{t_0}u(t),   
\end{equation}
where $\mathbf{Z}_{t_0}:L^2_\loc(\timeInt,\C^m) \to L^2_\loc(\timeInt,\C^m)$ is the \emph{transfer operator} at $t_0\in\timeInt$, defined such that
\begin{align}
       \label{def:transfer_op}
\mathbf Z_{t_0}u(t) \coloneqq \int_{t_0}^t C(t)\Phi(t,s)B(s)u(s)\td s + D(t)u(t),
\end{align}
for all $t\in\timeInt$.
The transfer operator is clearly a linear operator on $L^2_\loc(\timeInt,\C^m)$.

Consider now a compact subinterval $[t_0,t_1]\subseteq\timeInt$: On the one hand, it is clear that for every input $u\in L^2_\loc(\timeInt,\C^m)$, the restriction $(\mathbf Z_{t_0}u)|_{[t_0,t_1]}\in L^2([t_0,t_1],\C^m)$ depends only on $u|_{[t_0,t_1]}\in L^2([t_0,t_1],\C^m)$.
On the other hand, for every local input $\wt u\in L^2([t_0,t_1],\C^m)$ we can trivially extend $\wt u$ to
\[
\underline{\wt u} : \timeInt \to \C^m, \qquad t \mapsto \begin{cases}
    \wt u(t) & \text{for }t\in[t_0,t_1], \\
    0 & \text{otherwise},
\end{cases}
\]
which clearly satisfies $\underline{\wt u}\in L^2(\timeInt,\C^m)$.
Therefore, $\mathbf Z_{t_0}$ induces a \emph{local transfer operator} $\mathbf Z_{t_0,t_1}$ on $L^2([t_0,t_1],\C^m)$ that satisfies
\begin{equation}\label{eq:locTransferOp}
    \mathbf Z_{t_0,t_1}\wt u(t) = \int_{t_0}^t C(t)\Phi(t,s)B(s)\wt u(s)\td s + D(t)\wt u(t) = (\mathbf Z_{t_0}\underline{\wt u})|_{[t_0,t_1]}
\end{equation}
for all $\wt u\in L^2([t_0,t_1],\C^m)$ and $t\in[t_0,t_1]$.
In particular, we have the following result.
\begin{lemma}\label{lem:locTransferOp}
    For every $t_0,t_1\in\timeInt,\ t_0\leq t_1$, the local transfer operator $\mathbf Z_{t_0,t_1}$ in \eqref{eq:locTransferOp} is a well-defined and bounded linear operator.
    Furthermore, the formula
    \begin{equation}\label{eq:locTransferOp_to_y}
        y(t) = C(t)\stm(t,t_0)x_0 + (\mathbf Z_{t_0,t_1}u|_{[t_0,t_1]})(t)
    \end{equation}
    holds for all $t\in[t_0,t_1]$ and for every state-input-output solution $(x,u,y)$ of \eqref{eq:tv_system}.
\end{lemma}
\begin{proof} For the proof that $\mathbf Z_{t_0,t_1}$ is well-defined, bounded and linear, we refer to \Cref{lem:transferbounded}.
Furthermore, since
\[
(\mathbf Z_{t_0,t_1}u|_{[t_0,t_1]})(t) = \int_{t_0}^{t_1}C(t)\stm(t,s)B(s)u(s)\td s + D(t)u(t) = \mathbf Z_{t_0}u(t)
\]
holds for all $t\in[t_0,t_1]$, the formula \eqref{eq:locTransferOp_to_y}  follows from \eqref{eq:transferOp_to_y}.
 \end{proof}

\noindent We now analyze the relation between the Popov and the transfer operator, which is analogous to the one between the Popov and the transfer function in the LTI case.
\begin{theorem}\label{thm:PopovFromTransfer}
Consider a LTV system of the form \eqref{eq:tv_system} on $\timeInt$ with the transfer operator $\mathbf{Z}_{t_0,t_1}$ in the Hilbert space $L^2([t_0,t_1],\C^m)$ given by \eqref{def:transfer_op} and Popov operator $\LAMBDA_{t_0,t_1}$ given in \Cref{def:nonnegativePopov}, with $t_0,t_1\in\timeInt,\ t_0\leq t_1$. Then the following statements hold.
\begin{itemize}
    \item[\rm (a)] The adjoint $\ct{\mathbf{Z}}_{t_0,t_1}$ of the transfer operator $\mathbf{Z}_{t_0,t_1}$ \label{glo:transfer_op} is given by 
\[
\ct{\mathbf{Z}}_{t_0,t_1}u(t):=\int_{t}^{t_1} \ct{B(t)}\ct{\Phi(s,t)}\ct{C}(s) u(s)\td s+\ct{D(t)}u(t)\quad \text{for }t_0\leq t\leq t_1.
\]
    \item[\rm (b)] The Popov operator is given by $\mathbf{\Lambda}_{t_0,t_1}=\mathbf{Z}_{t_0,t_1}+\ct{\mathbf{Z}}_{t_0,t_1}$. In particular, $\mathbf{\Lambda}_{t_0,t_1}$ is self-adjoint.
\end{itemize}
Furthermore, 
\begin{equation}\label{eq:supplyWithPopov}
    \int_{t_0}^{t_1}\realPart\pset[\big]{\ct{y(t)}u(t)}\td t = \realPart\pset*{\pset*{\int_{t_0}^{t_1}\ct{u(t)}C(t)\Phi(t,t_0)\td t}x_0} + \frac{1}{2}\aset{\boldsymbol{\Lambda}_{t_0,t_1}u|_{[t_0,t_1]},u|_{[t_0,t_1]}}
\end{equation}
holds for every state-input-output solution $(x,u,y)$ of \eqref{eq:tv_system}.
\end{theorem}
\begin{proof}
Observe that $\mathbf Z_{t_0,t_1}=\mathbf Z_1+\mathbf Z_0$ with 
\[
	\mathbf Z_1 u(t) = \int_{t_0}^{t}C(t)\Phi(t,s)B(s)u(s) \td s, \qquad
	\mathbf Z_0 u(t) = D(t)u(t)
\]
Then 
\begin{align*}
	\aset{\mathbf Z_1u_2,u_1}
	&= \int_{t_0}^{t_1}\int_{t_0}^{t} \ct{u_1(t)}C(t)\Phi(t,s)B(s)u_2(s)\td s\td t  \\
	&= \int_{t_0}^{t_1} \ct{\pset*{ \int_{s}^{t_1} \ct{B(s)}\ct{\Phi(t,s)}\ct{C(t)}u_1(t)\td t }} u_2(s)\td s
    = \aset{u_2,\ct{\mathbf Z}_1 u_1}, \\
	\aset{\mathbf Z_0u_2,u_1}
	&= \int_{t_0}^{t_1} \ct{u_1(t)}D(t)u_2(t)\td t
	= \int_{t_0}^{t_1}\ct{\pset[\big]{\ct{D(t)}u_2(t)}}u_1(t)\td t
	= \aset{u_2,\ct{\mathbf Z}_0 u_1},
\end{align*}
where
\begin{align*}
	\ct{\mathbf Z}_1 u(t) &= \int_{t}^{t_1} \ct{B(t)}\ct{\Phi(s,t)}\ct{C(s)}u(s)\td s, \\
	\ct{\mathbf Z}_0 u(t) &= \ct{D(t)} u(t).
\end{align*}
It follows that 
\begin{equation*}
\LAMBDA_{t_0,t_1}u(t) = (\mathbf Z_{t_0,t_1}+\ct{\mathbf Z}_{t_0,t_1})u(t). 
\end{equation*}
This implies that $\LAMBDA_{t_0,t_1}$ is self-adjoint. Moreover, from \eqref{eq:locTransferOp_to_y} we obtain that
    \begin{align*}
        \int_{t_0}^{t_1}\realPart\pset[\big]{\ct{y(t)}u(t)}\td t &= \realPart\pset*{\int_{t_0}^{t_1}\ct{u(t)}\pset[\big]{C(t)\Phi(t,t_0)x_0 + (\mathbf{Z}_{t_0,t_1}u|_{[t_0,t_1]})(t)}\td t}  \\
        &= \realPart\pset*{\int_{t_0}^{t_1}\pset*{\ct{u(t)}C(t)\Phi(t,t_0)\td t}x_0} + \realPart(\aset{\mathbf{Z}_{t_0,t_1}u|_{[t_0,t_1]},u|_{[t_0,t_1]}})\\
        &= \realPart\pset*{\int_{t_0}^{t_1}\pset*{\ct{u(t)}C(t)\Phi(t,t_0)\td t}x_0} + \frac{1}{2}\aset{\LAMBDA_{t_0,t_1}u|_{[t_0,t_1]},u|_{[t_0,t_1]}}. \qedhere
    \end{align*}
\end{proof}

\noindent Since the first term in \eqref{eq:supplyWithPopov} vanishes for $x_0=0$, by \Cref{thm:PopovFromTransfer}, we immediately have the following result.
\begin{corollary}
\label{thm:PopovSupply}
    The supply of an LTV system \eqref{eq:tv_system} satisfies
    \begin{equation}\label{eq:PopovSupply}
        \int_{t_0}^{t_1}\realPart\pset[\big]{\ct{y(t)}u(t)}\td t = \frac{1}{2}\aset{\LAMBDA_{t_0,t_1}u,u}_{L^2}
    \end{equation}
    for every $t_0,t_1\in\timeInt,\ t_0\leq t_1$ and every state-input-output solution $(x,u,y)$ of \eqref{eq:tv_system} such that $x(t_0)=0$.
    In particular, the system \eqref{eq:tv_system} has nonnegative supply \refNN{} if and only if the Popov operator $\LAMBDA_{t_0,t_1}$ is positive semi-definite for all $t_0,t_1\in\timeInt,\ t_0\leq t_1$. 
\end{corollary}

\noindent \Cref{thm:PopovSupply} introduces a condition equivalent to \refNN{} in terms of the positive semi-definiteness of the Popov operator, which resembles the known definition of positive realness for LTI systems; see, e.g.~\cite{CheGH23}. In \cite{LozBEM00}, $\LAMBDA$ is defined as the integral of a kernel involving the Heaviside step function and the Dirac delta distribution. Similarly, the transfer function $\mathbf Z_{t_0}$ can be equivalently defined as the integral of a kernel, which was introduced in \cite{AndM68} under the name \emph{impulse response matrix} and in \cite{AndM74} as \emph{impedance matrix}.

\begin{remark}\label{rem:ydifferent}
    By only looking at \eqref{eq:PopovSupply}, one might be tempted to claim that $y=\frac{1}{2}\LAMBDA_{t_0,t_1}u$ for $x(t_0)=0$.
    However, this is in general not true, unless quite restrictive assumptions like $C(t)\stm(t,s)B(s)=0$ and $D(t)=\ct{D(t)}$ for all $t,s\in[t_0,t_1]$ hold. In fact, Consider the example
    \begin{equation*}
    \dot x = -x + u , \qquad y = x,
    \end{equation*}
    with $x(t),u(t),y(t)\in\C$ on $\timeInt=\R$, which is clearly \refPH{} with Hamiltonian $\mathcal H(x)=\frac{1}{2}\abs{x}^2$, and observe that $y=\frac{1}{2}\LAMBDA_{t_0,t_1}u$ for $x(t_0)=0$ would imply
    \[
    \int_{t_0}^{t}e^{t-s}u(s)\td s = \frac{1}{2}\int_{t_0}^{t}e^{t-s}u(s)\td s + \frac{1}{2}\int_{t}^{t_1}e^{s-t}u(s)\td s
    \implies \int_{t_0}^{t}e^{t-s}u(s)\td s = \int_{t}^{t_1}e^{s-t}u(s)\td s
    \]
    for every $u\in L^2_\loc(\R)$ and a.e.~$t\in\timeInt$, which is evidently not true.
\end{remark}

\subsection{Relation of nonnegative supply and passivity}

In this subsection, we study the property of having nonnegative supply \refNN{} and its relation to passivity. 

In \Cref{thm:Pa_to_KYP} we have assumed that $D(t)+\ct{D(t)}\geq cI_m$ in every compact subinterval, which in turn implies the weaker condition $D+\ct{D}\in L^\infty_\loc(\timeInt,\posSD[m])$, which is well-known to be necessary for the passivity of the system. The next results show that this condition is also necessary for an LTV system to have a nonnegative supply. 
\begin{theorem}\label{thm:nonnegativeimpliesDsemipos}
 If an LTV system of the from \eqref{eq:tv_system} has nonnegative supply \refNN{}, then $D+\ct{D}\in L^\infty_\loc(\timeInt,\posSD[m])$.
\end{theorem}
\begin{proof}
    Suppose for the sake of contradiction that $D+\ct{D}\notin L^\infty_\loc(\timeInt,\posSD[m])$.
    By \Cref{thm:posDef_equiv} there exist $t_0, t_1 \in \timeInt$ and $t_0\leq t_1$, and $u_0\in\C^{m}$ with $\norm{u_0}_2=1$ satisfying
    \begin{equation*}
        M \coloneqq -\ct{u_0}\pset*{\int^{t_1}_{t_0} \pset[\big]{D(t) + \ct{D}(t)} \td t} u_0> 0.
    \end{equation*} 
By a bisection procedure we construct  a sequence of intervals $\set{{\timeInt}_k\mid k\in\N}$ such that $\timeInt_0 = [t_0, t_1]$, $\timeInt_{k+1} \subseteq \timeInt_k = [t_0^k,t_1^k]$ and $|\timeInt_k | = 2^{-k}|\timeInt_0 | = 2^{-k} (t_1 - t_0)$ with
    \begin{equation*}
        \int_{\timeInt_k} \ct{u} \pset[\big]{D(t) + \ct{D}(t)} u \td t \leq - 2^{-k} M < 0 .   
    \end{equation*}
Let $\varepsilon>0$ be an arbitrarily small number, that we will select later. By \Cref{cor:compactIntervalL1} the inequalities
\[
        \int_{\timeInt_{k}} \norm{B(t)}_2 \td t < 2^{-\frac{k}{2}}\varepsilon,  \quad  \int_{\timeInt_k} \norm{C(t)}_2 \td t < 2^{-\frac{k}{2}}\varepsilon
\]
then hold for sufficiently large $k$, since $B,C\in L^2_\loc$.
Let now $u\equiv u_0\in L^2_\loc(\timeInt,\C^m)$ be a constant input and let $(x,u,y)$ be the state-input-output solution associated with $x(t_0^k)=0$, for such large $k$. Due to \Cref{thm:PopovSupply} it holds that
    \begin{align*}
        0 &\leq 2\int_{\timeInt_k}\realPart\pset[\big]{\ct{y(t)}u(t)}\td t = \aset{\LAMBDA_{t_0^k,t_1^k}u, u }_{L^2}
        = \int_{\timeInt_k} \ct{u(t)} \Bigg( \int_{t_0^k}^t C(t)\stm(t,s)B(s)u(s)\td s \\
        &\qquad\qquad + \int_{t}^{t_1^k} \ct{B(t)}\ct{\stm(s,t)}\ct{C(s)}u(s)\td s + \pset[\big]{D(t)+\ct{D(t)}}u(t) \Bigg) \td t\\
        &\leq \int_{\timeInt_k}\int_{\timeInt_k} \abs*{\ct{u_0} C(t)X(t) X^{-1}(s)B(s)u_0}\td s\td t \\
        &\qquad\qquad + \int_{\timeInt_k}\int_{\timeInt_k}\abs*{\ct{u_0}\ct{B(t)}\ict{X(t)} \ct{X(s)}\ct{C(s)}u_0}\td s \td t
        + \ct{u_0}\pset*{\int_{\timeInt_k} \pset[\big]{D(t)+\ct{D(t)}} \td t} u_0\\
        &\leq \int_{\timeInt_k}\int_{\timeInt_k} \norm{C(t)}\norm{X(t)}\norm{X(s)^{-1}}\norm{B(s)} \td s \td t \\
        &\qquad\qquad + \int_{\timeInt_k}\int_{\timeInt_k} \norm{B(t)}\norm{X(t)^{-1}}\norm{X(s)}\norm{C(s)} \td s \td t
        -2^{-k} M\\
        &\leq 2\norm{X}_{L^\infty}\norm{X^{-1}}_{L^\infty}\int_{\timeInt_k}\norm{C(t)}\td t\int_{\timeInt_k}\norm{B(s)}\td s - 2^{-k}M \\
        &\leq 2^{1-k}\varepsilon\norm{X}_{L^\infty}\norm{X^{-1}}_{L^\infty} - 2^{-k}M
        = 2^{-k}\pset*{ 2\varepsilon\norm{X}_{L^\infty}\norm{X^{-1}}_{L^\infty} - M },
\end{align*}
where the $L^\infty$ norm of $X$ and $X^{-1}$ is taken with respect to the interval $[t_0,t_1]$.
Then, by selecting
\[
    \varepsilon < \frac{M}{2\norm{X}_{L^\infty}\norm{X^{-1}}_{L^\infty}}
\]
and $k$ sufficiently large, we obtain the contradiction $0<0$ which implies that the  $D+\ct{D}\in L^\infty_\loc(\timeInt,\posSD[m])$ is necessary.
\end{proof}

\noindent Note that, by \Cref{thm:ph2All} and \Cref{thm:nonnegativeimpliesDsemipos}, it follows that $D+D\in L^\infty_\loc(\timeInt,\posSD[m])$ is also a necessary condition for \refPH{}, \refKYP{} and \refPa{} to be fulfilled.

In \cite[Theorem 4.105]{LozBEM00}, it was proven that, for completely reachable LTV systems, $\LAMBDA_{t_0,t_1}$ is nonnegative if and only if the system is passive. In \cite[Theorem 4]{HilM80} it was shown for real-valued LTV systems that a system has nonnegative supply if and only if a storage function can be defined in all reachable states. To be consistent with our setting, we present an alternative proof. For this, we need the following two lemmas.

\begin{lemma}\label{lem:availableStorageIneq}
    Let $t_{-1},t_0,t_1\in\timeInt,\ t_{-1}\leq t_0\leq t_1$ and let $(x,u,y)$ be a state-input-output solution of the LTV system \eqref{eq:tv_system} such that $x(t_{-1})=0$.
    Then the available storage function satisfies
    \begin{equation}\label{eq:availableStorageIneq}
        \avSt\pset[\big]{t_1,x(t_1)} \leq \avSt\pset[\big]{t_0,x(t_0)} + \int_{t_0}^{t_1}\realPart  \pset[\big]{ \ct{y(s)}u(s)}\td s.
    \end{equation}
    In particular, if $\avSt(t_0,x(t_0))$ is finite, then $\avSt(t_1,x(t_1))$ is also finite.
\end{lemma}
\begin{proof}
    Let $x_0=x(t_0)$, let $x_1= x(t_1)$, and let us consider the restricted input space
    \[
    \mathcal U = \set{ \wt u\in L^2_\loc(\timeInt,\C^m) \mid \wt u(t)=u(t)\text{ for a.e.~}t\in[t_0,t_1] }.
    \]
    If we denote by $(\wt x,\wt u,\wt y)$ any state-input-output solution with $\wt x(t_0)=x_0$ and $\wt u\in\mathcal U$, then we observe that $(x,u,y)=(\wt x,\wt u,\wt y)$ a.e.~on $[t_0,t_1]$, in particular $(\wt x,\wt y)$ are alternatively determined by $\wt x(t_1)=x_1$.
    Furthermore, for $t>t_1$, the output $\wt y(t)$  depends only on the values of $\wt u(t)$ for $t>t_1$. Then we have that
    \begin{align*}
        V_a\pset[\big]{t_0,x(t_0)} &= \sup_{\avStArg[t_0][t_2][x_0][\wt x][\wt u][\wt y]{}}\pset*{ -\int_{t_0}^{t_2}\realPart(\ct{\wt y(t)}\wt u(t))\td t } \\
        &\geq \sup_{\substack{u\in\mathcal U,\ t_2\geq t_1,\\ \wt x(t_0)=x_0}}\pset*{ -\int_{t_0}^{t_2}\realPart(\ct{\wt y(t)}\wt u(t))\td t } \\
        &= -\int_{t_0}^{t_1}\realPart(\ct{y(t)}u(t))\td t + \sup_{\avStArg[t_1][t_2][x_1][\wt x][\wt u][\wt y]{}}\pset*{ -\int_{t_0}^{t_2}\realPart(\ct{\wt y(t)}\wt u(t))\td t } \\
        &\leq -\int_{t_0}^{t_1}\realPart  \pset[\big]{ \ct{y(s)}u(s)}\td s + \avSt\pset[\big]{t_1,x(t_1)},
    \end{align*}
    which implies \eqref{eq:availableStorageIneq}.
    Suppose now that $V_a(t_0,x(t_0))$ is finite. Since the supply is finite, we immediately deduce from \eqref{eq:availableStorageIneq} that $V_a(t_1,x(t_1))$ is also finite.
\end{proof}

\begin{lemma}\label{lem:NN_implies_zeroAS}
    If the LTV system \eqref{eq:tv_system} has nonnegative supply \refNN{}, then its available storage satisfies $V_a(t,0)=0$ for all $t\in\timeInt$.
\end{lemma}
\begin{proof}
    Since the system has nonnegative supply, we have
    \[
    0 \leq \avSt(t_0,0) = -\inf_{\substack{(t,u)\,:\,t\geq t_0, \\ x(t_0)=0}}\int_{t_0}^{t}\realPart\pset[\big]{\ct{y(s)}u(s)}\td s \leq 0,
    \]
    implying that $\avSt(t_0,0)=0$ for all $t_0\in\timeInt$.
\end{proof}

\noindent We then have the following corollary.
\begin{corollary}\label{cor:NN_to_Pa}
    Consider an LTV system of the form \eqref{eq:tv_system}. If the system has a nonnegative supply \refNN{} and is completely reachable, then the system is passive \refPa{}.
\end{corollary}
\begin{proof}
    Since the system is completely reachable, every $(t_0,x_0)\in\timeInt\times\C^n$ is reachable from some $(t_{-1},0)\in\timeInt\times\C^n$ with $t_{-1}\leq t$. 
    Due to \Cref{lem:NN_implies_zeroAS}, we have $\avSt(t_{-1},0)=0$, thus $\avSt(t_0,x_0)$ is finite, because of \Cref{lem:availableStorageIneq}.
    Then by \Cref{thm:availableStorage} the system is passive.
\end{proof}

\noindent Note that without complete reachability the nonnegativity of the Popov operator is not sufficient for passivity, see \cite[Example 5]{CheGH23} for a linear time-invariant example.

One may wonder whether it was really necessary to restrict the choice of $\timeInt$ to open intervals, since in control one often considers time intervals of the form $[0,T]$ or $[0,+\infty)$. 
While many properties (e.g.~the existence of solutions) still hold when we replace an open interval $\timeInt$ with its closure, some of our results are lost.
The following example presents a passive LTV system defined on an open interval $\timeInt$ that is not passive when extended to $\overline{\timeInt}$.
\begin{example}
    Consider the LTV system
    \[
    \dot x(t) = tu(t), \qquad y(t) = x(t),
    \]
    for $t\in\timeInt=(0,\infty)$ and $x\in W^{1,1}_\loc(\timeInt,\C)$. This system has continuously differentiable coefficients and is \refPH{} with $\mathcal H(t,x)=\frac{1}{2}\ct{x}Q(t)x$ and $Q(t)=\frac{1}{t}$, $Q \in W^{1,1}_\loc(\timeInt,\posDef[1])$. By \Cref{thm:ph2All} we also have properties \refKYP{}, \refPa{} and \refNN{}.
 
Now consider the extended interval $\overline \timeInt =[0, \infty)$.
    Since the coefficients of the original system are continuously differentiable in $\overline\timeInt$, it is clear that the system is well-posed and for every $x_0\in\C$ and any admissible input $u\in L^2(\overline{\timeInt},\C)$, there exists a unique solution $x(t) = x(t_0) + \int_{t_0}^t s u(s) \td s$ such that $x(t_0)=x_0$, $t_0 \in \overline{\timeInt}$.
    Furthermore, the system has nonnegative supply \refNN{}, since 
    \begin{align*}
        &2 \int_{t_0}^{t_1}\realPart\pset[\big]{ \ct{y(t)}u(t) }\td t 
        =2 \int_{t_0}^{t_1}\realPart\pset[\Bigg]{ \ct{\pset[\bigg]{\int_{t_0}^t s u(s) \td s}}u(s) }\td t\\
        &\qquad= \int_{t_0}^{t_1}\int_{t_0}^t s\cc{u(s)}u(t)\td s\td t + \int_{t_0}^{t_1}\int_t^{t_1} t\cc{u(s)}u(t)\td s\td t = \int_{t_0}^{t_1}\int_{t_0}^{t_1}\min(s,t)\cc{u(s)}u(t)\td s\td t \\
        &\qquad= \int_{t_0}^{t_1}\int_{t_0}^{t_1}\int_0^{t_1}\heavi(s-w)\heavi(t-w)\td w\,\cc{u(s)}u(t)\td s\td t \\
        &\qquad= \int_0^{t_1}\cc{\int_{t_0}^{t_1}u(s)\heavi(s-w)\td s}\int_{t_0}^{t_1}u(t)\heavi(t-w)\td t\,\td w = \int_0^{t_1}\abs*{\int_{t_0}^{t_1}u(t)\heavi(t-w)\td t}^2 \geq 0.
\end{align*}
Here we have used that
\begin{equation*}
        \min(s,t) = \int_0^{t_1}\heavi(s-w)\heavi(t-w)\td w, \quad \text{with } \heavi(t) = \begin{cases}
            1 & \text{if } t \geq 0\\ 0 & \text{otherwise.}
        \end{cases} 
\end{equation*}
Moreover, for all $t_0 \in \overline{\timeInt}$  and all initial states $x_0 \in \C$, the input $u(t) = \frac{3t}{t_0-t_{-1}}x_0$ with $x(t_{-1}) = 0$, $t_{-1} \in \overline{\timeInt} $ and $t_{-1} \leq t_0$ drives $(t_{-1},0)$ to $ (t_0,x_0)$. Therefore, this system is completely reachable.
    However, the system is not passive, because the available storage satisfies $\avSt(0,x_0)=+\infty$ for all $x_0\in\C\setminus\{0\}$.
This results from the fact that for every input of the form $u(t)=-\varepsilon^{-\tfrac{3}{2}}x_0\heavi(\varepsilon-t)$ with $\varepsilon\geq 0$ and every $t_1\geq\varepsilon$ it holds that
    \[
    -2\int_{0}^{t_1}\ct{y(t)}u(t)\td t = 2\int_{0}^{\varepsilon}\frac{x_0}{\varepsilon^{\tfrac{3}{2}}}\pset*{\cc x_0-\int_0^t s\frac{\cc x_0}{\varepsilon^{\tfrac{3}{2}}}\td s}\td t = \pset*{\frac{2}{\varepsilon^{\tfrac{1}{2}}} - \frac{1}{3}}\abs{x_0}^2.
    \]
Since for $x_0\neq0$ the supremum for $\varepsilon>0$ of this quantity is $+\infty$, we conclude that $\avSt(0,x_0)=+\infty$ and the system is not passive.
\end{example}

\subsection{The relationship between nonnegative supply and KYP inequality}
Instead of going through \refPa{}, there exists a direct relation between \refNN{} and \refKYP{}.
For $t_0,t_1 \in \timeInt$, $t_0 \leq t_1$ let us consider the map
\[
\Psi_{t_0,t_1} : L^2([t_0, t_1],\C^m) \to W^{1,1}([t_0,t_1],\C^n), \qquad
\Psi_{t_0,t_1} u(t) \coloneqq \int_{t_0}^{t} \Phi(t,s)B(s)u(s)\td s,
\]
which is bounded and linear. By \eqref{eq:state_via_stm} $\Psi$ maps an input signal $u$ to the corresponding solution $x$ with initial condition $x(t_0)=0$, in particular $\dd{}{t}\circ\Psi_{t_0,t_1}=A\circ\Psi_{t_0,t_1}+B$ and $\Psi_{t_0,t_1}(\cdot)(t_0)=0$ hold.
For every $u\in L^2([t_0,t_1],\C^m)$ and $z\in L^2([t_0,t_1],\C^n)$ we have
\begin{align*}
    \aset{\Psi_{t_0,t_1} u,z}_{L^2} &= \int_{t_0}^{t_1} \int_{t_0}^{t} \ct{z(t)}\Phi(t,s)B(s)u(s) \td s\td t\\
    &= \int^{t_1}_{t_0} \int^{t_1}_{s} \ct{\pset*{ \ct{B(s)}\ct{\Phi(t,s)}z(t) }}u(s) \td t\td s\\
    &= \int_{t_0}^{t_1} \ct{\pset*{\ct{B(s)}\int^{t_1}_s \ct{\Phi(t,s)}z(t)\td t}}u(s)\td s.
\end{align*}
Therefore, $\Psi_{t_0,t_1}$ has the adjoint operator
\[
\ct{\Psi}_{t_0,t_1} : L^2([t_0,t_1],\C^n) \to L^2([t_0,t_1],\C^m), \qquad \ct{\Psi}_{t_0,t_1}z(s) \coloneqq \ct{B(s)}\int^{t_1}_{s}\ct{\Phi(t,s)}z(t)\td t
\]
and we have the following lemma.
\begin{lemma}\label{lem:KYP_to_NN}
    Let $t_0,t_1 \in \timeInt$, $t_0 \leq t_1$  and $Q \in W^{1,1}_\loc(\timeInt,\HerMat[n])$. Then it holds that
    \begin{equation}\label{eq:KYP_to_NN}
    \ct{\bmat{\Psi_{t_0,t_1} \\ \mathrm{Id}}}\bmat{-\ct{A}Q-QA-\dot Q & \ct{C}-QB \\ C - \ct{B}Q & D + \ct{D}}\bmat{\Psi_{t_0,t_1} \\ \mathrm{Id}} = \LAMBDA_{t_0,t_1} - \mathbf{R},
    \end{equation}
    where
    \[
    \mathbf R : L^2([t_0,t_1],\C^m) \to L^2([t_0,t_1],\C^m), \qquad \mathbf R u(t) \coloneqq \ct{B(t)}\ct{\Phi(t_1,t)}Q(t_1)\Psi_{t_0,t_1} u(t_1)
    \]
    is a self-adjoint bounded linear operator and $\LAMBDA_{t_0,t_1}$ is as defined in \Cref{def:nonnegativePopov}. If $Q\geq 0$, then also $\mathbf R\geq 0$.
\end{lemma}
\begin{proof}
    The left-hand side of \eqref{eq:KYP_to_NN} can be written as
    \[
    C\Psi_{t_0,t_1} + \ct{\Psi}_{t_0,t_1}\ct{C} + D + \ct{D} - \ct{\Psi}_{t_0,t_1}\ct{A}Q\Psi_{t_0,t_1} - \ct\Psi_{t_0,t_1} QA\Psi_{t_0,t_1} - \ct\Psi_{t_0,t_1}\dot Q\Psi_{t_0,t_1} - \ct{\Psi}_{t_0,t_1}QB - \ct{B}Q\Psi_{t_0,t_1}.
    \]
    Note that
    \begin{align*}
    & \pset[\big]{ C\Psi_{t_0,t_1} + \ct{\Psi}_{t_0,t_1}\ct{C} + D + \ct{D} }u(t) \\
    &= C(t)\int^t_{t_0}\Phi(t,s)B(s)u(s)\td s + \ct{B(t)}\int^{t_1}_t \ct{\Phi(s,t)}\ct{C(s)}u(s)\td s + \pset[\big]{ D(t) + \ct{D(t)} }u(t) \\
    &= \mathbf Z_{t_0,t_1} u(t)+\ct{\mathbf Z}_{t_0,t_1}u(t) r = \LAMBDA_{t_0,t_1}u(t)
    \end{align*}
    holds for all $u\in L^2([t_0,t_1],\C^m)$ and $t\in[t_0,t_1]$, i.e.,
    $
    C\Psi_{t_0,t_1} + \ct{\Psi}_{t_0,t_1}\ct{C} + D + \ct{D} = \LAMBDA_{t_0,t_1}
    $.
    Furthermore, using integration by parts and \eqref{eq:stateTransMatrix}, i.e. $\dd{}{t}\Phi(t,s)=A(t)\Phi(t,s)$ we obtain
    \begin{align*}
        &\ct{\Psi}_{t_0,t_1}\dot Q\Psi_{t_0,t_1} u(s)
        = \ct{B(s)}\int_s^{t_1}\ct{\Phi(t,s)}\dot Q(t)\Psi_{t_0,t_1} u(t)\td t \\
        &\qquad= \ct{B(s)}\pset*{ \ct{\Phi(t_1,s)}Q(t_1)\Psi_{t_0,t_1} u(t_1) - Q(s)\Psi_{t_0,t_1}u(s)} \\
        &\qquad -\ct{B(s)}\pset*{ \int_s^{t_1}\dd{}{t}\ct{\Phi(t,s)}Q(t)\Psi_{t_0,t_1} u(t)\td t - \int_s^{t_1} \ct{\Phi(t,s)}Q(t)\dd{}{t}\Psi_{t_0,t_1} u(t)\td t}\\
        &\qquad= \mathbf R u(s) - \ct{B}Q\Psi_{t_0,t_1}u(s) - \int_s^{t_1}\ct{\Phi(t,s)}\ct{A(t)}Q(t)\Psi_{t_0,t_1} u(t)\td t - \ct{\Psi}_{t_0,t_1}\big(Q(A\Psi_{t_0,t_1} u+Bu)\big)(s) \\
        &\qquad= \mathbf R u(s) - \ct{B}Q\Psi_{t_0,t_1}u(s) - \ct\Psi_{t_0,t_1}\ct{A}Q\Psi_{t_0,t_1}u(s) - \ct\Psi_{t_0,t_1} QA\Psi_{t_0,t_1}u(s) - \ct\Psi_{t_0,t_1} Q B u(s) \\
        &\qquad= (\mathbf R - \ct{B}Q\Psi_{t_0,t_1} - \ct\Psi_{t_0,t_1} \ct{A}Q\Psi_{t_0,t_1} - \ct\Psi_{t_0,t_1} QA\Psi_{t_0,t_1} - \ct\Psi_{t_0,t_1} QB)u(s)
    \end{align*}
    for all $u\in L^2([t_0,t_1],\C^m)$ and $t\in [t_0, t_1]$, i.e.,
    $
    \ct{\Psi}_{t_0,t_1}\ct{A}Q\Psi_{t_0,t_1} + \ct\Psi_{t_0,t_1} QA\Psi_{t_0,t_1} + \ct\Psi_{t_0,t_1}\dot Q\Psi_{t_0,t_1} + \ct{\Psi}_{t_0,t_1} QB + \ct{B}Q\Psi_{t_0,t_1} = \mathbf R.
    $
    It follows that the identity \eqref{eq:KYP_to_NN} holds.
    Note that the operator $\mathbf R$ is self-adjoint, since
    \begin{align*}
        \aset{\mathbf R u,\wt u}_{L^2}
        &= \int^{t_1}_{t_0} \ct{\wt u(t)}\ct{B(t)}\ct{\Phi(t_1,t)}Q(t_1)\Psi_{t_0,t_1} u (t_1)\td t \\
        &= \ct{\pset*{\int^{t_1}_{t_0} \Phi(t_1,t)B(t)\wt u(t)\td t}}Q(t_1)\Psi_{t_0,t_1} u(t_1) = \ct{\Psi_{t_0,t_1}\wt u(t_1)}Q(t_1)\Psi_{t_0,t_1}u(t_1),
    \end{align*}
    and therefore
    \[
    \aset{u,\mathbf R\wt u}_{L^2}
    = \cc{\aset{\mathbf R\wt u,u}_{L^2}}
    = \cc{\ct{\Psi_{t_0,t_1} u(t_1)}Q(t_1)\Psi_{t_0,t_1} \wt u(t_1)}
    = \ct{\Psi_{t_0,t_1} \wt u(t_1)}Q(t_1)\Psi_{t_0,t_1}u(t_1)
    = \aset{\mathbf R u,\wt u}_{L^2},
    \]
    for all $u,\wt u\in L^2([t_0,t_1],\C^m)$, i.e., $\mathbf R=\ct{\mathbf R}$.
    Finally, if $Q\geq 0$, then in particular
    \[
    \aset{\mathbf R u,u}_{L^2} = \ct{\Psi_{t_0,t_1}u(t_1)}Q(t_1)\Psi_{t_0,t_1}u(t_1) = \lVert\Psi_{t_0,t_1} u(t_1)\rVert_{Q(t_1)} \geq 0
    \]
    holds for all $u\in L^2([t_0,t_1],\C^m)$, i.e., $\mathbf R\geq 0$.
\end{proof}
\noindent By \Cref{thm:ph2All}, we already know that \refKYP{} implies \refNN{}. However,  \Cref{lem:KYP_to_NN} provides an explicit relation between the \refKYP{} and the Popov operator $\LAMBDA_{t_0,t_1}$ and we can directly obtain the following result. 
\begin{corollary}
    Suppose that there exists $Q\in W^{1,1}_\loc(\timeInt,\posSD[n])$ that satisfies the KYP inequality \eqref{eq:KYP}. Then the system has a nonnegative supply \refNN{}.
\end{corollary}
\begin{proof}
    Because of \Cref{lem:KYP_to_NN}, we have
    \[
    \LAMBDA_{t_0,t_1} = \ct{\bmat{\Psi_{t_0,t_1} \\ \mathrm{Id}}}\bmat{-\ct{A}Q-Q\ct{A}-\dot Q & \ct{C} -QB\\ C - \ct{B}Q & D + \ct{D}}\bmat{\Psi_{t_0,t_1} \\ \mathrm{Id}} + \mathbf R \geq 0. \qedhere
    \]
    By \Cref{thm:PopovSupply}, if $\LAMBDA_{t_0,t_1}$ is positive semidefinite, the system fulfills \refNN{}.
\end{proof}
\begin{remark}
    In this section, it was shown that complete reachability is a sufficient condition to obtain a \refPH{} realization from a given LTV system \eqref{eq:tv_system} with \refNN{} supply and $D+\ct{D}\geq c I_m$. Thus, in order to obtain a \refPH{} realization, one has to extract a completely reachable realization using the global Kalman canonical decomposition \cite{JikH14}. If the extracted realization (using the same approach) is also completely reconstructable, then $Q$ in \refKYP{} and \refPH{} is also invertible.
\end{remark}
\section{Applications}\label{sec:applications}
This section considers examples of physical systems that naturally lead to LTV~systems. Furthermore, it is shown that these systems can be formulated as time-varying \refPH{} systems, and therefore they are also passive \refPa{} and have nonnegative supply \refNN{}.  

\subsection{Motion of time-varying masses}
Motivated by the so-called ``rocket problems'', where the mass of objects in motion decreases with time due to fuel consumption, we revisit an example from \cite{PlaM92}. 

The equations of motion of the time-varying mass can be described by 
\begin{align}
    \label{eq:newton}
\dot v m+v_e\dot m=F_{\mathrm{ext}}
\end{align}
where $m$ and $v$ are the total mass and velocity of the rocket,
$v_e$ is the effective exhaust velocity of the fuel (often assumed to be constant), and $F_{\mathrm{ext}}$ combines the external forces, e.g.\ gravitational force or drag force. 

Assume that the rocket movement is one-dimensional with $z$ denoting the height, then $v=\dot z$ and we consider the momentum $p=m\dot z$ whose evolution is described by
\[
\dot p=\dot m v+m\dot v=\dot m v-v_e\dot m+F_{\mathrm{ext}}=\tfrac{\dot m}{m} p-v_e\dot m+F_{\mathrm{ext}},
\]
where we used \eqref{eq:newton}.
If we assume that the total mass of the rocket $m$ is controlled independently, weakly decreasing in time but always positive, and by using the trivial equation $\dot z=\tfrac{p}{m}$, we obtain the following LTV system.
\begin{align*}
\begin{bmatrix}
    \dot z(t)\\
    \dot p(t)
\end{bmatrix}
=\begin{bmatrix} 0& \tfrac{1}{m(t)}  \\  0& \tfrac{\dot m(t)}{m(t)}
\end{bmatrix}\begin{bmatrix}
    z(t)\\ p(t)
\end{bmatrix} + \begin{bmatrix}
    0&0\\ -\dot m(t)&1
\end{bmatrix} \begin{bmatrix}
    v_e(t)\\ F_{\mathrm{ext}}(t)
\end{bmatrix}.
\end{align*}
Letting the Hamiltonian represent the kinetic energy leads to
\[
\mathcal H(t,z,p)=\frac{p^2}{2m(t)}=\frac12\begin{bmatrix}
    z\\ p
\end{bmatrix}^\top\begin{bmatrix}0 &0\\ 0& \tfrac{1}{m(t)}
\end{bmatrix}\begin{bmatrix}
    z\\ p
\end{bmatrix},\quad Q(t):=\begin{bmatrix}0 &0\\ 0& \tfrac{1}{m(t)}
\end{bmatrix}.
\]
Using the null space decomposition and \Cref{thm:pH_decoupled}, we obtain
\[
K(t)=\frac12 Q^{-1}(t)\dot Q(t)=\frac12\begin{bmatrix}
  0&0\\0&-\tfrac{\dot m(t)}{m(t)}
\end{bmatrix},
\] 
which leads to the pH formulation
\begin{align*}
\begin{split}
\begin{bmatrix}
    \dot z(t)\\ \dot p(t)
\end{bmatrix}&=\Bigg(\Bigg(\underbrace{\begin{bmatrix}
    0&1\\-1&0
\end{bmatrix}}_{=:J(t)}-\underbrace{\tfrac{1}{2}\begin{bmatrix}
    0&0\\ 0 & -\dot m(t)
\end{bmatrix}}_{=:R(t)}\Bigg)Q(t)-K(t)\Bigg)\begin{bmatrix}
    z(t)\\ p(t)
\end{bmatrix} + \underbrace{\begin{bmatrix}
    0&0\\ -\dot m(t) & 1
\end{bmatrix}}_{=:G(t)}\underbrace{\begin{bmatrix}
    v_e(t)\\ F_{\mathrm{ext}}(t)
\end{bmatrix}}_{=:u(t)},\\ y(t)&=G(t)^\top Q(t)\begin{bmatrix}
    z(t)\\ p(t)
\end{bmatrix}=\begin{bmatrix}0&-\tfrac{\dot m(t)}{m(t)}\\ 0&\tfrac{1}{m(t)}\end{bmatrix}\begin{bmatrix}
    z(t)\\ p(t)
\end{bmatrix}=\begin{bmatrix}-\dot m(t)v(t)\\ v(t)\end{bmatrix},
\end{split}
\end{align*}
where $P$, $S$, and $N$ in \eqref{def:pH} are set to zero and the output $y$ is defined in a collocated way. 
In order to integrate gravitational effects into the model, the external force can be substituted by $F_\mathrm{ext} =-mg$. Alternatively, this force could be incorporated into the Hamiltonian by including an additional term $mgz$. However, this modification makes the Hamiltonian a quadratic polynomial unbounded from below, instead of a positive semidefinite quadratic form, rendering the system non-port-Hamiltonian in the classical linear framework. Despite this, such Hamiltonians can be analyzed within the framework of cyclo-passive systems \cite{HilM80}. 

\subsection{District heating networks with stratified storage models}  
Linear time-varying pH systems appear in the context of district heating systems that contain more detailed (stratified) storage tank models~\cite{MacFCS22,RosGMS24}. In this model, water storage is subdivided into a hot and a cold layer having time-varying volumes $V_h$ and $V_c$ and it is assumed that within each of these volumes, the temperature is homogeneous and given by $T_h$ and $T_c$, respectively. Furthermore, mass flows $q_p(t)>0$ and $q_d(t)>0$ are given for all $t\in\timeInt$ on the production and demand side that satisfy $q_p,q_d\in L^1_{\mathrm{loc}}([0,\infty),\R)$. 
The operation of the system can be described in the following way: Heat is supplied from the hot layer through a mass flow $q_d$ with temperature $T_h$ and there is a return mass flow $q_d$ with temperature $T_{\textrm{in},d}$ into the cold layer. Furthermore, there is a mass flow $q_p$ leaving the cold layer with temperature $T_c$ that is heated by a producer, resulting in a temperature $T_{\textrm{out},p}$ that enters the hot layer of the tank. In summary, we obtain the following set of equations 
\begin{align}
\label{heat1}
\tfrac{\rm d}{{\rm d}t}(V_h(t)T_h(t))&=q_{p}(t)T_{\textrm{in},p}(t)-q_{d}(t)T_h(t),\\ \label{heat2}
\tfrac{\rm d}{{\rm d}t}(V_c(t)T_c(t))&=q_{d}(t)T_{\textrm{in},d}(t)-q_{p}(t)T_c(t),\\ \label{heat3}
\tfrac{\rm d}{{\rm d}t}V_h(t)&=q_p(t)-q_d(t).
\end{align}
Since $q_p$ and $q_d$ are fixed, the volume of the hot layer is given by
\[
V_h(t)=V_{h,0}+\int_0^t\pset[\big]{q_{d}(\tau)-q_{s}(\tau)} \mathrm{d}\tau,\quad  V_c(t)=V_s-V_h(t),\quad \mbox{\rm for all }t\in\timeInt,
\]
where $V_s$ is the total mass of the storage, and we assume that $V_h$ and $V_c$ are uniformly positive. Hence, we can neglect \eqref{heat3} and the remaining system \eqref{heat1} and \eqref{heat2} is of the form \eqref{eq:tv_system} with
\[
A(t)=\begin{bmatrix}
    -\frac{q_d(t)}{V_h(t)}&0\\0& -\frac{q_p(t)}{V_c(t)} 
\end{bmatrix},\quad B(t)=\begin{bmatrix}
    q_p(t)&0\\ 0&q_d(t)
\end{bmatrix},\quad x=\begin{bmatrix}
    V_hT_h\\V_cT_c
\end{bmatrix},\quad  u=\begin{bmatrix}
    T_{\textrm{in},p}\\ T_{\textrm{in},d}
\end{bmatrix}.
\]
The energy of the system is described by the following Hamiltonian 
\[
\mathcal{H}(t,T_c,T_h)=\tfrac12\begin{bmatrix}
    T_h\\ T_c
\end{bmatrix}^\top Q(t)\begin{bmatrix}
    T_h\\ T_c
\end{bmatrix},\quad Q(t):=\begin{bmatrix}\frac{1}{V_h(t)} &0\\ 0&\frac{1}{V_c(t)}
\end{bmatrix}.
\]
Furthermore, $Q$ solves the KYP inequality
\begin{align*}
-\ct{A}Q-Q\ct{A}-\dot Q&=2\begin{bmatrix}
    \frac{q_d}{V_h^2}&0\\0&\frac{q_p}{V_c^2}
\end{bmatrix} +\begin{bmatrix}\frac{\dot V_h}{V_h^2}&0\\0&\frac{\dot V_c}{V_c^{2}}\end{bmatrix}\\&=2\begin{bmatrix}
    \frac{q_d}{V_h^{2}}&0\\0& \frac{q_p}{V_c^{2}}
\end{bmatrix} +\begin{bmatrix}\frac{q_p-q_d}{V_h^2}&0\\0&\frac{q_d-q_p}{V_c^2}\end{bmatrix}
\\ &=\begin{bmatrix}
    \frac{q_d+q_c}{V_h^2}&0\\&\frac{q_d+q_c}{V_c^2}
\end{bmatrix}\geq 0.
\end{align*}
Since $Q$ is pointwise invertible, the LTV pH system is given by setting 
\[
K(t)=\tfrac12 Q^{-1}(t)\dot Q(t)=-\tfrac12\begin{bmatrix}
\frac{q_p(t)-q_d(t)}{V_h(t)}&0\\0&\frac{q_d(t)-q_p(t)}{V_c(t)}
\end{bmatrix}.
\]
Then the pH system representation~\eqref{def:pH} is given by 
\begin{align*}
\begin{bmatrix}
    \dot x_1\\
    \dot x_2
\end{bmatrix}&=\Bigg(-\underbrace{\frac{1}{2}\begin{bmatrix}
   q_d+q_p
   &0\\0&q_d+q_p
\end{bmatrix}}_{=:R(t)}\underbrace{\begin{bmatrix}
    \frac{1}{V_h}&0\\0&\frac{1}{V_c}
\end{bmatrix}}_{=Q}-K\Bigg)\begin{bmatrix}
    x_1\\ x_2
\end{bmatrix}+\underbrace{\begin{bmatrix}
    q_p&0\\0& q_d
\end{bmatrix}}_{=G}\underbrace{\begin{bmatrix}
    T_{\textrm{in},p}\\ T_{\textrm{in},d}
\end{bmatrix}}_{=u}\\
y&=\ct{G}Qx=\begin{bmatrix}
    q_pT_h\\q_dT_c
\end{bmatrix},\quad \begin{bmatrix}
    x_1\\ x_2
\end{bmatrix}=\begin{bmatrix}
    V_hT_h\\V_cT_c
\end{bmatrix},
\end{align*}
where the matrices $P$, $J$, $S$, and $N$ in \eqref{def:pH} are set to zero and the output $y$ is defined in the collocated way.

\section{Conclusion}\label{sec:conclusions}
This paper presents a modified definition of linear time-varying (LTV) port-Hamiltonian systems \refPH{} and studies the relation to passivity \refPa{}, the existence of solutions of the Kalman-Yakubovich-Popov inequality \refKYP{} and the available storage to be finite. In addition, the positive real property for time-invariant systems is generalized to time-varying systems by studying the nonnegativity property \refNN{}. The necessary and sufficient conditions obtained between all these properties are summarized in Figure~\ref{fig:overvieweinvertibleThm}.

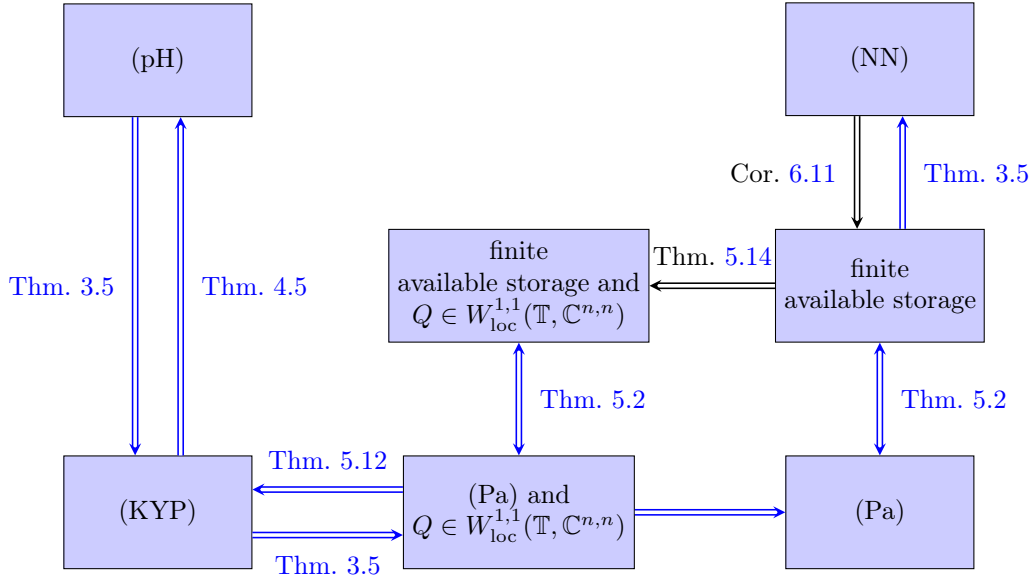
\begin{figure}[htbp!]
    \centering
        \begin{tikzpicture}
        \node[block] (a) {(pH)};
        \node[block, below =4.5cm of a]   (b){(KYP)};
        \node[block, right =2cm of b]   (bc){
        (Pa) and \\$Q\in W^{1,1}_{\mathrm{loc}}(\timeInt,\mathbb C^{n,n})$};
        \node[block, right =2cm of bc]   (c){(Pa)};
        \node[block, above =1.5cm of bc]   (e){finite\\ available storage and \\$Q\in W^{1,1}_{\mathrm{loc}}(\mathbb T,\mathbb C^{n,n})$};
       \node[block, above =1.5 cm of c]   (f){finite\\ available storage};
        \node[block, above =1.5cm of f]   (d){(NN)};
        
        \draw[->,semithick,double,double equal sign distance,>=stealth, color=blue] ([xshift=-2ex]a.south) -- ([xshift=-2ex]b.north) node[midway,left = 1 ex]{ Thm.~\ref{thm:ph2All} };
        \draw[->,semithick,double,double equal sign distance,>=stealth, color=blue] ([xshift=2ex]b.north) -- ([xshift=2ex]a.south) node[midway,right = 1 ex]{Thm.~\ref{thm:KYP_to_PH}};

        \draw[->,semithick,double,double equal sign distance,>=stealth, color=blue] ([yshift=-2ex]b.east) -- ([yshift=-2ex]bc.west) node[midway,below = 1 ex]{Thm.~\ref{thm:ph2All}};
        \draw[->,semithick,double,double equal sign distance,>=stealth, color=blue] ([yshift=2ex]bc.west) -- ([yshift=2ex]b.east) node[midway,above = 1 ex]{Thm.~\ref{thm:AC_Pa_to_KYP}};

        \draw[<->,semithick,double,double equal sign distance,>=stealth, color=blue] ([xshift=0ex]e.south) -- ([xshift=0ex]bc.north) node[midway,right = 1 ex]{Thm.~\ref{thm:availableStorage} };
        
        \draw[->,semithick,double,double equal sign distance,>=stealth, color=blue] ([yshift=0ex]bc.east) -- ([yshift=0ex]c.west) node[midway,above = 1 ex]{};

        \draw[->,semithick,double,double equal sign distance,>=stealth, color=black] ([yshift=0ex]f.west) -- ([yshift=0ex]e.east) node[midway,above = 1 ex]{Thm.~\ref{thm:Pa_to_KYP}};

        \draw[<->,semithick,double,double equal sign distance,>=stealth, color=blue]([xshift=0ex]c.north) -- ([xshift=0ex]f.south) node[midway,right = 1 ex]{Thm.~\ref{thm:availableStorage}};
        
        \draw[->,semithick,double,double equal sign distance,>=stealth, color=blue]([xshift=2ex]f.north) -- ([xshift=2ex]d.south) node[midway,right = 1 ex]{Thm.~\ref{thm:ph2All}};
        \draw[->,semithick,double,double equal sign distance,>=stealth,color=black] ([xshift=-2ex]d.south) -- ([xshift= -2ex]f.north) node[midway,left = 1 ex]{Cor.~\ref{cor:NN_to_Pa} };
          \end{tikzpicture}
    \caption{The relationship between \refPH, \refKYP, \refPa{} and \refNN{} for linear time-varying system \eqref{eq:tv_system}. The blue color marks implications without additional assumptions and the black color sufficient implications with additional assumptions. Counterexamples can be found in \cite{CheGH23}. 
    }
    \label{fig:overvieweinvertibleThm}
\end{figure}  

 Compared to the linear-time invariant (LTI) case, the relation between \refPa{} and \refKYP{} is more involved because \refPa{} does not guarantee the existence of a weakly differentiable solution to \refKYP{}. 
Finally, we presented some examples of systems with time-varying masses that can be formulated as LTV pH systems. The presented linear time-varying pH framework can also be used for structure-preserving local linearization of nonlinear pH systems into linear time-varying pH systems or to formulate switched pH systems. Future work will include the extension to descriptor systems as well as scattering passive systems.

\subsection*{Declaration of competing interest}

The authors declare that they have no known competing financial interests or personal relationships that could have appeared to influence the work reported in this paper. 


\bibliographystyle{abbrv}
\bibliography{sample.bib}

\pagebreak

\appendix

\section{Technical lemmas}

In the paper, we often refer to the following generalization of the Hölder inequality.

\begin{theorem}[Generalized Hölder inequality]\label{thm:genHolder}
    Let $F\in L^p(\timeInt,\C^{l,m})$ and $G\in L^q(\timeInt,\C^{m,n})$, where $1\leq p,q,r\leq\infty$ and $\frac{1}{p}+\frac{1}{q}=\frac{1}{r}$. Then $FG\in L^r(\timeInt,\C^{l,n})$ with
    $\norm{FG}_{L^r} \leq \norm{F}_{L^p}\norm{G}_{L^q}$.
\end{theorem}
\begin{proof}
    By definition we have
    \[
    \norm{FG}_{L^r} = \norm[\big]{\norm{FG}_2}_{L^r} \leq \norm[\big]{\norm{F}_2\norm{G}_2}_{L^r} \leq \norm[\big]{\norm{F}_2}_{L^p}\norm[\big]{\norm{G}_2}_{L^q} = \norm{F}_{L^p}\norm{G}_{L^q},
    \]
    where we applied the better known scalar generalized Hölder inequality to $\norm{F}_2\in L^p(\timeInt,\R)$ and $\norm{G}_2\in L^q(\timeInt,\R)$, see e.g.~\cite[Remark 2 in Chapter 4]{Bre10}
\end{proof}

\begin{lemma}\label{lem:integralIsAC}
    Let $A\in L^1([t_0,t_1],\C^{n,m})$. Then for every $t_0\in\timeInt$ the map
    \[
    \mathbf{A}_{t_0} : \timeInt \to \C, \qquad t\mapsto \int_{t_0}^t A(s)\td s
    \]
    satisfies $\mathbf{A}_{t_0}\in W^{1,1}_\loc(\timeInt,\C^{n,m})$ and $\dot{\mathbf{A}}_{t_0}=A$.
\end{lemma}

\begin{proof}
    Let $A=[a_{ij}]\in L^1_\loc(\timeInt,\R^{2n,2m})$. Then we can equivalently prove the statement for every entry $a_{ij}\in L^1_\loc(\timeInt,\R)$.
    Let $K\subseteq\timeInt$ be any compact subset and let $[a,b]\subseteq\timeInt$ be any compact subinterval such that $K\cup\set{t_0}\subseteq[a,b]$.
    Then, it is sufficient to show that, given $f\in L^1([a,b],\R)$, the function $F:[a,b]\to\R,\ t\mapsto\int_{t_0}^t f(s)\td s$ satisfies $F\in W^{1,1}([a,b],\R)$ and $\dot F=f$.
    Due to \cite[Theorem 20.9 and Lemma 20.14]{Car00}, we see that the map $\wt F:[a,b]\to\R,\ t\mapsto\int_a^tf(s)\td s$ satisfies $\wt F\in W^{1,1}([a,b],\R)$ and $\dot{\wt F}=f$. Since
    \[
    F(t) = \int_a^t f(s)\td s = \int_{t_0}^t f(s)\td s + \int_a^{t_0} f(s)\td s = \wt F(t) + c
    \]
    for all $t\in[a,b]$, with a constant $c\in\R$  depending only on $a,t_0,f$, we deduce that also $F\in W^{1,1}([a,b],\R)$ and $\dot F=\dot{\wt F}=f$.
\end{proof}

\begin{corollary}\label{cor:compactIntervalL1}
    Let $A\in L^p([t_0,t_1],\C^{n,m})$ with $1\leq p<\infty$. Then for every $\varepsilon>0$ there is $\delta>0$ such that
    \[
    \int_r^s \norm{A(t)}_2\td t < (s-r)^{1-\frac{1}{p}} \varepsilon
    \]
    holds for every $r,s\in[t_0,t_1]$ with $0<s-r<\delta$.
\end{corollary}

\begin{proof}
    Let $a:[t_0,t_1]\to\R,\ t\mapsto\norm{A(t)}_2^p$. Then, in particular, $a\in L^1([t_0,t_1],\R)$. Let 
    \[
    f : \timeInt \to \R, \qquad t\mapsto \int_{t_0}^t a(s)\td s, 
    \]
    which is in $W^{1,1}([t_0,t_1],\R)$, because of \Cref{lem:integralIsAC}.
    In particular, $f$ is uniformly continuous, so there exists $\delta>0$ such that
    \[
    \int_{r}^{s}\norm{A(t)}^p\td t = \int_{r}^{s}a(t)\td t = f(s) - f(r) < \varepsilon^p
    \]
    By applying the generalized Hölder inequality (\Cref{thm:genHolder}), we obtain %
    \[
    \int_{r}^{s}\norm{A(t)}\td \leq \pset*{\int_{r}^{s}1\,\td t}^{1-\frac{1}{p}} \pset*{\int_{r}^{s}\norm{A(t)}^p\td t}^{\frac{1}{p}} < (s-r)^{1-\frac{1}{p}} \varepsilon. \qedhere
    \]
\end{proof}

\begin{lemma}\label{lem:young}
    Let $f\in L^p(\timeInt,\R)$ with $1\leq p\leq\infty$ and suppose that for every $t\in\timeInt$ we are given a subset $\timeInt_t\subseteq\timeInt$ such that the quantities
    \[
    a \coloneqq \sup_{t\in\timeInt,\,s\in\timeInt_{t}}\abs{t-s} \qquad\text{and}\qquad
    b \coloneqq \inf_{t\in\timeInt}\abs{\timeInt_{t}}
    \]
    are bounded and positive. Suppose additionally that the map $\timeInt\to\R,\ t\mapsto\abs{\timeInt_{t}}$ is measurable.
    Then the map
    \[
    \wt f : \timeInt \to \R, \qquad t \mapsto \frac{1}{\abs{\timeInt_t}} \int_\R f(s)\td s
    \]
    is in $L^p(\R,\R)$, satisfies $\norm{\wt f}_{L^p}\leq\pset{\frac{2a}{b}}^{1/p}\norm{f}_{L^p}$, and can be written equivalently as
    \begin{equation}\label{eq:convolution}
        \wt f(t) = (\varphi_t \star \overline f)(t) = \int_{\R} \varphi_t(t-s)\overline{f}(s)\td s,
    \end{equation}
    where $\star$ denotes the convolution product. Here, the map $\varphi_t$ is defined as
    \[
    \varphi_t : \R \to \R, \qquad s\mapsto \frac{\chi_{\timeInt_t}(t-s)}{\abs{\timeInt_t}}
    \]
    for all $t\in\timeInt$, and $\overline{f}$ denotes the extension of $f$ to $L^p(\R,\R)$, obtained by setting $\overline{f}(s)=0$ for $s\in\R\setminus\timeInt$.
\end{lemma}

\begin{proof}
    This statement and its proof closely follow Young's theorem (see, e.g.~\cite[Theorem 4.15]{Bre10}), which cannot be applied directly, since $\varphi_t$ is parameterized with $t$.

    By construction, for every $t\in\timeInt$ it holds that $\varphi_t\in L^1\cap L^\infty$ with $\norm{\varphi_t}_{L^1}=\norm{\varphi_t}_{L^\infty}=1$.
    Furthermore, we have that
    \[
    \int_\R \varphi_t(t-s)\overline f(s)\td s
    = \int_\R \frac{\chi_{\timeInt_t}(s)}{\abs{\timeInt_t}}\overline f(s)\td s
    = \frac{1}{\abs{\timeInt_s}}\int_{\timeInt_t}\overline{f}(s)\td s
    \]
    for every $t\in\timeInt$, thus \eqref{eq:convolution} holds.
    
    It remains to show that $\wt f\in L^p$ with the requested bound on its norm.
    Suppose first that $p=\infty$. In this case, it clearly holds that
    \[
    \abs{g(t)} \leq \int_\R\abs{\varphi_t(t-s)}\norm{\overline f}_{L^\infty}\td s = \norm{f}_{L^\infty}
    \]
    for all $t\in\timeInt$, in particular $g\in L^\infty$ and $\norm{g}_{L^\infty}\leq\norm{f}_{L^\infty}$, as requested.

    Suppose now that $1\leq p<\infty$. For a.e.~$s\in\timeInt$ it holds that
    \[
    \int_\R\abs{\varphi_t(t-s)}\abs{\overline f(s)}^p\td t
    \leq \abs{\overline f(s)}^p \int_\R \frac{\chi_{\timeInt_t}(s)}{\abs{\timeInt_t}}\td t
    \leq \abs{\overline f(s)}^p \int_{s-a}^{s+a} \frac{1}{\abs{\timeInt_t}}\td t
    \leq \frac{2a}{b}\abs{\overline f(s)}^p < \infty.
    \]
    In particular, we deduce from Fubini's theorem \cite[Theorem 4.5]{Bre10} that
    \[
    \int_\R\int_\R\abs{\varphi_t(t-s)}\abs{\overline f(s)}^p\td s\td t
    = \int_\R\int_\R\abs{\varphi_t(t-s)}\abs{\overline f(s)}^p\td t\td s
    \leq \frac{2a}{b}\norm{f}_{L^p}^p < \infty.
    \]
    Then, by splitting $\varphi_t(t-s)\abs{\overline f(s)}=\varphi_t(t-s)^{1/p'}\varphi_t(t-s)^{1/p}\abs{\overline f(s)}$ with $\frac{1}{p}+\frac{1}{p'}=1$ and applying Hölder's inequality, we obtain that
    \[
    \abs*{ \int_\R \varphi_{t}(t-s)\abs{\overline f(s)}\td s }^p
    \leq \pset*{\int_\R\varphi_{t}(t-s)\td s}^{\frac{p}{p'}} \pset*{\int_\R\abs{\varphi_t(t-s)}\abs{\overline f(s)}^p\td s}
    \leq \int_\R\abs{\varphi_t(t-s)}\abs{\overline f(s)}^p\td s,
    \]
    and therefore
    \[
    \norm{g}_{L^p}
    = \pset*{\int_\R \abs*{ \int_\R \varphi_{t}(t-s)\abs{\overline f(s)}\td s }^p \td t}^{\frac 1p} 
    \leq \pset*{ \int_\R\int_\R\abs{\varphi_t(t-s)}\abs{\overline f(s)}^p\td s\td t }^{\frac 1p}
    \leq \pset*{\frac{2a}{b}}^{\frac{1}{p}}\norm{f}_{L^p},
    \]
    as asserted.
\end{proof}
\begin{theorem}\label{thm:integralAverage}
    Let $f\in L^p(\timeInt,\R)$ with $1\leq p<\infty$ and suppose that for every $t\in\timeInt$ and $k\in\mathbb N$ we are given a subset $\timeInt_{t,k}\subseteq\timeInt$ such that the quantities
    \[
    a_k \coloneqq \sup_{t\in\timeInt,\,s\in\timeInt_{t,k}}\abs{t-s} \qquad\text{and}\qquad
    b_k \coloneqq \inf_{t\in\timeInt}\abs{\timeInt_{t,k}}
    \]
    are bounded and satisfy $\lim_{k\to\infty}a_k=0$ and $b_k\geq ca_k>0$ for all $k\in\N$, where $c>0$ is a fixed constant.
    Suppose additionally that the map $\timeInt\to\R,\ t\mapsto\abs{\timeInt_{t,k}}$ is measurable for every $k\in\N$.
    Then the sequence of functions $\set{f_k:\timeInt\to\R\mid k\in\N}$ defined as
    \[
    \wt f_k(t) \coloneqq \frac{1}{\abs{\timeInt_{t,k}}}\int_{\timeInt_{t,k}}f(s)\td s
    \]
    satisfies $\wt f_k\in L^p(\timeInt,\R)$ for all $k\in\N$ and converges to $f$ in $L^p$ for $k\to\infty$.
\end{theorem}
\begin{proof}
    Because of \Cref{lem:young}, we can write equivalently $\wt f_k(t)=(\varphi_{t,k}\star\overline f)(t)$ with
    \[
    \varphi_{t,k}(t) = \frac{\chi_{\timeInt_{t,k}}(t-s)}{\abs{\timeInt_{t,k}}^{-1}},
    \]
    and $\wt f_k\in L^p(\timeInt,\R)$ holds for all $k\in\N$.
    
    Since the continuous functions with compact support are dense in $L^p$, for every $\varepsilon>0$ there exists $f_\varepsilon\in\mathcal C_c(\R,\R)$ such that $\norm{f_\varepsilon-\overline f}_{L^p}<\varepsilon$. Let us define now $\wt f_{\varepsilon,k}:\timeInt\to\R,\ t\mapsto(\varphi_{t,k}\star f_\varepsilon)(t)$ for all $k\in\N$.
    Since $\varphi_{t,k}(t-s)f_\varepsilon(s)=0$ for every $s\in\R\setminus\timeInt$, we have equivalently $\wt f_{\varepsilon,k}(t)=(\varphi_{t,k}\star \overline{f_\varepsilon|_{\timeInt}})(t)$, thus we can apply to this function the results of \Cref{lem:young}.
    Clearly
    \[
    \norm{\wt f_k-f}_{L^p(\timeInt)}
    \leq \norm{\wt f_k-\wt f_{\varepsilon,k}}_{L^p(\timeInt)} + \norm{\wt f_{\varepsilon,k}-f_\varepsilon|_{\timeInt}}_{L^p(\timeInt)} + \norm{f_\varepsilon|_{\timeInt}-f}_{L^p(\timeInt)}
    \]
    holds for all $k\in\N$.
    For every $t\in\timeInt$ it holds that
    \[
    \wt f_k(t)-\wt f_{\varepsilon,k}(t)
    = (\varphi_{t,k}\star\overline f)(t) - (\varphi_{t,k}\star\overline{f_\varepsilon|_{\timeInt}})(t)
    = \pset[\big]{\varphi_{t,k}\star\overline{(f-f_\varepsilon|_{\timeInt})}}(t),
    \]
    in particular,
    \[
    \norm{ \wt f_k - \wt f_{\varepsilon,k} }_{L^p(\timeInt)} \leq \pset*{ \frac{2a_k}{b_k} }^{\frac{1}{p}} \norm{f-f_\varepsilon|_{\timeInt}}_{L^p(\timeInt)} \leq c^{\frac{1}{p}}\varepsilon.
    \]
    Since $f_\varepsilon$ is continuous with compact support, it is in particular uniformly continuous, thus for every $\varepsilon_1>0$ there exists $\delta>0$ such that $\abs{f_\varepsilon(s)-f_\varepsilon(t)}<\varepsilon_1$ for every $s,t\in\timeInt$ with $\abs{t-s}<\delta$.
    Since $\lim_{k\to\infty}a_k=0$, there exists $k_0\in\N$ such that $a_k<\delta$ whenever $k\geq k_0$. In particular, for every $t\in\timeInt$ and $s\in\timeInt_{t,k}$ with $k\geq k_0$ it holds that $\abs{t-s}<\delta$.
    Furthermore, if we denote by $K_\varepsilon$ the support of $f_\varepsilon$, by $B_1\coloneqq\set{s\in\R,\ \abs{s}<1}$ the unit ball, and we assume without loss of generality that $\delta<1$, it follows that $f_\varepsilon(s)=f_\varepsilon(t)=0$ for all $s\in\timeInt_{t,k}$ whenever $t\in\timeInt\setminus(K_\varepsilon+B_1)$ and $k\geq k_0$.
    Note that the set $\timeInt_\varepsilon\coloneqq\timeInt\cap (K_\varepsilon+B_1)$ has finite measure and does not depend on the choice of $\varepsilon_1$.
    Therefore, by exploiting the fact that $\int_\R\varphi_{t,k}(t-s)\td s=1$ for every fixed $t\in\timeInt$, we obtain
    \begin{align*}
        &\int_\timeInt \abs*{ \wt f_{\varepsilon,k}(t) - f_\varepsilon|_{\timeInt}(t) }^p \td t
        = \int_\timeInt \abs*{ \int_\R \varphi_{t,k}(t-s)f_{\varepsilon}(s) \td s - \int_\R \varphi_{t,k}(t-s)f_{\varepsilon}(t) \td s }^p \td t \leq \\
        &\qquad\leq \int_\timeInt \pset*{ \int_\R \varphi_{t,k}(t-s) \abs{ f_{\varepsilon}(s) - f_{\varepsilon}(t) } \td s }^p \td t
        = \int_{\timeInt} \pset*{ \frac{1}{\abs{\timeInt_{t,k}}} \int_{\timeInt_{t,k}} \abs{f_\varepsilon(s)-f_\varepsilon(t)} \td s }^p \td t \leq \\
        &\qquad\leq \int_{\timeInt_\varepsilon} \pset*{ \frac{1}{\abs{\timeInt_{t,k}}}\abs{\timeInt_{t,k}}\varepsilon_1 }^p = \varepsilon_1^p\abs{\timeInt_\varepsilon},
    \end{align*}
    and thus $\norm{\wt f_{\varepsilon,k}-f_\varepsilon|_\timeInt}_{L^p(\timeInt)}\leq\varepsilon_1\abs{\timeInt_\varepsilon}^{1/p}$, for every $k\geq k_0$. We deduce that
    \[
    \norm{\wt f_k-f}_{L^p(\timeInt)}
    \leq \pset[\big]{ c^{\frac{1}{p}} + 1 }\varepsilon + \varepsilon_1\abs{\timeInt_\varepsilon}^{\frac 1p}.
    \]
    It is then clear, by taking $\varepsilon$ arbitrarily small and e.g.~$\varepsilon_1<\varepsilon\abs{\timeInt_\varepsilon}^{-\frac{1}{p}}$, that $\wt f_k$ converges to $f$ in $L^p$.
\end{proof}

\begin{lemma}\label{lem:popovbounded}    For all $t_0,t_1\in\timeInt,\ t_0\leq t_1$ the Popov operator $\LAMBDA_{t_0,t_1}$ as in \eqref{def:popov_op} is a well-defined bounded linear operator.
\end{lemma}
\begin{proof}
    Let us split $\LAMBDA_{t_0,t_1}=\LAMBDA_1+\LAMBDA_2+\LAMBDA_3$ such that
    \begin{align*}
    \LAMBDA_1u(t) &= \int_{t_0}^t C(t)\stm(t,s)B(s)u(s)\td s, \\
    \LAMBDA_2u(t) &= \int_{t}^{t_1} \ct{B(t)}\ct{\stm(s,t)}\ct{C(s)}u(s)\td s, \\
    \LAMBDA_3u(t) &= \pset[\big]{D(t)+\ct{D(t)}}u(t).
    \end{align*}
    We will show first that $\LAMBDA_1$, $\LAMBDA_2$ and $\LAMBDA_3$ are well defined, linear, and bounded, so that $\LAMBDA_{t_0,t_1}$ inherits the same properties.
    It is clear that, if they are well defined, then they are linear.
    For every $u\in L^2([t_0,t_1],\C^m)$ it holds that
    \begin{align*}
        & \int_{t_0}^{t_1} \norm*{ \int_{t_0}^t C(t)\stm(t,s)B(s)u(s)\td s }^2 \td t = \\
        &\qquad= \int_{t_0}^{t_1}\int_{t_0}^{t}\int_{t_0}^{t} \ct{u(s)}\ct{B(s)}\ct{\stm(t,s)}\ct{C(t)}C(t)\stm(t,r)B(r)u(r)\td r\td s \td t \leq \\
        &\qquad\leq \int_{t_0}^{t_1}\int_{t_0}^{t}\int_{t_0}^{t} \norm{C(t)}^2\norm{\stm(t,r)}\norm{\stm(t,s)}\norm{B(r)}\norm{B(s)}\norm{u(r)}\norm{u(s)} \td r\td s \td t \leq \\
        &\qquad\leq \norm{\stm}_{L^\infty([t_0,t_1])}^2 \int_{t_0}^{t_1}\norm{C(t)}^2\td t \pset*{\int_{t_0}^{t_1}\norm{B(s)}\norm{u(s)}\td s}^2 \leq \\
        &\qquad\leq \norm{\stm}_{L^\infty([t_0,t_1])}^2 \norm{C}_{L^2([t_0,t_1])}^2 \norm{B}_{L^2([t_0,t_1])}^2 \norm{u}_{L^2([t_0,t_1])}^2,
    \end{align*}
    thus $\norm{\LAMBDA_1u}_{L^2}\leq\norm{\stm}_{L^\infty}\norm{C}_{L^2}\norm{B}_{L^2}\norm{u}_{L^2}$ and therefore $\LAMBDA_1$ is well-defined and bounded. The proof for $\LAMBDA_2$ is analogous, while for $\LAMBDA_3$ it immediately follows from
    $
    \norm{(D+\ct{D})u}_{L^2} \leq 2\norm{D}_{L^2}\norm{u}_{L^\infty}.
    $
\end{proof}

\begin{lemma}\label{lem:transferbounded}.
    For every $t_0,t_1\in\timeInt,\ t_0\leq t_1$ the local transfer operator $\mathbf Z_{t_0,t_1}$ as in \eqref{def:transfer_op} is bounded.
\end{lemma}
\begin{proof} Since $\stm\in L^\infty$, we obtain for every $u\in L^2([t_0,t_1],\C^m)$ that
    \begin{align*}
        \norm{\mathbf Z_{t_0,t_1}u}_{L^2}
        &= \pset*{\int_{t_0}^{t_1}\norm*{\int_{t_0}^t C(t)\stm(t,s)B(s)u(s)\td s + D(t)u(t)}^2\td t}^{\tfrac{1}{2}}\\
        &\leq \norm{\stm}_{L^\infty} \pset*{\int_{t_0}^{t_1}\pset*{\int_{t_0}^{t_1}\norm{C(t)}\norm{B(s)}\norm{u(s)}\td s + \norm{D}_{L^\infty}\norm{u(t)}}^2\td t}^{\tfrac{1}{2}} \\
        &\leq \norm{\stm}_{L^\infty} \pset*{ \int_{t_0}^{t_1}\norm{C(t)}^2\td t\pset*{\int_{t_0}^{t_1}\norm{B(s)}\norm{u(s)}\td s}^2 + \norm{D}_{L^\infty}^2\int_{t_0}^{t_1}\norm{u(t)}^2\td t }^{\tfrac{1}{2}} \\
        &\leq \norm{\stm}_{L^\infty} \pset{\norm{C}_{L^2}^2\norm{B}_{L^2}^2\norm{u}_{L^2}^2 + \norm{D}_{L^\infty}^2\norm{u}_{L^2}^2}^{\tfrac{1}{2}} \\
        &\leq \norm{\stm}_{L^\infty}(\norm{B}_{L^2}\norm{C}_{L^2} + \norm{D}_{L^\infty})\norm{u}_{L^2},
    \end{align*}
    in particular $\norm{\mathbf Z_{t_0,t_1}}\leq\norm{\stm}_{L^\infty}(\norm{B}_{L^2}\norm{C}_{L^2} + \norm{D}_{L^\infty})<\infty$, and hence $\mathbf Z_{t_0,t_1}$ is bounded.
\end{proof}

\begin{lemma}\label{lem:positiveIntegrals}
    Let $f\in L^1_\loc(\timeInt,\R)$ such that $\int_{t_0}^{t_1}f(t)\td t\geq 0$ holds for all $t_0,t_1\in\timeInt,\ t_0\leq t_1$. Then $f(t)\geq 0$ for a.e.~$t\in\timeInt$.
\end{lemma}

\begin{proof}
    Let $\timeInt_N\subseteq\timeInt$ for $N\in\N$ be any sequence of compact intervals such that $\timeInt_N\subseteq\timeInt_{N+1}$ for all $N\in\N$ and $\timeInt=\bigcup_{N\in\N}\timeInt_N$, let $E\coloneqq\set{t\in\timeInt\mid f(t)<0}$ and let $E_N\coloneqq E\cap\timeInt_N$ for all $N\in\N$.

    Fix any $N\in\N$.
    Since $E_N$ is measurable, for every $n\in\N$ there exists an open set $\Omega_n\subseteq\timeInt$ such that $E_N\subseteq\Omega_n$ and $\abs{\Omega_n\setminus E_N}<\frac{1}{n}$. We assume without loss of generality that $\Omega_{n+1}\subseteq\Omega_n$ for all $n\in\N$, up to replacing $\Omega_{n+1}$ with $\Omega_n\cap\Omega_{n+1}$.
    For every $n\in\N$ the set $\Omega_n$ is open, and therefore there are at most countably many disjoint intervals $I_{n,k}$ such that $\Omega_n=\bigcup_k I_{n,k}$. In particular, we have that
    \[
    \int_{\Omega_n}f(t)\td t = \sum_k \int_{I_{n,k}}f(t)\td t \geq 0
    \]
    for all $n\in\N$. It follows that
    \[
    \int_{E_N} f(t)\td t = \int_{\Omega_n} f(t)\td t - \int_{\Omega_n\setminus E_N}f(t)\td t \geq - \int_{\Omega_n\setminus E_N}f(t)\td t,
    \]
    and, therefore,
    \[
    \int_{E_N} f(t)\td t \geq -\lim_{n\to\infty}\int_{\Omega_n\setminus E_N}f(t)\td t = 0,
    \]
    because of the dominated convergence theorem, see e.g. \cite[Theorem 4.2]{Bre10}. Since $f<0$ on $E_N$, we conclude that necessarily $\abs{E_N}=0$.

    Thus, since $E=\bigcup_{N\in\N}E_N$, we deduce that
    \[
    \abs{E} = \abs*{\bigcup_{N\in\N}E_N} \leq \sum_{N\in\N}\abs{E_N} = 0,
    \]
    thus $f\geq 0$ a.e.~on $\timeInt$.
\end{proof}

\begin{lemma}\label{lem:nullIntegrals}
    Let $A\in L^1_\loc(\timeInt,\C^{n,m})$ such that $\int_{t_0}^{t_1}A(t)\td t=0$ for all $t_0,t_1\in\timeInt,\ t_0\leq t_1$. Then $A(t)=0$ for a.e.~$t\in\timeInt$.
\end{lemma}
\begin{proof}
    If $A=[a_{ij}]\in L^1_\loc(\timeInt,\R^{2n,2m})$, then $a_{ij}\in L^1_\loc(\timeInt,\R)$ satisfies $\int_{t_0}^{t_1}a_{ij}(t)\td t=0$ for all $t_0,t_1\in\timeInt,\ t_0\leq t_1$ and all $i,j$. In particular, applying \Cref{lem:positiveIntegrals} to $a_{ij}$ and $-a_{ij}$, we obtain $a_{ij}(t)=0$ for a.e.~$t\in\timeInt$, for all $i,j$.
    We conclude that also $A(t)=0$ for a.e.~$t\in\timeInt$.
\end{proof}

\begin{lemma}\label{lem:Cholesky}
    Let $Q\in W^{1,p}_\loc(\timeInt,\posDef[n])$ for some $p\in[1,\infty]$.
    Then there exists a lower triangular matrix function $L\in W^{1,p}_\loc(\timeInt,\GL[n])$ with real and positive diagonal entries such that $Q(t)=\ct{L(t)}L(t)$ for all $t\in\timeInt$.
\end{lemma}
\begin{proof}
    We prove the statement by induction on the dimension of $Q$, exploiting the explicit construction of the Cholesky decomposition.
    Write
    \[
    Q = \bmat{\alpha & b \\ \ct{b} & Q_1}
    \]
    with $\alpha\in W^{1,p}_\loc(\timeInt,\posR)$, $b\in W^{1,p}_\loc(\timeInt,\mathbb C^{n-1})$ and $Q_1\in W^{1,p}_\loc(\timeInt,\posDef[n-1])$.
    The pointwise Cholesky decomposition of $Q$ is then given by $\ct{L}L$ with
    \[
    L = \bmat{\alpha^{\tfrac{1}{2}} & \alpha^{-\tfrac{1}{2}}b \\ 0 & L_1},
    \]
    where $L_1$ is constructed recursively as the Cholesky factor of $Q_1+\alpha^{-1}\ct{b}b$.
    Let us denote by $f,g\in\mathcal C^\infty(\posR,\posR)$ the functions
    \[
    f:\posR\to\posR,\ s\mapsto s^{\tfrac{1}{2}} \qquad\text{and}\qquad g:\posR\to\posR,\ s\mapsto s^{-1}.
    \]
Then, the functions $\alpha^{\tfrac{1}{2}}=f\circ\alpha$ and $\alpha^{-\tfrac{1}{2}}=g\circ f\circ\alpha$ are in $W^{1,p}_\loc$, and so is $\alpha^{-\tfrac{1}{2}}b$.
    If $n=1$, then only $\alpha^{\tfrac{1}{2}}$ appears and we conclude that $L$ is in $W^{1,p}_\loc$.
    Otherwise, $L_1$ is also in $W^{1,p}_\loc$ by the induction hypothesis, so we have $L$ is in $W^{1,p}_\loc$.
\end{proof}

\section{Solution theory}
In this section, we recall conditions for the existence and uniqueness of solutions of (linear) time-varying equations, and the properties of the associated fundamental solution and state-transition matrices.
We start by recalling that the conditions that we assumed for the coefficient matrices in \eqref{eq:tv_system} guarantee the existence of global solutions and their uniqueness up to fixing one point.

\begin{theorem}\label{thm:solution_L1}
    Let $A\in L^1_\loc(\timeInt,\C^{n,n})$ and $b\in L^1_\loc(\timeInt,\C^n)$. Then for every $(t_0,x_0)\in\timeInt\times\C^n$ the ordinary differential equation
 \begin{equation}\label{eq:solution_L1}
        \dot x = A(t)x + b(t)
    \end{equation}
    has exactly one solution $x\in W^{1,1}_\loc(\timeInt,\C^n)$ such that $x(t_0)=x_0$.
\end{theorem}

\begin{proof}
    The proof is exactly that of \cite[Theorem 3]{Fil88}, up to identifying $\C$ with $\R^2$.
\end{proof}

\begin{corollary}\label{cor:solution_tv}
    For every initial condition $(t_0,x_0)\in\timeInt\times\C^n$ and input $u\in L^2_\loc(\timeInt,\C^m)$, the LTV system \eqref{eq:tv_system} has exactly one solution $x\in W^{1,1}_\loc(\timeInt,\C^n)$ that satisfies $x(t_0)=x_0$. Furthermore, the corresponding output $y$ is an element of $L^2_\loc(\timeInt,\C^m)$.
\end{corollary}
\begin{proof}
    By H\"older's inequality, we have $b\coloneqq Bu\in L^1_\loc(\timeInt,\C^m)$, and therefore we can apply \Cref{thm:solution_L1} to obtain a unique solution $x\in W^{1,p}_\loc(\timeInt,\C^n)$.
    Applying then the generalized H\"older inequality (\Cref{thm:genHolder}) to the output equation, we get $y=Cx+Du\in L^2_\loc(\timeInt,\C^m)$.
\end{proof}

\noindent Next we  study the  regularity properties of the fundamental solution matrix associated to the homogeneous differential equation $\dot x=A(t)x$.
For that, we first need the following lemma.
\begin{lemma}\label{lem:inverseContinuous}
    Let $X:\timeInt\to\GL[n]$ and $X^{-1}:\timeInt\to\C^{n,n},\ t\mapsto X(t)^{-1}$. Then the following statements hold:
    \begin{enumerate}
  \item\label{it:inverseContinuous:1} If $X\in\mathcal C(\timeInt,\GL[n])$, then  $X^{-1}\in\mathcal C(\timeInt,\GL[n])$.
        \item If $X\in W^{1,1}_\loc(\timeInt,\GL[n])$, then  $X^{-1}\in W^{1,1}_\loc(\timeInt,\GL[n])$.
    \end{enumerate}
\end{lemma}
\begin{proof}
    Using that the inverse is $X^{-1}=\det(X)^{-1}\operatorname{adj}(X)$, where for all $t\in\timeInt$, $\operatorname{adj}(X(t))\in\GL[n]$ denotes the adjugate matrix of $X(t)$.
    \begin{enumerate}
        \item Suppose that $X$ is continuous. Then $\det(X)$ and $\operatorname{adj}(X)$ are also continuous. Since $\det(X(t))\neq 0$ for all $t\in\timeInt$, we have that also $\det(X(t))^{-1}$ is continuous and thus, $X^{-1}=\det(X)^{-1}\operatorname{adj}(X)\in\mathcal C(\timeInt,\GL[n])$.
        \item Due to \ref{it:inverseContinuous:1}., we have that $X^{-1}\in\mathcal C\subseteq L^\infty_\loc\subseteq L^1_\loc$.
        From matrix differential calculus, we know that $\dd{}{t}(X^{-1})=X^{-1}\dot XX^{-1}$ is the weak derivative of $X^{-1}$.
        By the generalized Hölder inequality (\Cref{thm:genHolder}) it follows that $\dd{}{t}(X^{-1})\in L^1_\loc(\timeInt,\GL[n])$, and therefore $X^{-1}\in W^{1,1}_\loc(\timeInt,\GL[n])$.
        \qedhere
    \end{enumerate}
\end{proof}

\noindent We proceed by studying the properties of the fundamental solution matrix.
\begin{theorem}\label{thm:fundamentalSolution}
    Let $A\in L^1_\loc(\timeInt,\C^{n,n})$. Then the following statements hold:
    \begin{enumerate}
        \item For every $t_0\in\timeInt$ the homogeneous ordinary matrix differential equation $\dot X(t)=A(t)X(t)$ has exactly one solution $X\in W^{1,1}_\loc(\timeInt,\C^{n,n})$ such that $X(t_0)=I_n$.
\item\label{it:fundamentalSolution:2} For every $(t_0,x_0)\in\timeInt\times\C^n$ the unique solution of the homogeneous differential equation $\dot x=A(t)x$ that satisfies the initial condition $x(t_0)=x_0$ can be expressed as $x(t)=X(t)x_0$ for all $t\in\timeInt$.
\item $X(t)$ is invertible for all $t\in\timeInt$. In particular $X^{-1}\in W^{1,1}_\loc(\timeInt,\C^{n,n})$.
\item For every $(t_0,x_0)\in\timeInt\times\C^n$ the unique solution $x\in W^{1,1}_\loc(\timeInt,\C^n)$ of the inhomogeneous differential equation \eqref{eq:solution_L1} that satisfies the initial condition $x(t_0)=x_0$, can be expressed as
\begin{equation}\label{eq:solution_with_fundSol}
            x(t) = X(t)\pset*{ x_0 + \int_{t_0}^t X^{-1}(s)b(s)\td s }
        \end{equation}
        for all $t\in\timeInt$.
    \end{enumerate}
\end{theorem}
\begin{proof}
\begin{enumerate}
\item The existence and uniqueness of the solution $X\in W^{1,1}_\loc(\timeInt,\C^{n,n})$ follows immediately from \Cref{thm:solution_L1}, reinterpreting $\dot X=AX$ as $\dd{}{t}\mathrm{vec}(X)=(I_n\otimes A(t))\mathrm{vec}(X)$ under vectorization, where $\otimes$ represents the Kronecker product, and applying \Cref{thm:solution_L1}.
\item Let $x(t)\coloneqq X(t)x_0$. Then $\dot x(t)=\dot X(t)x_0=A(t)X(t)x_0=A(t)x(t)$ and $x(t_0)=A(t_0)x_0=x_0$.
\item Suppose for the sake of contradiction that there exists $t_1\in\timeInt$ for which $X(t_1)$ is singular, i.e. there is $x_1\in\C^n\setminus\{0\}$ such that $X(t_1)x_1=0$.
Because of \ref{it:fundamentalSolution:2}., both $x(t)\coloneqq X(t)x_1$ and $\wt x(t)\equiv 0$ are solutions of $\dot x(t)=A(t)x(t)$ satisfying $x(t_1)=\wt x(t_1)=0$.
It follows that $x=\wt x$, and therefore $x_1=X(t_0)x_1=x(t_0)=\wt x(t_0)=0$, in contradiction with $x_1\neq 0$.
 Therefore,  $X(t)$ is invertible for all $t\in\timeInt$, and by \Cref{lem:inverseContinuous} it follows that $X^{-1}\in W^{1,1}_\loc(\timeInt,\GL[n])$.
\item Define
\[
  f:\timeInt\to\C^{n},\qquad t\mapsto\int_{t_0}^{t}X^{-1}(s)b(s)\td s
\]
and $\wt x\coloneqq x-Xf$.
By the generalized Hölder inequality (\Cref{thm:genHolder}), $X^{-1}b\in L^1_\loc(\timeInt,C^n)$. Thus, $f\in W^{1,1}_\loc(\timeInt,\C^n)$ with $\dot{f}=X^{-1}b$, by \Cref{lem:integralIsAC}.
Since $W^{1,1}_\loc(\timeInt,\C)$ is an algebra, we get from $X\in W^{1,1}_\loc(\timeInt,\C^n)$ that $\wt x=x-Xf\in W^{1,1}_\loc(\timeInt,\C^n)$.
Furthermore, we obtain that
\[
    \dot{\wt x} = \dot x - \dot Xf - X\dot f = Ax + b - AXf - XX^{-1}b = A(x-Xf) = A\wt x
\]
and $\wt x(t_0)=x_0$. Thus $\wt x(t)=X(t)x_0$ for all $t\in\timeInt$, because of \eqref{it:fundamentalSolution:2}.
We conclude that
\[
        x(t) = \wt x(t) + X(t)f(t) = X(t)x_0 + X(t)\int_{t_0}^{t}X^{-1}(s)b(s)\td s = X(t)\pset*{ x_0 + \int_{t_0}^{t}X^{-1}(s)b(s)\td s}
\]
holds for all $t\in\timeInt$. \qedhere
    \end{enumerate}
\end{proof}

\noindent Many of the properties stated in \Cref{thm:fundamentalSolution} also apply to the state-transition matrix. We first need the following lemma.
\begin{lemma}\label{lem:Fubini_generalized}
    Let $\timeInt_1,\timeInt_2\subseteq\R$ be measurable sets and $u\in W^{1,1}(\timeInt_1\times\timeInt_2,\R)$ (resp.~$u\in W_\loc^{1,1}(\timeInt_1\times\timeInt_2,\R)$). Then $u_s\coloneqq u(\cdot,s)\in W^{1,1}(\timeInt_1,\R)$ (resp.~$u_s\in W^{1,1}_\loc(\timeInt_1,\R)$) with $\dot u_s=\pd{u}{t}(\cdot,s)$ for a.e.~$s\in\timeInt_2$.
\end{lemma}

\begin{proof}
    Assume first that $u\in W^{1,1}(\timeInt_1\times\timeInt_2,\R)$.
    By applying Fubini's theorem \cite{Bre10} to $u\in L^1$ and $\pd{u}{t}\in L^1$, we obtain that $u_s\in L^1$ and $\dot u_s\coloneqq\pd{u}{t}(\cdot,s)\in L^1$ for a.e.~$s\in\timeInt_2$.
    For every $\varphi_1\in\mathcal C_c^1(\timeInt_1,\R)$ and $\varphi_2\in\mathcal C_c^1(\timeInt_2,\R)$ we have then
    \begin{align*}
        &\int_{\timeInt_2}\varphi_2(s)\int_{\timeInt_1}\pd{u}{t}(t,s)\varphi_1(t)\td t\td s
        = \iint_{\timeInt_1\times\timeInt_2}\pd{u}{t}(t,s)\varphi_1(t)\varphi_2(s)\td t\td s \\
        &\qquad= -\iint_{\timeInt_1\times\timeInt_2}u(t,s)\dot\varphi_1(t)\varphi_2(s)\td t\td s
        = -\int_{\timeInt_2}\varphi_2(s)\int_{\timeInt_1}u_s(t,s)\dot\varphi_1(t)\td t\td s,
    \end{align*}
    thus
    \[
    \int_{\timeInt_1}\pd{u}{t}(t,s)\varphi_1(t)\td t = -\int_{\timeInt_1}u_s(t,s)\dot\varphi_1(t)\td t,
    \]
    for a.e.~$s\in\timeInt_2$, and therefore $u_s\in W^{1,1}(\timeInt_1,\R)$ with $\dot u_s=\pd{u}{t}(\cdot,s)$.

    Suppose now that $u\in W^{1,1}_\loc(\timeInt_1,\timeInt_2)$ and let $K_1\subseteq\timeInt_1$ be a compact set. 
    Since $u_s|_{K_1}=u|_{K_1\times\set{s}}(\cdot,s)$ for all $s\in\timeInt_2$, and $u_s\in W^{1,1}(K_1\times\set{s},\R)$, we conclude from the first part of the proof that $u_s|_{K_1}\in W^{1,1}(K_1,\R)$ with $\dot u_s|_{K_1}=\pd{u}{t}(\cdot,s)|_{K_1}$ for a.e.~$s\in\timeInt_2$.
\end{proof}

\begin{theorem}\label{thm:stateTransitionMatrix}
    Let $A\in L^1_\loc(\timeInt,\C^{n,n})$. Then the partial differential equation
    \begin{equation}\label{eq:stm}
    \pd{}{t}\stm(t,s) = A(t)\stm(t,s)\qquad \text{for }t,s\in\timeInt
    \end{equation}
    has exactly one solution $\stm\in W^{1,1}_\loc(\timeInt\times\timeInt,\C^{n,n})$ such that $\stm(s,s)=I_n$ for all $s\in\timeInt$.
    Furthermore, $\stm$ satisfies the following properties: 
    \begin{enumerate}
        \item It holds that $\stm(t,s)=X(t)X(s)^{-1}$ for a.e.~$t,s\in\timeInt$, where $X\in W^{1,1}_\loc(\timeInt,\GL[n])$ is any fundamental solution matrix of $\dot X(t)=A(t)X(t)$. In particular, $\stm\in \mathcal C(\timeInt\times\timeInt,\GL[n])$.
        \item For every $(t_0,x_0)\in\timeInt\times\C^n$, the unique solution of the homogeneous differential equation $\dot x=A(t)x(t)$ satisfying the initial condition $x(t_0)=x_0$, can be expressed as $x(t)=\stm(t,t_0)x_0$ for all $t\in\timeInt$.
        \item For every $(t_0,x_0)\in\timeInt\times\C^n$ the unique solution of the inhomogeneous differential equation \eqref{eq:solution_L1} satisfying the initial condition  $x(t_0)=x_0$ can be expressed as
        \begin{equation*}
            x(t) = \stm(t,t_0)x_0 + \int_{t_0}^t\stm(t,s)b(s)\td s
        \end{equation*}
        for all $t\in\timeInt$.
        \item For every $t,s\in\timeInt$ the matrix $\stm(t,s)\in\C^{n,n}$ is invertible and $\stm(t,s)^{-1}=\stm(s,t)$. In particular, $\stm^{-1}:\timeInt\times\timeInt\to\C^{n,n}$ is also a continuous $W^{1,1}$ matrix function.
        \item $\pd{}{s}\stm(t,s)=\stm(t,s)A(s)$ holds for all $t,s\in\timeInt$.
    \end{enumerate}
\end{theorem}
\begin{proof}
    We start by showing that, for every fundamental solution matrix $X\in W^{1,1}_\loc(\timeInt,\GL[n])$ of $\dot X(t)=A(t)X(t)$, $\stm(t,s)\coloneqq X(t)X(s)^{-1}$ defines a solution of \eqref{eq:stm}.
    In fact, it holds that
    \[
    \pd{}{t}\pset[\big]{X(t)X(s)^{-1}} = \dot X(t)X(s)^{-1} = A(t)X(t)X(s)^{-1} = A(t)\stm(t,s)
    \]
    and $\stm(s,s)=X(s)X(s)^{-1}=I_n$ for all $t,s\in\timeInt$. Note that $\stm\in\mathcal C(\timeInt\times\timeInt,\GL[n])$, since $X\in\mathcal C(\timeInt,\GL[n])$.
    To show that $\stm\in W^{1,1}_\loc(\timeInt\times\timeInt,\GL[n])$, we proceed as follows.
    Let $K\subseteq\timeInt\times\timeInt$ be a compact set. Since the projections $\pi_1:\timeInt\times\timeInt\to\timeInt,\ (t,s)\mapsto t$ and $\pi_2:\timeInt\times\timeInt\to\timeInt,\ (t,s)\mapsto s$ are continuous maps, $K_1\coloneqq\pi_1(K)$ and $K_2\coloneqq\pi_2(K)$ are compact subsets of $\timeInt$ such that $K\subseteq K_1\times K_2\subseteq\timeInt\times\timeInt$. Let us set $Y\coloneqq\dd{}{t}(X^{-1})\in W^{1,1}_\loc(\timeInt,\GL[n])$. Then,
    \begin{align*}
        \norm{\stm}_{W^{1,1}(K)}
        &\leq \norm{\stm}_{W^{1,1}(K_1\times K_2)}
        = \iint_{K_1\times K_2}\pset*{ \norm{\stm(t,s)} + \norm*{\pd{\stm}{t}(t,s)} + \norm*{\pd{\stm}{s}(t,s)} }\td(t\times s) \\
        &\leq \int_{K_1\times K_2}\pset*{ \norm{X(t)}\norm{Y(s)} + \norm{\dot X(t)}\norm{Y(s)} + \norm{X(t)}\norm*{Y(s)} }\td(t\times s) \\
        &= \int_{K_1}\norm{X(t)}\td t\int_{K_2}\norm{Y(s)}\td s + \int_{K_1}\norm{\dot X(t)}\td t\int_{K_2}\norm{Y(s)}\td s + \int_{K_1}\norm{X(t)}\td t\int_{K_2}\norm{\dot Y(s)}\td s \\
        &= \norm{X}_{L^1(K_1)}\norm{Y}_{L^1(K_2)} + \norm{\dot X}_{L^1(K_1)}\norm{Y}_{L^1(K_2)} + \norm{X}_{L^1(K_1)}\norm{\dot Y}_{L^1(K_2)} < \infty.
    \end{align*}
    We proceed to prove the other statements:
    \begin{enumerate}
        \item It remains to show that the solution $\stm$ of \eqref{eq:stm} is unique.
        Given any solution $\stm$, define $X_s:\timeInt\to\C^{n,n},\ t\mapsto\stm(t,s)$ for all fixed $s\in\timeInt$. Because of \Cref{lem:Fubini_generalized}, $X_s\in W^{1,1}_\loc(\timeInt,\GL[n])$ is a solution of
        $
        \dot X_s(t) = A(t)X_s(t),\ X_s(s)=\stm(s,s)=I_n
        $
        for a.e.~$s\in\timeInt$, i.e., it is the fundamental solution matrix of $\dot X(t)=A(t)X(t)$ with $X(s)=I_n$.
        This determines $\stm$ uniquely. 
        \item Let $X\in W^{1,1}_\loc(\timeInt,\GL[n])$ be the fundamental solution matrix of $\dot X(t)=A(t)X(t)$ with $X(t_0)=I_n$.
        Due to \Cref{thm:fundamentalSolution}, it holds that $\Phi(t,t_0)x_0=X(t)X(t_0)^{-1}x_0=X(t)x_0=x(t)$ for all $t\in\timeInt$.
        \item From \eqref{eq:solution_with_fundSol} and $X(t_0)=I_n$ we deduce that
        \[
        x(t) = X(t)X^{-1}(t_0)x_0 + \int_{t_0}^{t}X(t)X^{-1}(s)b(s)\td s = \stm(t,t_0)x_0 + \int_{t_0}^t\stm(t,s)b(s)\td s
        \]
        holds for all $t\in\timeInt$.
        \item For every $t,s\in\timeInt$, we have $\stm(t,s)^{-1}=(X(t)X(s)^{-1})^{-1}=X(s)X(t)^{-1}=\stm(s,t)$.
        Since $\stm^{-1}$ is the composition of $\stm$ with the transposition of its two arguments, it is clear that it preserves continuity, differentiability, and integrability properties.
        \item We have
        \[
        \pd{}{s}\stm(t,s) = \pd{}{s}\pset[\big]{X(t)X(s)^{-1}} = X(t)X(s)^{-1}\dot X(s)X(s)^{-1} = \stm(t,s)A(s)
        \]
        for every $t,s\in\timeInt$.
        \qedhere
    \end{enumerate}
\end{proof}

\noindent We now switch our focus to the existence of solutions of Riccati differential equations, see \cite{AboFIJ12,Rei72} for an extensive analysis.
\begin{theorem}\label{thm:RDE_localSolutions}
    Let $A\in L^1_\loc(\timeInt,\C^{n,n})$, $F,G\in L^1_\loc(\timeInt,\HerMat[n])$, $t_0\in\timeInt$, and $Q_0\in\HerMat[n]$.
    Then there exists $\delta>0$ such that the Riccati differential equation
    \begin{equation}\label{eq:RDE_localSolutions}
        \dot Q + \ct{A}Q + QA + QFQ + G = 0
    \end{equation}
    has a unique local solution $Q\in W^{1,1}((t_0-\delta,t_0+\delta),\HerMat[n])$ such that $Q(t_0)=Q_0$.
\end{theorem}

\begin{proof}
    Since $\timeInt$ is open, the inclusion $\timeInt_0\coloneqq(t_0-\delta,t_0+\delta)\subseteq\timeInt$ holds for arbitrary small $\delta>0$.
    Due to \Cref{lem:integralIsAC}, up to choosing $\delta>0$ sufficiently small, we can ensure that $\norm{A}_{L^1(\timeInt_0)}$, $\norm{F}_{L^1(\timeInt_0)}$, and $\norm{G}_{L^1(\timeInt_0)}$ are arbitrarily small. In particular, let us impose
    \[
    \norm{A}_{L^1} \leq \max\pset*{ \frac{1}{6(\norm{Q_0}+1)},\ \frac{1}{8} }, \quad
    \norm{F}_{L^1} \leq \max\pset*{ \frac{1}{3(\norm{Q_0}+1)^2},\ \frac{1}{8(\norm{Q_0}+1)} }, \quad
    \norm{G}_{L^1} \leq \frac{1}{3}.
    \]
    Let us restrict our search to functions $Q\in W^{1,1}(\timeInt_0,\mathcal{B})$, where
    \[
    \mathcal{B} \coloneqq \set*{ S \in \HerMat[n] \mid \norm{S-Q_0}_2 \leq 1 }.
    \]
    Every solution $Q$ of \eqref{eq:RDE_localSolutions} has to satisfy the integral equation
    \begin{equation}\label{eq:RDE_integralEquation}
        Q(t) = Q_0 - \int_{t_0}^{t}(\ct{A}Q + QA + QFQ + G)\td s
    \end{equation}
    for all $t\in\timeInt_0$, or equivalently $Q=T(Q)$, where
    \[
    T : L^\infty(\timeInt_0,\mathcal B) \to L^\infty(\timeInt_0,\mathcal B), \quad T(Q)(t) \coloneqq Q_0 - \int_{t_0}^{t}(\ct{A}Q + QA + QFQ + G)\td s.
    \]
    Note that $T$ is well-defined, since
    \begin{align*}
        \norm{T(Q)(t)-Q_0}_2 &\leq \int_{t_0}^{t}\pset[\big]{2\norm{A(s)}\norm{Q(s)} + \norm{F(s)}\norm{Q(s)}^2 + \norm{G}}\td s \\
        &\leq 2\norm{A}_{L^1}(\norm{Q_0}+1) + \norm{F}_{L^1}(\norm{Q_0}+1)^2 + \norm{G}_{L^1} < \frac{1}{3} + \frac{1}{3} + \frac{1}{3} = 1
    \end{align*}
    for all $t\in\timeInt_0$.
    Furthermore, note that $W^{1,1}(\timeInt_0,\mathcal B)\subseteq L^\infty(\timeInt_0,\mathcal B)$, and that $L^\infty(\timeInt_0,\mathcal B)$ is a Banach space, since it is a closed subspace of $L^\infty(\timeInt_0,\C^{n,n})$.
    For every $Q,\wt Q\in L^\infty(\timeInt_0,\mathcal B)$ and $t\in\timeInt_0$ it then holds that
    \begin{align*}
        &\norm{ T(Q)(t) - T(\wt Q)(t) } = \norm*{\int_{t_0}^{t}\pset[\big]{ \ct{A}(\wt Q-Q) + (\wt Q-Q)A + (\wt Q-Q)F\wt Q + QF(\wt Q-Q) }\td s} \\
        &\qquad\leq \int_{t_0}^{t}\pset[\big]{ 2\norm{A(s)}\norm{\wt Q(s)-Q(s)} + \norm{F(s)}\norm{\wt Q(s)-Q(s)}(\norm{Q(s)}+\norm{\wt Q(s)}) }\td s \\
        &\qquad\leq \norm{\wt Q-Q}_{L^\infty}\pset[\big]{ 2\norm{A}_{L^1} + 2\norm{F}_{L^1}(\norm{Q_0}+1)}
        \leq \norm{\wt Q-Q}_{L^\infty}\pset*{\frac{1}{4}+\frac{1}{4}} = \frac{1}{2}\norm{\wt Q-Q}_{L^\infty},
    \end{align*}
    and therefore $T$ is a contraction. We can then apply the Banach fixed point theorem, see e.g.\ \cite[Theorem 5.7]{Bre10} to show that $T$ has a unique fixed point $Q\in L^\infty(\timeInt_0,\mathcal B)$, which is the limit of the sequence $Q^{(k+1)}=T(Q^{(k)}),\ Q^{(0)}=Q_0$, for $k\to\infty$.
    In particular, the integral equation \eqref{eq:RDE_integralEquation} has $Q\in L^\infty(\timeInt_0,\mathcal B)$ as its unique local solution.

    To prove that $Q\in W^{1,1}(\timeInt_0,\HerMat[n])$, we observe that $K\coloneqq\ct{A}Q+QA+QFQ+G\in L^1(\timeInt_0,\HerMat[n])$ because of the generalized Hölder inequality (\Cref{thm:genHolder}), thus $L:t\mapsto\int_{t_0}^t K(s)\td s$ satisfies $L\in W^{1,1}(\timeInt_0,\HerMat[n])$ due to \Cref{lem:integralIsAC}, and therefore $Q=Q_0+L\in W^{1,1}(\timeInt_0,\HerMat[n])$.
    Due to \Cref{lem:integralIsAC} we further obtain that
    \[
    \dot Q = \dot L = \ct{A}Q + QA + QFQ + G,
    \]
    thus $Q$ is indeed a local solution of \eqref{eq:RDE_localSolutions}, and uniqueness follows from the one of the solutions of \eqref{eq:RDE_integralEquation}.
\end{proof}

\noindent If the quadratic term in \eqref{eq:RDE_localSolutions} is zero, then we have a Lyapunov differential equation. In that case, we are able to prove that its solution is global, and it can be expressed explicitly in terms of the state-transition matrix associated to some LTV system.
\begin{theorem}\label{thm:ODE_linearMatrix}
    Let $A\in L^1_\loc(\timeInt,\C^{n,n})$, $C\in L^1_\loc(\timeInt,\HerMat[n])$, $t_0\in\timeInt$, and $Q_0\in\HerMat[n]$.
    Then the initial value problem
    \begin{equation}\label{eq:ODE_linearMatrix}
        \dot Q(t) = \ct{A}(t)Q(t) + Q(t)A(t) + G(t), \qquad Q(t_0)=Q_0
    \end{equation}
    has the unique global solution
\begin{equation}\label{eq:ODE_linearMatrix_solution}
        Q(t) = \stm_{\ct A}(t,t_0)Q_0\ct{\stm_{\ct A}(t,t_0)} + \int_{t_0}^{t}\stm_{\ct A}(t,s)G(s)\ct{\stm_{\ct A}(t,s)}\td s,
    \end{equation}
    where $\stm_{\ct A}\in W^{1,1}_\loc(\timeInt,\GL[n])$ denotes the state-transition matrix associated to $\dot x=\ct{A}x$.
\end{theorem}

\begin{proof}
    The existence of a unique local solution follows immediately from \Cref{thm:RDE_localSolutions} with $F=0$.
    Note that $Q$ defined as in \eqref{eq:ODE_linearMatrix_solution} satisfies $Q(t_0)=Q_0$ and is in $W^{1,1}_\loc(\timeInt,\C^{n,n})$, by construction.
    By the properties of differentiation, we obtain then
    \begin{align*}
        \dot Q(t) &= \dot\stm_{\ct A}(t,t_0)Q_0\ct{\stm_{\ct A}(t,t_0)} + \stm_{\ct A}(t,t_0)Q_0\ct{\dot\stm_{\ct A}(t,t_0)} + \stm_{\ct A}(t,t)C(t)\stm_{\ct A}(t,t)~+ \\
        &\qquad+ \int_{t_0}^{t}\pset[\big]{\dot\stm_{\ct A}(t,s)G(s)\ct{\stm_{\ct A}(t,s)} + \stm_{\ct A}(t,s)G(s)\ct{\dot\stm_{\ct A}(t,s)}}\td s = \\
        &= \ct{A(t)}\stm_{\ct A}(t,t_0)Q_0\ct{\stm_{\ct A}(t,t_0)} + \stm_{\ct A}(t,t_0)Q_0\stm_{\ct A}(t,t_0)A(t) + G(t)~+ \\
        &\qquad+ \int_{t_0}^{t}\pset[\big]{\ct{A(t)}\stm_{\ct A}(t,s)G(s)\ct{\stm_{\ct A}(t,s)} + \stm_{\ct A}(t,s)G(s)\ct{\stm_{\ct A}(t,s)}A(t)}\td s = \\
        &= \ct{A(t)}\pset*{ \stm_{\ct A}(t,t_0)Q_0\ct{\stm_{\ct A}(t,t_0)} + \int_{t_0}^{t}\stm_{\ct A}(t,s)G(s)\ct{\stm_{\ct A}(t,s)}\td s }~+ \\
        &\qquad+ \pset*{\stm_{\ct A}(t,t_0)Q_0\ct{\stm_{\ct A}(t,t_0)} + \int_{t_0}^{t}\stm_{\ct A}(t,s)G(s)\ct{\stm_{\ct A}(t,s)}\td s}A(t) + G(t) = \\
        &= \ct{A(t)}Q(t) + Q(t)A(t) + G(t),
    \end{align*}
    thus $Q$ is in fact the unique solution of \eqref{eq:ODE_linearMatrix}.
\end{proof}

\noindent Finally, we study the solutions of a matrix integral equation, that resembles a Lyapunov equation but with an additional term. 
\begin{corollary}\label{cor:OIE_linearMatrix}
    Let $A\in L^1_\loc(\timeInt,\C^{n,n})$, $L\in L^1_\loc(\timeInt,\HerMat[n])$, $M\in L^\infty_\loc(\timeInt,\HerMat[n])$, $t_0\in\timeInt$, and $Q_0\in\HerMat[n]$. Then the integral equation
 \begin{equation}\label{eq:OIE_linearMatrix}
        Q(t) = Q(t_0) + \int_{t_0}^{t}\pset[\big]{ \ct{A(t)}Q(s) + Q(s)A(s) + L(s) }\td s  + M(t) - M(t_0), \qquad Q(t_0)=t_0
    \end{equation}
    has the unique global solution $Q\in L^\infty_\loc(\timeInt,\HerMat[n])$ of the form
    \begin{equation}\label{eq:OIE_linearMatrix_solution}
    \begin{split}
    Q(t) &= M(t) + \stm_{\ct A}(t,t_0)\pset[\big]{Q_0-M(t_0)}\ct{\stm_{\ct A}(t,t_0)}~+ \\
    &\qquad+ \int_{t_0}^{t}\stm_{\ct A}(t,s)\pset[\big]{L(s) +\ct{A}(s)M(s) + M(s)A(s)}\ct{\stm_{\ct A}(t,s)}\td s,
    \end{split}
    \end{equation}
    where $\stm_{\ct A}\in W^{1,1}_\loc(\timeInt,\GL[n])$ denotes the state-transition matrix associated to $\dot x=\ct{A}x$.
\end{corollary}
\begin{proof}
    Defining $\wt Q\coloneqq Q-M$, there is an obvious bijection between the solutions of \eqref{eq:OIE_linearMatrix} and the ones of
    \[
        \wt Q(t) = \wt Q(t_0) + \int_{t_0}^{t}\pset*{ \ct{A(s)}\wt Q(s) + \wt Q(s)A(s) + L(s) + \ct{A(s)}M(s) + M(s)A(s) }\td s,
    \]
    with $\wt Q(t_0)=Q(t_0)-M(t_0)$, which are exactly the solutions of the ordinary differential matrix equation
    \begin{equation}\label{eq:OIE_linearMatrix_proof:1}
        \dot{\wt Q}(t) = \ct{A(t)}\wt Q(t) + \wt Q(t)A(t) + \wt L(t), \qquad \wt Q(t_0)=Q_0-M(t_0),
    \end{equation}
    with $\wt L\coloneqq L+\ct{A}M+MA\in L^1_\loc(\timeInt,\HerMat[n])$.
    It then immediately follows from \Cref{thm:ODE_linearMatrix} that \eqref{eq:OIE_linearMatrix_proof:1} has a unique solution of the form
    \[
    \wt Q(t) = \stm_{\ct A}(t,t_0)\pset[\big]{Q_0-M(t_0)}\ct{\stm_{\ct A}(t,t_0)} + \int_{t_0}^{t}\stm_{\ct A}(t,s)\wt L(s)\ct{\stm_{\ct A}(t,s)}\td s,
    \]
    and by setting $Q(t)=\wt Q(t)+M(t)$ and $\wt L(s)=L(s)+\ct{A(s)}M(s)+M(s)A(s)$ we obtain \eqref{eq:OIE_linearMatrix_solution}.
\end{proof}

\section{System transformations}\label{sec:systemTransformations}

In this section, we prove the invariance of the disspativity properties under transformations presented in~\Cref{thm:invariance}.

\subsection{State transformations}

Since the input and output variables are unchanged under state space transformations, it is clear that the property of a system to have nonnegative supply \refNN{} is invariant.
It is also clear that the passive property \refPa{} is invariant under state space transformations, up to the replacement of a storage function $V:\timeInt\times\C^n\to\R$ with
$
\wt V : \timeInt\times\C^n \to \R, \ 
(t,\wt x) \mapsto V\pset{ t, Z(t)\wt x }.
$
It is somewhat more involved to show that the properties of admitting a \refKYP{} solution and a \refPH{} representation are also invariant under state-space transformations.
\begin{lemma}
   Let $Z\in W^{1,1}_\loc(\timeInt, \GL[n])$ be a state transformation applied to a LTV system of the form \eqref{eq:tv_system}. A matrix function $Q\in W^{1,1}_\loc(\timeInt,\posSD[n])$ is a solution of the \refKYP{} inequality \eqref{eq:KYP} if and only if $\wt Q\coloneqq \ct{Z}QZ\in W^{1,1}_\loc(\timeInt,\posSD[n])$ is a solution of the KYP inequality
    \begin{equation*}
    \bmat{
        -\ct{\wt A}\wt Q - \wt Q\wt A - \dot{\wt Q} & \ct{\wt C}-\wt Q\wt B \\
        \wt C-\ct{\wt B}\wt Q & D+\ct{D}
        } \geq 0
    \end{equation*}
    of \eqref{eq:tv_system_changeOfVariables}.
\end{lemma}
\begin{proof}
    Suppose that $Q$ is a solution of the KYP inequality \eqref{eq:KYP}. By construction, $\wt Q\in\posSD[n]$ and
    \[
    \bmat{
        - \wt{\ct{A}}\wt Q - \wt Q\ct{\wt A} - \dot{\wt Q} & \ct{\wt C} - \wt Q\wt B \\
        \wt C - \ct{\wt B}\wt Q & D + \ct{D}
    }
    = \ct{\bmat{Z & 0 \\ 0 & I_m}}
    \bmat{-\ct AQ-QA-\dot Q & \ct{C}-QB \\ C-\ct{B}Q & D+\ct{D}}
    \bmat{Z & 0 \\ 0 & I_m} \geq 0.
    \]
    Since $Z$ is pointwise invertible, the converse statement follows analogously.
\end{proof}

\begin{lemma}\label{lem:PHS_changeOfVariables}
Let $Z\in W^{1,1}_\loc(\timeInt, \GL[n])$ be a state transformation applied to a LTV system of the form \eqref{eq:tv_system}. The system \eqref{eq:tv_system} admits a pH representation of the form \eqref{eq:pH_coefficients} if and only if the transformed system \eqref{eq:tv_system_changeOfVariables} admits a pH representation of the form
    \begin{equation}\label{eq:pH_stateTransCoeff}
        \bmat{\wt A & \wt B \\ \wt C & D} = \bmat{ \pset[\big]{\wt J-\wt R}\wt Q-\wt K & \wt G-\wt P \\ \pset[\big]{\ct{\wt G}+\ct{\wt P}}\wt Q & S-N },
    \end{equation}
    where $\wt K=Z^{-1}(\dot Z+KZ)$, $\wt J=Z^{-1}J\ict{Z}$, $\wt R=Z^{-1}R\ict{Z}$, $\wt Q=\ct{Z}QZ$, $\wt G=Z^{-1}G$, and $\wt P=Z^{-1}P$.
\end{lemma}

\begin{proof}
Suppose that \eqref{eq:tv_system} admits a pH representation of the form \eqref{eq:pH_coefficients}.
Note first that
\begin{align*}
        & (\wt J-\wt R)\wt Q-\wt K = Z^{-1}(J-R)QZ - Z^{-1}(\dot Z+KZ) = Z^{-1}(AZ-\dot Z) = \wt A, \\
        & \wt G-\wt P = Z^{-1}(G-P) = Z^{-1}B = \wt B, \\
        & \ct{(\wt G+\wt P)}\wt Q = \ct{(G+P)}QZ = CZ = \wt C
\end{align*}
holds, i.e., the coefficient representation \eqref{eq:pH_stateTransCoeff} is well-defined.
It is also clear that all matrix functions retain their regularity, that $\wt Q\in\posSD[n]$, and that
\begin{align*}
        \wt Q\wt K + \ct{\wt K}\wt Q
        &= \ct{Z}Q(\dot Z+KZ) + \ct{(\dot Z+KZ)}QZ = \\
        &= \ct{Z}Q\dot Z + \ct{Z}(QK+\ct{K}Q)Z + \ct{\dot Z}Q\dot Z
        = \dd{}{t}(\ct{Z}QZ) = \dot{\wt Q}.
\end{align*}
Finally, it holds that
\[
    \wt W \coloneqq \bmat{\wt R & \wt P \\ \ct{\wt P} & S} = \bmat{Z^{-1} & 0 \\ 0 & I_m}W\ct{\bmat{Z^{-1} & 0 \\ 0 & I_m}},
\]
in particular $\wt W(t)\in\posSD[n+m]$ for a.e.~$t\in\timeInt$.
Therefore, \eqref{eq:pH_stateTransCoeff} is a pH system with the given representation.
Since $Z$ is pointwise invertible, the converse statement follows analogously.
\end{proof}

\subsection{Input-output transformations}
If we apply a change of input-output variables of the form $\check u=V^{-1}u$, $\check y=\ct{V}y$ with $V,V^{-1}\in L^\infty_\loc(\timeInt,\GL[m])$, it is clear that $\ct{\check y}\check u=\ct{y}u$. It immediately follows that the properties of having a nonnegative supply \refNN{} and of being passive \refPa{} are invariant under this transformation, with the same storage functions.
We explicitly prove the invariance of \refPH{} and \refKYP{}:

\begin{lemma}
    Consider an LTV system of the form \eqref{eq:tv_system} and a change of input-output variables $V,V^{-1}\in L^\infty_\loc(\timeInt,\GL[m])$.
    Then $Q\in W^{1,1}_\loc(\timeInt,\posSD[n])$ is a solution of the \refKYP{} inequality \eqref{eq:KYP} if and only if it is also a solution of the KYP inequality
    \begin{equation*}
        \bmat{-\ct{A}Q - QA - \dot Q & \ct{\check C}-Q\check B \\ \check C-\ct{\check B}Q & \check{D}+\ct{\check{D}}} \geq 0
    \end{equation*}
    associated to the transformed system \eqref{eq:LTV_IO_ChangeOfVariables}.
\end{lemma}

\begin{proof}
    Suppose that $Q\in W^{1,1}_\loc(\timeInt,\posSD[n])$ is a solution of the KYP inequality \eqref{eq:KYP}. Then
    \[
    \bmat{-\ct{A}Q - QA - \dot Q & \ct{\check C}-Q\check B \\ \check C-\ct{\check B}Q & \check{D}+\ct{\check{D}}} = \ct{\bmat{I_n & 0 \\ 0 & V}}\bmat{-\ct{A}Q-QA-\dot Q & \ct{C}-QB \\ C-\ct{B}Q & D+\ct{D}}\bmat{I_n & 0 \\ 0 & V} \geq 0
    \]
    holds, by construction. Since $V$ is a.e.~invertible, the converse statement is proven analogously.
\end{proof}

\begin{lemma}
    Consider an LTV system of the form \eqref{eq:tv_system} and a change of input-output variables $V,V^{-1}\in L^\infty_\loc(\timeInt,\GL[m])$.
    Then the LTV system \eqref{eq:tv_system} has a \refPH{} formulation of the form \eqref{eq:PHS} if and only if
    \begin{equation}\label{eq:pH_IO_changeOfVariables}
        \bmat{A & \check{B} \\ \check{C} & \check{D}} = \bmat{ (J-R)Q-K & \check{G}-\check{P} \\ \ct{(\check{G}+\check{P})}Q & \check{S}-\check{N} }
    \end{equation}
    with $\check{G}=GV$, $\check{P}=PV$, $\check{S}=\ct{V}SV$, and $\check{N}=\ct{V}NV$ is a \refPH{} formulation for the transformed system \eqref{eq:LTV_IO_ChangeOfVariables}.
\end{lemma}
\begin{proof}
    Suppose that \eqref{eq:tv_system} admits a pH representation of the form \eqref{eq:PHS}. By construction, we have that
    \[
    \bmat{ (J-R)Q-K & \check{G}-\check{P} \\ \ct{(\check{G}+\check{P})}Q & \check{S}-\check{N} }
    = \bmat{ A & (G-P)V \\ \ct{V}\ct{(G+P)}Q & \ct{V}(S-N)V }
    = \bmat{ A & BV \\ \ct{V}C & \ct{V}DV }
    = \bmat{A & \check{B} \\ \check{C} & \check{D}},
    \]
    i.e., \eqref{eq:pH_IO_changeOfVariables} holds.
    Furthermore, it is clear that the properties $J(t)=-\ct{J(t)}$, $\check{N}(t)=-\ct{N(t)}$, $Q(t)K(t)+\ct{K(t)}Q(t)=\dot Q(t)$ and
    \[
    \check W(t) \coloneqq \bmat{ R(t) & \check{P}(t) \\ \ct{\check{P}(t)} & \check{S}(t) }
    = \ct{\bmat{I_n & 0 \\ 0 & V(t)}}W(t)\bmat{I_n & 0 \\ 0 & V(t)} \in \posSD[n+m]
    \]
    hold for a.e.~$t\in\timeInt$, thus $Q,K,J,R,\check{G},\check{P},\check{S},\check{N}$ define a pH formulation for \eqref{eq:LTV_IO_ChangeOfVariables}.
    Since $V$ is a.e.~invertible, the converse statement is proven analogously.
\end{proof}

\subsection{Time transformations}

An immediate consequence of \Cref{lem:timeTransformation} is that the property of having a nonnegative supply \refNN{} is invariant under time transformations.
Similarly, passivity \refPa{} is also invariant under time transformations, up to replacing a storage function $V:\timeInt\times\C^n\to\R$ with $\wh V:\wh\timeInt\times\C^n\to\R,\ (\wh t,\wh x)\mapsto V(\theta(\wh t),\wh x)$. It is again more involved to show that the properties of admitting a \refKYP{} solution and a \refPH{} representation are invariant under time transformations.

\begin{lemma}
    Given a system of the form \eqref{eq:tv_system} and a diffeomorphism $\theta\in\mathcal C^1(\wh\timeInt,\timeInt)$ with $\dot\theta>0$ pointwise, a matrix function $Q\in W^{1,1}_\loc(\timeInt,\C^{n,n})$ satisfying $Q=\ct{Q}\geq 0$ pointwise is a solution of the KYP inequality \eqref{eq:KYP} of \eqref{eq:tv_system} if and only if $\wh Q\coloneqq Q\circ\theta\in W^{1,1}_\loc(\wh\timeInt,\C^{n,n})$ is a solution of the KYP inequality
    \begin{equation}\label{eq:timeTransKYP}
    \bmat{
        (-\ct{\wh A}\wh Q - \wh Q\wh A - \dot{\wh Q})(\wh t) & (\ct{\wh C}-\wh Q\wh B)(\wh t) \\
        (\wh C-\ct{\wh B}\wh Q)(\wh t) & (\wh D+\ct{\wh D})(\wh t)
        } \geq 0
    \end{equation}
    of \eqref{eq:tv_system_changeOfVariables}, for a.e.~$\wh t\in\wh\timeInt$.
\end{lemma}

\begin{proof}
    Suppose that $Q$ is a solution of the KYP inequality of \eqref{eq:tv_system}.
    Note first that trivially $\wh Q=\ct{\wh Q}\geq 0$ pointwise.
    Since $\dot{\wh Q}=\dot\theta(Q\circ\theta)$ and $\dot\theta>0$ pointwise, we immediately deduce that
    \[
    \bmat{
        -\ct{\wh A}\wh Q - \wh Q\wh A - \dot{\wh Q} & \ct{\wh C}-\wh Q\wh B \\
        \wh C-\ct{\wh B}\wh Q & \wh D+\ct{\wh D}
        }
        = \dot\theta\pset*{\bmat{-\ct AQ-QA-\dot Q & \ct{C}-QB \\ C-\ct{B}Q & D+\ct{D}}\circ\theta} \geq 0,
    \]
    i.e., $\wh Q$ is a solution of \eqref{eq:timeTransKYP}.
    Since $\theta$ is a diffeomorphism and $\dd{}{t}(\theta^{-1})=1/(\dot\theta\circ\theta^{-1})>0$ on $\timeInt$, the converse statement can be proven analogously.
\end{proof}

\begin{lemma}
Consider a system of the form \eqref{eq:tv_system} and a diffeomorphism $\theta\in\mathcal C^1(\wh\timeInt,\timeInt)$ with $\dot\theta>0$ pointwise. Then the system \eqref{eq:tv_system} admits a pH representation of the form \eqref{eq:pH_coefficients} if and only if the transformed system \eqref{eq:tv_system_timeTrans} admits a pH representation of the form
    \begin{equation}\label{eq:pH_timeTransCoeff}
        \bmat{\wh A & \wh B \\ \wh C & \wh D} = \bmat{ \pset[\big]{\wh J-\wh R}\wh Q-\wh K & \wh G-\wh P \\ \pset[\big]{\ct{\wh G}+\ct{\wh P}}\wh Q & \wh S-\wh N },
    \end{equation}
    where $\wh Q=Q\circ\theta$ and $\wh\wildcard=\dot\theta(\wildcard\circ\theta)$ for $\wildcard\in\set{J,R,K,G,P,S,N}$.
\end{lemma}

\begin{proof}
    Suppose that \eqref{eq:tv_system} admits a pH representation of the form \eqref{eq:pH_coefficients}.
    Note first that
    \begin{align*}
        & (\wh J-\wh R)\wh Q - \wh K = \dot\theta\pset[\big]{ (J\circ\theta)-(R\circ\theta) }(Q\circ\theta) - \dot\theta(K\circ\theta) = \dot\theta(A\circ\theta) = \wh A, \\
        & \wh G - \wh P = \dot\theta(G\circ\theta) - \dot\theta(P\circ\theta) = \dot\theta(B\circ\theta) = \wh B, \\
        & \ct{(\wh G+\wh P)}\wh Q = \dot\theta\ct{(G\circ\theta)+(P\circ\theta)}(Q\circ\theta) = \dot\theta(C\circ\theta) = \wh C, \\
        & \wh S-\wh N = \dot\theta(S\circ\theta) - \dot\theta(N\circ\theta) = \dot\theta(D\circ\theta) = \wh D
    \end{align*}
    holds, i.e., the coefficient representation \eqref{eq:pH_timeTransCoeff} is well-defined.
    It is also clear that all matrix functions retain their regularity, that $\wh Q=\ct{\wh Q}\geq 0$ by construction, and that
    \[
    \wh Q\wh K + \ct{\wh K}\wh Q = \dot\theta\pset[\big]{ (Q\circ\theta)(K\circ\theta) + (K\circ\theta)^\top(Q\circ\theta) } = \dot\theta(\dot Q\circ\theta) = \dot{\wh Q}.
    \]
    Finally, it holds that
    \[
    \wh W \coloneqq \bmat{\wh R & \wh P \\ \ct{\wh P} & \wh S}
    = \dot\theta\bmat{ R & P \\ \ct P & S }\circ\theta,
    \]
    in particular $\wh W=\ct{\wh W}\geq 0$ almost everywhere.
    Therefore, \eqref{eq:tv_system_timeTrans} is a pH system with the given representation.
    Since $\theta$ is invertible, the converse statement can be proven analogously.
\end{proof}

\section{The Riemann-Stieltjes integral}\label{sec:RiemannStieltjes}

We recall the Riemann-Stieltjes integral and its main properties; see e.g.~\cite{Car00} for an overview in the real-valued case.
In this section, $[t_0,t_1]$ always denotes a fixed compact real interval, unless otherwise specified.

\begin{definition}\label{def:RS_integral}
    Given two maps $f,g:[t_0,t_1]\to\C$, the \textit{Riemann-Stieltjes integral} of $f$ with respect to $g$ over $[t_0,t_1]$ is defined (when it exists) as
    \begin{equation}
    	\int_{t_0}^{t_1}f(t)\td g(t) \coloneqq \lim_{\norm{\partit}\to 0}\sum_{k=1}^{K} f(\wt s_k)\pset[\big]{ g(\wt t_k) - g(\wt t_{k-1}) },
    \end{equation}
    where for every partition $\partit=\set{\wt t_0,\ldots,\wt t_K}$ of $[t_0,t_1]$ we take any $\wt s_k\in[\wt t_{k-1},\wt t_k]$, for $k=1,\ldots,K$.
    If the Riemann-Stieltjes integral exists, we write $f\in\RS{g}([t_0,t_1],\C)$, or in the short notation $f\in\RS{g}$, when $[t_0,t_1]$ and $\C$ are clear from the context.
\end{definition}

\noindent Note that, for the Riemann-Stieltjes integral to be well defined, the convergence of the limit should be independent of the choice of $\set{\wt t_k}$ and $\set{\wt s_k}$.
We define implicitly
\[
\int_{t_0}^{t_1}\td g(t) \coloneqq \int_{t_0}^{t_1} 1\, \td g(t) = g(t_1)-g(t_0),
\]
by setting $f:t\mapsto 1$, obtaining something similar to the fundamental theorem of calculus.

We recall now several fundamental properties of the Riemann-Stieltjes integral without providing their proof, see, e.g.~\cite{Car00}.

\begin{theorem}
    For every $g:[t_0,t_1]\to\C$, the set $\RS{g}([t_0,t_1],\C)$ is a vector space and an algebra.
\end{theorem}

\begin{theorem}\label{thm:RS_integrationByParts}
    Let $g\in\timeInt\to\C$ and $f\in\RS{g}$. Then $g\in\RS{f}$ and the integration by parts formula
    \begin{equation}\label{eq:RS_integrationByParts}
        f(t_1)g(t_1) - f(t_0)g(t_0) = \int_{t_0}^{t_1}f(t)\td g(t) + \int_{t_0}^{t_1}g(t)\td f(t)
    \end{equation}
    holds.
\end{theorem}

\begin{theorem}\label{thm:RS_C_and_BV}
    If $f\in\mathcal C([t_0,t_1],\C)$ and $g\in\BV([t_0,t_1],\C)$, then $f\in\RS{g}$, $g\in\RS{f}$, and
    \[
    \abs*{ \int_{t_0}^{t_1} f(t)\td g(t) } \leq \norm{f}_{L^\infty([t_0,t_1])}\totVar[t_0][t_1](g).
    \]
\end{theorem}

\begin{theorem}\label{thm:RS_C_and_AC}
    If $f\in\mathcal C([t_0,t_1],\C)$ and $g\in W^{1,1}([t_0,t_1],\C)$, then
    \[
    \int_{t_0}^{t_1} f(t)\td g(t) = \int_{t_0}^{t_1} f(t)\dot g(t)\td t.
    \]
\end{theorem}

\noindent Before focusing on the matrix-valued case, we show that under some additional regularity assumptions, \eqref{eq:RS_integrationByParts} can be extended to the product of three functions.
\begin{lemma}\label{lem:RS_threeTerms}
    Let $f,g,h\in\BV([t_0,t_1],\C)$ and suppose that any two among $f,g,h$ are continuous. Then $fg\in\RS{h}$, $fh\in\RS{g}$, $gh\in\RS{f}$ and the formula
    \begin{equation}\label{eq:RS_threeTerms}
        f(t_1)g(t_1)h(t_1) - f(t_0)g(t_0)h(t_0) = \int_{t_0}^{t_1}f(t)g(t)\td h(t) + \int_{t_0}^{t_1}f(t)h(t)\td g(t) + \int_{t_0}^{t_1}g(t)h(t)\td f(t)
    \end{equation}
    holds.
\end{lemma}
\begin{proof}
    Due to the obvious symmetry of the statement, it is sufficient to prove it for $f,g\in\mathcal C([t_0,t_1],\C)$.
    Since $fg\in\mathcal C$ and $h\in\BV$, $g\in\mathcal C$ and $fh\in\BV$, and $f\in\mathcal C$ and $gh\in\BV$, we deduce from \Cref{thm:RS_C_and_BV} that $fg\in\RS{h}$, $fh\in\RS{g}$ and $gh\in\RS{f}$.
    
    For $\varepsilon>0$ arbitrarily small, let $\delta>0$ be such that, for all partitions $\partit=\set{\wt t_0,\ldots,\wt t_K}\in\partInt{[t_0,t_1]}$ with $\norm{\partit}<\delta$ and every choice of $\wt s_k\in[\wt t_{k-1},\wt t_k]$ for $k=1,\ldots,K$, it holds that $\abs{D_1},\abs{D_2},\abs{D_3}<\varepsilon$, where
    \begin{align*}
        D_1 & \coloneqq \int_{t_0}^{t_1}f(t)g(t)\td h(t) - \sum_{k=1}^K f(\wt s_k)g(\wt s_k)\pset[\big]{(h(\wt t_k)-h(\wt t_{k-1})}, \\
        D_2 & \coloneqq \int_{t_0}^{t_1}f(t)h(t)\td g(t) - \sum_{k=1}^K f(\wt s_k)h(\wt s_k)\pset[\big]{(g(\wt t_k)-g(\wt t_{k-1})}, \\
        D_3 & \coloneqq \int_{t_0}^{t_1}g(t)h(t)\td f(t) - \sum_{k=1}^K g(\wt s_k)h(\wt s_k)\pset[\big]{(f(\wt t_k)-f(\wt t_{k-1})}.
    \end{align*}
    Up to selecting a smaller $\delta$, we can also suppose that $\abs{f(t)-f(s)}<\varepsilon$ for every $s,t\in[t_0,t_1]$ with $\abs{t-s}<\delta$, since $f$ is uniformly continuous.
    Let us now denote $f_k\coloneqq f(\wt t_k)$, $g_k\coloneqq g(\wt t_k)$ and $h_k\coloneqq h(\wt t_k)$ for $k=0,\ldots,K$. Then we have
    \begin{align*}
        & \sum_{k=1}^K f_kg_k(h_k-h_{k-1}) + \sum_{k=1}^K f_{k-1}h_{k-1}(g_k-g_{k-1}) + \sum_{k=1}^K g_{k-1}h_{k-1}(f_k-f_{k-1}) \\
        &= f_Kg_Kh_K + \sum_{k=1}^{K-1}f_kg_kh_k - \sum_{k=1}^{K}f_kg_kh_{k-1} + \sum_{k=1}^K f_{k-1}g_kh_{k-1} \\
        &\qquad- f_0g_0h_0 - \sum_{k=2}^K f_{k-1}g_{k-1}h_{k-1} + \sum_{k=1}^K f_kg_{k-1}h_{k-1} - \sum_{k=1}^K f_{k-1}g_{k-1}h_{k-1} \\
        &= f_Kg_Kh_K - f_0g_0h_0 - \sum_{k=1}^K(f_kg_k - f_{k-1}g_k - f_kg_{k-1} + f_{k-1}g_{k-1})h_{k-1} \\
        &= f(t_1)g(t_1)h(t_1) - f(t_0)g(t_0)h(t_0) - \sum_{k=1}^K(f_k-f_{k-1})(g_k-g_{k-1})h_{k-1},
    \end{align*}
    and therefore
    \begin{align*}
        & \int_{t_0}^{t_1}f(t)g(t)\td h(t) + \int_{t_0}^{t_1}f(t)h(t)\td g(t) + \int_{t_0}^{t_1}g(t)h(t)\td f(t) \\
        &= \sum_{k=1}^K f_kg_k(h_k-h_{k-1}) + \sum_{k=1}^K f_{k-1}h_{k-1}(g_k-g_{k-1}) + \sum_{k=1}^K g_{k-1}h_{k-1}(f_k-f_{k-1}) + D_1 + D_2 + D_3 \\
        &= f(t_1)g(t_1)h(t_1) - f(t_0)g(t_0)h(t_0) - \sum_{k=1}^K(f_k-f_{k-1})(g_k-g_{k-1})h_{k-1} + D_1 + D_2 + D_3.
    \end{align*}
    Since $t_k-t_{k-1}\leq\norm{\partit}<\delta$, we have
    \[
    \abs*{\sum_{k=1}^K(f_k-f_{k-1})(g_k-g_{k-1})h_{k-1}} \leq \sum_{k=1}^K\varepsilon\abs{g_k-g_{k-1}}\norm{h}_{L^\infty} \leq \varepsilon\totVar[t_0][t_1]\norm{h}_{L^\infty},
    \]
    and therefore
    \begin{align*}
        & \abs*{\int_{t_0}^{t_1}f(t)g(t)\td h(t) + \int_{t_0}^{t_1}f(t)h(t)\td g(t) + \int_{t_0}^{t_1}g(t)h(t)\td f(t) - \pset[\big]{f(t_1)g(t_1)h(t_1) - f(t_0)g(t_0)h(t_0)}} \\
        &\leq \abs*{\sum_{k=1}^K(f_k-f_{k-1})(g_k-g_{k-1})h_{k-1}} + \abs{D_1} + \abs{D_2} + \abs{D_3} < \varepsilon(\totVar[t_0][t_1]\norm{h}_{L^\infty}+3).
    \end{align*}
    Since $\varepsilon$ can be chosen arbitrarily small, we conclude that \eqref{eq:RS_threeTerms} holds.
\end{proof}

\subsection{The matrix Riemann-Stieltjes integral}

We extend the Riemann-Stieltjes integral to matrix-valued functions in the following way.
\begin{definition}\label{def:RSM_integral}
    Given three matrix functions $F=[f_{ij}]:[t_0,t_1]\to\C^{m,n}$, $G=[g_{jk}]:[t_0,t_1]\to\C^{n,p}$ and $Z=[z_{kl}]:[t_0,t_1]\to\C^{p,q}$, the \emph{Riemann-Stieltjes integral} of $F$ and $Z$ with respect to $G$ is defined (when it exists) as
    \begin{equation}
    	\int_{t_0}^{t_1}\pset[\big]{F(t)\td G(t)Z(t)} \coloneqq
        \bset*{ \sum_{j,k}\int_{t_0}^{t_1}f_{ij}(t)z_{kl}(t)\td g_{jk}(t) }_{i,l} \in \C^{m,q},
    \end{equation}
    consistently with matrix multiplication. If the Riemann-Stieltjes integral exists, we write $(F,Z)\in\RS{G}([t_0,t_1],\C^{m,q})$, or, in short, $(F,Z)\in\RS{G}$, when $[t_0,t_1]$ and $\C^{m,q}$ are clear from the context.
\end{definition}

\noindent It is also useful to consider the following equivalent definition of the matrix Riemann-Stieltjes integral.
\begin{lemma}\label{lem:RSM_altDef}
    Let $F=[f_{ij}]:[t_0,t_1]\to\C^{m,n}$, $G=[g_{jk}]:[t_0,t_1]\to\C^{n,p}$ and $Z=[z_{kl}]:[t_0,t_1]\to\C^{p,q}$ such that $(F,Z)\in\RS{G}$. Then
    \begin{equation}\label{eq:RSM_altDef}
    \int_{t_0}^{t_1}\pset[\big]{F(t)\td G(t)Z(t)} = \lim_{\norm{\partit}\to 0} \sum_{k=1}^K F(\wt s_k)\pset[\big]{G(\wt t_k)-G(\wt t_{k-1})}Z(\wt s_k),
    \end{equation}
    where for every partition $\partit=\set{\wt t_0,\ldots,\wt t_K}\in\partInt{[t_0,t_1]}$ we take any $\wt s_k\in[\wt t_{k-1},\wt t_k]$, for $k=1,\ldots,K$.
\end{lemma}
\begin{proof}
    From the definition of the scalar Riemann-Stieltjes integral and of matrix multiplication we obtain that
    \begin{align*}
        \int_{t_0}^{t_1}\pset[\big]{F(t)\td G(t)Z(t)}
        &= \bset*{ \sum_{j,k}\int_{t_0}^{t_1}f_{ij}(t)z_{kl}(t)\td g_{jk}(t) }_{i,l} \\
        &= \bset*{ \sum_{j,k}\lim_{\norm{\partit}\to 0} \sum_{m=1}^K f_{ij}(\wt s_m)z_{kl}(\wt s_m)\pset[\big]{g_{jk}(\wt t_m)-g_{jk}(\wt t_{m-1})} }_{i,l} \\
        &= \lim_{\norm{\partit}\to 0} \sum_{m=1}^K \bset*{ \sum_{j,k} f_{ij}(\wt s_m)z_{kl}(\wt s_m)\pset[\big]{g_{jk}(\wt t_m)-g_{jk}(\wt t_{m-1})} }_{i,l} \\
        &= \lim_{\norm{\partit}\to 0} \sum_{m=1}^K F(\wt s_m)\pset[\big]{G(\wt t_m)-G(\wt t_{m-1})}Z(\wt s_m). \qedhere
    \end{align*}
\end{proof}

\noindent We define implicitly
\[
\int_{t_0}^{t_1} (F\td G) \coloneqq \int_{t_0}^{t_1} \pset[\big]{F(\td G)I_p}, \qquad
\int_{t_0}^{t_1} \pset[\big]{(\td G)Z} \coloneqq \int_{t_0}^{t_1} \pset[\big]{I_n(\td G)Z},
\]
and
\[
\int_{t_0}^{t_1} \td G \coloneqq \int_{t_0}^{t_1} \pset[\big]{I_n(\td G)I_p} = G(t_1) - G(t_0),
\]
which follows immediately from \Cref{lem:RSM_altDef}.
We proceed by extending to the matrix-valued case several statements that are true for the scalar Riemann-Stieltjes integral. 
\begin{lemma}
    Let $F\in\mathcal C([t_0,t_1],\C^{m,n})$, $G\in\BV([t_0,t_1],\C^{n,p})$ and $Z\in\mathcal C([t_0,t_1],\C^{p,q})$. Then $(F,Z)\in\RS{G}$ and
    \begin{equation}
        \norm*{ \int_{t_0}^{t_1}\pset[\big]{F(t)\td G(t)Z(t)} }_2 \leq \norm{F}_{L^\infty}\norm{Z}_{L^\infty}\totVar[t_0][t_1](G),
    \end{equation}
    where $\totVar[t_0][t_1](G)$ denotes the total variation of $G$.
\end{lemma}
\begin{proof}
    Let $F=[f_{ij}]$, $G=[g_{jk}]$ and $Z=[z_{kl}]$ denote the entries of the matrix functions.
    Since $F$ and $Z$ are continuous and $G$ is of bounded variation, then  $f_{ij},z_{kl}\in\mathcal{C}([t_0,t_1],\C)$ and $g_{jk}\in\BV([t_0,t_1],\C)$ for all $i,j,k,l$,
    and therefore the Riemann-Stieltjes integral $\int_{t_0}^{t_1}f_{ij}(t)z_{kl}(t)\td g_{jk}(t)$ exists for all $i,j,k,l$, because of \Cref{thm:RS_C_and_BV}.
    We conclude that the matrix Riemann-Stieltjes integral is well-defined.

    To obtain the inequality, we observe with the alternative definition \eqref{eq:RSM_altDef} that
    \begin{align*}
    \norm*{ \sum_{k=1}^K F(\wt s_k)\pset[\big]{G(\wt t_k)-G(\wt t_{k-1})}Z(\wt s_k) }_2
    &\leq \sum_{k=1}^K \norm{F}_{L^\infty}\norm{Z}_{L^\infty}\norm{G(\wt t_k)-G(\wt t_{k-1})} \\ &\leq \norm{F}_{L^\infty}\norm{Z}_{L^\infty}\totVar[t_0][t_1](G)
    \end{align*}
    holds for every partition $\partit=\set{\wt t_k}\in\partInt{[t_0,t_1]}$ and corresponding choice of $\wt s_k$, and therefore
    \[
    \norm*{ \int_{t_0}^{t_1}\pset[\big]{F(t)\td G(t)Z(t)} }_2 \leq \norm{F}_{L^\infty}\norm{Z}_{L^\infty}\totVar[t_0][t_1](G),
    \]
    by passing to the limit.
\end{proof}

\begin{lemma}\label{lem:RiemannStieltjes_matrixIntegrationByParts}
    Let $F\in\mathcal C\cap\BV([t_0,t_1],\C^{m,n})$, $G\in\BV([t_0,t_1],\C^{n,p})$ and $Z\in\mathcal C\cap\BV([t_0,t_1],\C^{p,q})$. Then $(AG,I)\in\RS{B}$, $(A,B)\in\RS{G}$ and $(I,GB)\in\RS{A}$.
	Furthermore, the integration by parts formula
	\begin{equation}
		F(t_1)G(t_1)Z(t_1) - F(t_0)G(t_0)Z(t_0) = \int_{t_0}^{t_1}\pset[\big]{F G(\td Z)} + \int_{t_0}^{t_1}\pset[\big]{F(\td G)Z} + \int_{t_0}^{t_1}\pset[\big]{(\td F)GZ}
	\end{equation}
	holds.
\end{lemma}

\begin{proof}
	By applying \Cref{lem:RS_threeTerms} entrywise, we obtain
    \begin{align*}
        & F(t_1)G(t_1)Z(t_1) - F(t_0)G(t_0)Z(t_0) = \bset*{ \sum_{j,k} \pset[\big]{ f_{ij}(t_1)g_{jk}(t_1)z_{kl}(t_1) - f_{ij}(t_0)g_{jk}(t_0)z_{kl}(t_0)} }_{i,l} \\
        &= \bset*{ \sum_{j,k}\pset*{\int_{t_0}^{t_1} f_{ij}g_{jk}\td z_{kl} + \int_{t_0}^{t_1} f_{ij}z_{jk}\td g_{kl} + \int_{t_0}^{t_1} g_{ij}z_{jk}\td f_{kl}}}_{i,l} \\
        &= \bset*{ \sum_{j,k}\int_{t_0}^{t_1} f_{ij}g_{jk}\td z_{kl} }_{i,l} + \bset*{\sum_{j,k}\int_{t_0}^{t_1} f_{ij}z_{jk}\td g_{kl}}_{i,l} + \bset*{\sum_{j,k}\int_{t_0}^{t_1} g_{ij}z_{jk}\td f_{kl} }_{i,l} \\
        &= \int_{t_0}^{t_1}\pset[\big]{ F(t)G(t)\td Z(t) } + \int_{t_0}^{t_1}\pset[\big]{ F(t)\td G(t)Z(t) } + \int_{t_0}^{t_1}\pset[\big]{ \td F(t)G(t)Z(t) }. \qedhere
    \end{align*}
\end{proof}

\begin{lemma}
	If $F\in\mathcal C([t_0,t_1],\C^{m,n})$, $G\in W^{1,1}([t_0,t_1],\C^{n,p})$ and $Z\in\mathcal C([t_0,t_1],\C^{p,q})$, then
	\[
	\int_{t_0}^{t_1} \pset[\big]{F(t)\td G(t)Z(t)} = \int_{t_0}^{t_1} F(t)\dot G(t)Z(t)\td t.
	\]
\end{lemma}

\begin{proof}
	By applying \Cref{thm:RS_C_and_AC} entrywise we obtain
	\[
		\int_{t_0}^{t_1}\pset[\big]{F(\td G)Z}
		= \bset*{ \sum_{j,k}\int_{t_0}^{t_1}f_{ij}z_{kl}\td g_{jk} }_{i,l}
		= \bset*{ \sum_{j,k}\int_{t_0}^{t_1}f_{ij}\dot g_{jk}z_{kl} \td t}_{i,l}
		= \int_{t_0}^{t_1}F\dot GZ\,\td t. \qedhere
	\]
\end{proof}

\noindent We next show that the matrix Riemann-Stieltjes integral satisfies something similar to the fundamental theorem of calculus.
\begin{lemma}\label{lem:RiemannStieltjes_oneTermApprox}
    Let $F\in\mathcal C([t_0,t_1],\C^{m,n})$, $G\in\BV([t_0,t_1],\C^{n,p})$, and $Z\in\mathcal C([t_0,t_1],\C^{p,q})$.
    Then for every $\varepsilon>0$ there is $\delta>0$ such that, for every $h\in[0,\delta]$, $t\in[t_0,t_1-\delta]$ and $s\in[t,t+h]$ it holds that
    \[
    \norm*{ \int_{t}^{t+h}\pset[\big]{F(\td G)Z} - F(s)\pset[\big]{ G(t+h) - G(t) }Z(s) }_2 \leq \varepsilon \totVar[t][t+h](G).
    \]
\end{lemma}

\begin{proof}
    Let $\varepsilon_1>0$ be an arbitrarily small positive constant that we are going to select later.
    Since $F$ and $Z$ are continuous matrix functions on a compact interval, there is $\delta>0$ such that $\norm{F(s_1)-F(s_0)}<\varepsilon_1$ and $\norm{Z(s_1)-Z(s_0)}<\varepsilon_1$ for every $s_0,s_1\in[t_0,t_1]$ such that $\abs{s_1-s_0}<\delta$.
    For every $h\in[0,\delta]$ and $t\in[t_0,t_1-h]$ let then $\partit=\set{\wt t_0,\ldots,\wt t_K}\in\partInt{[t,t+h]}$ be a partition such that
    \[
    S_1 \coloneqq \norm*{ \int_{t}^{t+h}\pset[\big]{F(\td G)Z} - \sum_{k=1}^K F(\wt t_k)\pset[\big]{G(\wt t_k)-G(\wt t_{k-1})}Z(\wt t_k) } \leq \varepsilon_1 \totVar[t][t+h](G).
    \]
    For every $s\in[t,t+h]$ it holds that
    \begin{align*}
        S_2 &\coloneqq \norm*{ \sum_{k=1}^K F(\wt t_k)\pset[\big]{G(\wt t_k)-G(\wt t_{k-1})}Z(\wt t_k) - F(s)\pset[\big]{G(t+h)-G(t)}Z(s)} \\
        &= \norm*{ \sum_{k=1}^K\pset[\Big]{F(\wt t_k)\pset[\big]{G(\wt t_k)-G(\wt t_{k-1})}Z(\wt t_k) - F(s)\pset[\big]{G(\wt t_k)-G(\wt t_{k-1})}Z(s)} } \\
        &\leq \sum_{k=1}^K \pset*{ \norm*{ F(\wt t_k) - F(s) }\norm[\big]{G(\wt t_k)-G(\wt t_{k-1})}\norm{Z(\wt t_k)} + \norm{F(s)}\norm[\big]{G(\wt t_k)-G(\wt t_{k-1})}\norm*{ Z(\wt t_k) - Z(s) } } \\
        &\leq \varepsilon_1\sum_{k=1}^K (\norm{F(s)}+\norm{Z(\wt t_k)}) \norm[\big]{G(\wt t_k)-G(\wt t_{k-1})} \leq \varepsilon_1(\norm{F}_{L^\infty}+\norm{Z}_{L^\infty}) \totVar[t][t+h](G),
    \end{align*}
    and therefore
    \begin{align*}
        & \norm*{ \int_{t}^{t+h}\pset[\big]{F(\td G)Z} - F(s)\pset[\big]{ G(t+h) - G(t) }Z(s) } \\
        &\qquad\leq S_1 + S_2 \leq \varepsilon_1\pset{ 1 + \norm{F}_{L^\infty}+\norm{Z}_{L^\infty}} \totVar[t][t+h](G) \leq \varepsilon \totVar[t][t+h](G),
    \end{align*}
    by choosing $\varepsilon_1$ appropriately small.
\end{proof}

\noindent We then have the following result.
\begin{theorem}\label{thm:RS_fundCalc}
    Let $F\in\mathcal C([t_0,t_1],\C^{m,n})$, $G\in \BV_\loc([t_0,t_1],\C^{n,p})$, and $Z\in\mathcal C([t_0,t_1],\C^{p,q})$, and define
    \[
    H : [t_0,t_1] \to \C^{m,q}, \qquad t \mapsto \int_{t_\textnormal{ref}}^t \pset[\big]{F(t)\td G(t)Z(t)},
    \]
    for a fixed $t_\textnormal{ref}\in[t_0,t_1]$.
    Then $H\in\BV([t_0,t_1],\C^{m,q})$ and $\td H=F(\td G)Z$, in the sense that
    \begin{equation}\label{eq:RS_fundCalc}
        \int_{t_0}^{t_1} \pset[\big]{L(t)\td H(t)M(t)} = \int_{t_0}^{t_1} \pset[\big]{L(t)F(t)\td G(t)Z(t)M(t)}
    \end{equation}
    holds for all $L\in\mathcal C([t_0,t_1],\C^{l,m})$ and $M\in\mathcal C([t_0,t_1],\C^{q,r})$.
\end{theorem}
\begin{proof}
    For every partition $\partit=\set{\wt t_0,\ldots,\wt t_K}$ of $[t_0,t_1]$ we have
    \begin{align*}
        &\sum_{k=1}^K \norm{ H(\wt t_k) - H(\wt t_{k-1}) }
        = \sum_{k=1}^K \norm*{ \int_{\wt t_k}^{\wt t_{k-1}} \pset[\big]{F(t)\td G(t)Z(t)} } \\
        &\qquad\leq \sum_{k=1}^{K}\norm{F}_{L^\infty}\norm{Z}_{L^\infty}\totVar[\wt t_{k-1}][\wt t_k](G)
        \leq \norm{F}_{L^\infty}\norm{Z}_{L^\infty}\totVar[t_0][t_1](G) < \infty,
    \end{align*}
    thus $H\in\BV([t_0,t_1],\C^{m,q})$ and all Riemann-Stieltjes integrals are well-defined.

    For every $\varepsilon>0$, let $\delta>0$ be as in the statement of \Cref{lem:RiemannStieltjes_oneTermApprox}.
    For every partition $\partit=\set{\wt t_0,\ldots,\wt t_K}$ of $[t_0,t_1]$ with $\norm{\partit}<\delta$ we have then
    \begin{align*}
        & \norm*{\sum_{k=1}^K L(\wt t_k)\pset[\big]{H(\wt t_k)-H(\wt t_{k-1})}M(\wt t_k) - \sum_{k=1}^K L(\wt t_k)F(\wt t_k)\pset[\big]{G(\wt t_k)-G(\wt t_{k-1})}Z(\wt t_k)M(\wt t_k)} \\
        &\qquad\leq \norm{L}_{L^\infty}\norm{M}_{L^\infty} \sum_{k=1}^K\norm*{ \int_{\wt t_{k-1}}^{\wt t_{k}}\pset[\big]{F(t)\td G(t)Z(t)} - F(\wt t_k)\pset[\big]{G(\wt t_k)-G(\wt t_{k-1})}Z(\wt t_k) } \\
        &\qquad\leq \varepsilon\norm{L}_{L^\infty}\norm{M}_{L^\infty} \sum_{k=1}^K \totVar[\wt t_{k-1}][\wt t_k](G)
        \leq \varepsilon\norm{L}_{L^\infty}\norm{M}_{L^\infty} \totVar[t_0][t_1](G).
    \end{align*}
    Therefore, by taking $\varepsilon$ and $\norm{\partit}$ arbitrarily small, we conclude that \eqref{eq:RS_fundCalc} holds.
\end{proof}

\begin{remark}
    While up to now we have always worked on a compact interval $[t_0,t_1]\subseteq\R$, the definitions and properties of the Riemann-Stieltjes integral extend straightforwardly to open time intervals $\timeInt\subseteq\R$ and local function spaces like $\BV_\loc(\timeInt,\C^{m,n})$ and $W^{1,1}_\loc(\timeInt,\C^{m,n})$.
    However, the Riemann-Stieltjes integrals are still only taken over compact intervals $[t_0,t_1]\subseteq\timeInt$.
\end{remark}

\noindent Going back to \Cref{cor:OIE_linearMatrix}, we now show that, when $M\in\BV_\loc(\timeInt,\HerMat[n])$, the solution of \eqref{eq:OIE_linearMatrix} can be expressed in a more compact way by using the matrix Riemann-Stieltjes integral.

\begin{theorem}\label{thm:OIE_linearMatrix_RS}
    Let $A\in L^1_\loc(\timeInt,\C^{n,n})$, $L\in L^1_\loc(\timeInt,\HerMat)$, $M\in\BV_\loc(\timeInt,\HerMat[n])$, $t_0\in\timeInt$, and $Q_0\in\posSD[n]$. Then the integral equation
    \begin{equation}\label{eq:OIE_linearMatrix_RS}
        Q(t) = Q(t_0) + \int_{t_0}^{t}\pset[\big]{ \ct{A(s)}Q(s) + Q(s)A(s) + L(s) }\td s  + M(t) - M(t_0), \qquad Q(t_0)=Q_0
    \end{equation}
    has a unique global solution $Q\in\BV_\loc(\timeInt,\HerMat[n])$ of the form
    \begin{equation}\label{eq:OIE_linearMatrix_RS_solution}
        Q(t) = \stm_{\ct A}(t,t_0)Q_0\ct{\stm_{\ct A}(t,t_0)}
        + \int_{t_0}^{t}\stm_{\ct A}(t,s)\pset[\big]{L(s)\td s + \td M(s)}\ct{\stm_{\ct A}(t,s)}\td s,
    \end{equation}
    where $\stm_{\ct A}\in W^{1,1}_\loc(\timeInt,\GL[n])$ denotes the state-transition matrix associated to $\dot x=\ct{A}x$.
\end{theorem}

\begin{proof}
    Since $\BV_\loc\subseteq L^\infty_\loc$,  because of \Cref{cor:OIE_linearMatrix} we know that \eqref{eq:OIE_linearMatrix_RS} has a unique solution $Q\in L^\infty_\loc(\timeInt,\C^{n,n})$ of the form \eqref{eq:OIE_linearMatrix_solution}. Furthermore, since $Q=\wt Q+M$ with $\wt Q\in W^{1,1}_\loc\subseteq\BV_\loc$ and $M\in\BV_\loc$, we deduce that $Q\in\BV_\loc$.
    By applying the properties of the Riemann-Stieltjes integral, we then obtain that
    \begin{align*}
        & \int_{t_0}^{t} \pset*{\stm_{\ct{A}}(t,s)\td M(s)\ct{\stm_{\ct{A}}(t,s)}} = \stm_{\ct{A}}(t,t)M(t)\ct{\stm_{\ct{A}}(t,t)} - \stm_{\ct{A}}(t,t_0)M(t_0)\ct{\stm_{\ct{A}}(t,t_0)} \\
        &\qquad- \int_{t_0}^{t}\pset*{\pd{\stm_{\ct{A}}}{s}(t,s)M(s)\stm_{\ct{A}}(t,s) + \stm_{\ct{A}}(t,s)M(s)\pd{\stm_{\ct{A}}}{s}(t,s)}\td s \\
        &\qquad= M(t) - \stm_{\ct{A}}(t,t_0)M(t_0)\ct{\stm_{\ct{A}}(t,t_0)} + \int_{t_0}^t\stm_{\ct{A}}(t,s)\pset[\big]{\ct{A(s)}M(s) + M(s)A(s)}\ct{\stm_{\ct{A}}(t,s)}\td s,
    \end{align*}
    from which it follows that \eqref{eq:OIE_linearMatrix_solution} and \eqref{eq:OIE_linearMatrix_RS_solution} are in fact the same expression.
\end{proof}

\begin{theorem}\label{thm:RiemannStieltjesIntegralInequality}
    Let $A_1\in\BV_\loc(\timeInt,\C^{p,p})$, $A_0\in L^1_\loc(\timeInt,\C^{n,n})$, $Z\in L^2_\loc(\timeInt,\C^{n,m})$, $M\in L^\infty_\loc(\timeInt,\C^{m,m})$, and $V\in\mathcal C\cap\BV_\loc(\timeInt,\C^{p,n})$ be such that $A_1(t)$, $A_0(t)$ and $M(t)$ are self-adjoint for a.e.~$t\in\timeInt$.
    Then the following statements are equivalent:
    \begin{enumerate}[label=\rm(\roman*)]
        \item\label{it:RSII:1} The Riemann-Stieltjes integral matrix inequality
        \begin{equation}\label{eq:RiemannStieltjesIntegralInequality:1}
            \int_{t_0}^{t_1} \bmat{\ct{V(t)}\td A_1(t)V(t) + A_0(t)\td t & Z(t)\td t \\ \ct{Z(t)}\td t & M(t)\td t} \geq 0
        \end{equation}
        holds for all $t_0,t_1\in\timeInt,\ t_0\leq t_1$.
        \item\label{it:RSII:2} The Riemann-Stieltjes integral scalar inequality
        \begin{equation}\label{eq:RiemannStieltjesIntegralInequality:2}
            \int_{t_0}^{t_1} \ct{\bmat{x(t) \\ u(t)}}\bmat{\ct{V(t)}\td A_1(t)V(t) + A_0(t)\td t & Z(t)\td t \\ \ct{Z(t)}\td t & M(t)\td t}\bmat{x(t) \\ u(t)} \geq 0
        \end{equation}
        holds for all $x\in\mathcal C\cap\BV_\loc(\timeInt,\C^n)$, $u\in L^2_\loc(\timeInt,\C^m)$, and $t_0,t_1\in\timeInt,\ t_0\leq t_1$.
        \item\label{it:RSII:3} The Riemann-Stieltjes integral matrix inequality
        \begin{equation}\label{eq:RiemannStieltjesIntegralInequality:3}
            \int_{t_0}^{t_1} \ct{\bmat{X(t) \\ U(t)}}\bmat{\ct{V(t)}\td A_1(t)V(t) + A_0(t)\td t & Z(t)\td t \\ \ct{Z(t)}\td t & M(t)\td t}\bmat{X(t) \\ U(t)} \geq 0
        \end{equation}
        holds for all $X\in\mathcal C\cap\BV_\loc(\timeInt,\C^{n,q})$, $U\in L^2_\loc(\timeInt,\C^{m,q})$, $t_0,t_1\in\timeInt,\ t_0\leq t_1$, and $q\in\N$.
    \end{enumerate}
\end{theorem}

\begin{proof}
    \begin{description}
        \item[\ref{it:RSII:1}$\implies$\ref{it:RSII:2}] Suppose first that \eqref{eq:RiemannStieltjesIntegralInequality:1} holds.
        Let $x\in\mathcal C\cap\BV_\loc(\timeInt,\C^n)$, $u\in L^2_\loc(\timeInt,\C^m)$, and $t_0,t_1\in\timeInt,\ t_0\leq t_1$.
        For every $N\in\N$, let us split $[t_0,t_1]$ into $N$ equal subintervals of length $h_N$, i.e., $\wt t_{N,k}\coloneqq t_0+\frac{k}{N}(t_1-t_0)$ for $k=1,\ldots,N$.
        We then define $\wt x_{N,k}\coloneqq x(\wt t_{N,k-1})$ and
        \[
        \wt u_{N,k} \coloneqq \frac{1}{h_N}\int_{\wt t_{N,k-1}}^{\wt t_{N,k}}u(s)\td s
        \]
        for all $N\in\N$ and $k=1,\ldots,N$.
        For simplicity, let us fix $N\in\N$ and remove it from the notation, e.g.~$\wt t_k\coloneqq\wt t_{N,k}$.
        Then for $k=1,\ldots,N$ it holds that
        \[
        \int_{\wt t_{k-1}}^{\wt t_k}\bmat{\ct{V}\td A_1V+A_0\td t & Z\td t \\ \ct{Z}\td t & M\td t} \geq 0.
        \]
        Let us now define $\timeInt_{t,N}\coloneqq[\wt t_{k-1},\wt t_k]$ for all $t\in[t_0,t_1]$ and $N\in\N$, where $k=k(t)$ is the minimum index such that $t\in[\wt t_{k-1},\wt t_k]$.
        We define then
        \[
        \wt x_N(t) \coloneqq x(\min(\timeInt_{t,N})), \qquad
        \wt u_N(t) \coloneqq \frac{1}{h_N}\int_{\timeInt_{t,N}}u(s)\td s,
        \]
        for all $t\in[t_0,t_1]$ and $N\in\N$.
        Since $x$ is continuous and $[t_0,t_1]$ is compact, clearly $\wt x_N\to x|_{[t_0,t_1]}$ uniformly.
        Furthermore, because of \Cref{thm:integralAverage}, $\wt u_N\to u$ in $L^2([t_0,t_1],\C^m)$.
        Let us then consider the combination
        \[
        \sum_{k=1}^N \ct{\bmat{\wt x_k \\ \wt u_k}}\int_{\wt t_{k-1}}^{\wt t_k}\bmat{\ct{V}\td A_1V+A_0\td t & Z\td t \\ \ct{Z}\td t & M\td t} \bmat{\wt x_k \\ \wt u_k} \geq 0.
        \]
        This combination naturally splits into $S_1+S_0+2\realPart(S_Z)+S_M$, where
        \begin{alignat*}{2}
            S_1 &\coloneqq \sum_{k=1}^N\ct{\wt x_k}\pset*{\int_{\wt t_{k-1}}^{\wt t_k}\ct{V(t)}\td A_1(t)V(t)}\wt x_k, &\qquad
            S_0 &\coloneqq \sum_{k=1}^N\int_{\wt t_{k-1}}^{\wt t_k}\ct{\wt x_k}A_0(t)\wt x_k\td t, \\
            S_Z &\coloneqq \sum_{k=1}^N\int_{\wt t_{k-1}}^{\wt t_k}\ct{\wt x_k}Z(t)\wt u_k\td t, &\qquad
            S_M &\coloneqq \sum_{k=1}^N\int_{\wt t_{k-1}}^{\wt t_k}\ct{\wt u_k}M(t)\wt u_k\td t.
        \end{alignat*}
        Because of \Cref{thm:RS_fundCalc}, it holds that
        \[
        \lim_{N\to\infty} S_1 = \int_{t_0}^{t_1}\ct{x(t)}\td H(t)x(t) = \int_{t_0}^{t_1}\ct{x(t)}\ct{V(t)}\td A_1(t)V(t)x(t)\td t,
        \]
        with $H:t\mapsto\int_{t_0}^{t}\ct{V(s)}\td A_1(s)V(s)\td s$.
        Furthermore, since we can write equivalently
        \begin{align*}
        S_0 &= \int_{t_0}^{t_1}\ct{\wt x_N(t)}A_0(t)\wt x_N(t)\td t, \\
        S_Z &= \int_{t_0}^{t_1}\ct{\wt x_N(t)}Z(t)\wt u_N(t)\td t, \\
        S_M &= \int_{t_0}^{t_1}\ct{\wt u_N(t)}M(t)\wt u_N(t)\td t,
        \end{align*}
        with the generalized H\"older inequality (\Cref{thm:genHolder}) we deduce that
        \begin{align*}
            \lim_{N\to\infty} S_0 &= \int_{t_0}^{t_1}\ct{x}(t)A_0(t)x(t)\td t, \\
            \lim_{N\to\infty} S_Z &= \int_{t_0}^{t_1}\ct{x}(t)Z(t)u(t)\td t, \\
            \lim_{N\to\infty} S_M &= \int_{t_0}^{t_1}\ct{u}(t)M(t)u(t)\td t.
        \end{align*}
        Thus, in the limit, we obtain exactly \eqref{eq:RiemannStieltjesIntegralInequality:2}.
        \item[\ref{it:RSII:2}$\implies$\ref{it:RSII:3}] Suppose now that \eqref{eq:RiemannStieltjesIntegralInequality:2} holds.
        Let $X\in\mathcal C\cap BV_\loc(\timeInt,\C^{n,q})$, $U\in L^2_\loc(\timeInt,\C^{m,q})$, $t_0,t_1\in\timeInt,\ t_0\leq t_1$, and $q\in\N$.
        For every $v\in\C^{n+m}$ it holds that
        \begin{align*}
            & \ct{v} \pset*{\int_{t_0}^{t_1} \ct{\bmat{X(t) \\ U(t)}}\bmat{\ct{V(t)}\td A_1(t)V(t) + A_0(t)\td t & Z(t)\td t \\ \ct{Z(t)}\td t & M(t)\td t}\bmat{X(t) \\ U(t)}}v = \\
            &\qquad= \int_{t_0}^{t_1} \ct{\bmat{X(t)v \\ U(t)v}}\bmat{\ct{V(t)}\td A_1(t)V(t) + A_0(t)\td t & Z(t)\td t \\ \ct{Z(t)}\td t & M(t)\td t}\bmat{X(t)v \\ U(t)v} \geq 0,
        \end{align*}
        since $x\coloneqq Xv\in\mathcal C\cap BV_\loc(\timeInt,\C^n)$ and $u\coloneqq Uv\in L^2_\loc(\timeInt,\C^m)$. We deduce that \eqref{eq:RiemannStieltjesIntegralInequality:3} is satisfied.
        \item[\ref{it:RSII:3}$\implies$\ref{it:RSII:1}] This follows immediately by choosing for $X$ and $U$ constant matrices such that
        \[
        \bmat{X(t) \\ U(t)} = \bmat{I_n & 0 \\ 0 & I_m},
        \]
        for all $t\in\timeInt$. \qedhere
    \end{description}
    
\end{proof}

\begin{corollary}\label{lem:RiemannStieltjes_increasing}
    Let $G\in\BV_\loc(\timeInt,\HerMat[n])$.
    Then $G$ is weakly increasing if and only if $\td G\geq 0$, in the sense that
    \[
    \int_{t_0}^{t_1}\ct{X(t)}\td G(t)X(t) \geq 0
    \]
    holds for all $X\in\mathcal C\cap\BV_\loc(\timeInt,\C^{n,q})$. Similarly, $G$ weakly decreases if and only if $\td G\leq 0$.
\end{corollary}

\begin{proof}
    By applying \Cref{thm:RiemannStieltjesIntegralInequality} with $m=0$ and $V=I_n$, we obtain that the inequalities
    \[
    G(t_1) - G(t_0) = \int_{t_0}^{t_1}\td G(t) \geq 0, \qquad\text{for all }t_0,t_1\in\timeInt,\ t_0\leq t_1
    \]
    and
    \[
    \int_{t_0}^{t_1}\ct{X(t)}\td G(t)X(t) \geq 0, \qquad\text{for all $X\in\mathcal C\cap\BV_\loc(\timeInt,\C^{n,q})$ and }t_0,t_1\in\timeInt,\ t_0\leq t_1
    \]
    are equivalent. In particular, $G$ increases monotonically if and only if $\td G\geq 0$.
    To show the second part of the statement, it suffices to apply the first part to $-G$ and $\td(-G)=-\td G$.
\end{proof}

\end{document}